\documentclass[a4paper,11pt]{amsart}
\usepackage{amsmath,amsthm,amssymb}
\usepackage{comment}
\newcommand{\mychoice}[3]{#1
}
\mychoice{
\newcommand{\plabel}[1]{ \label{#1}}
\newcommand{\gbibitem}[1]{ \bibitem{#1}}
\newcommand{\snewpage}{}

\specialcomment{commentx}{}{}\excludecomment{commentx}
\specialcomment{commenty}{}{}\excludecomment{commenty}
\specialcomment{commentz}{}{}

}{
\usepackage{xcolor}
\newcommand{\plabel}[1]{ \label{#1}\rlap{\smash{${}^{^{[#1]}}$}}}
\newcommand{\gbibitem}[1]{ \bibitem{#1}\rlap{\smash{${}^{^{[#1]}}$}}}
\newcommand{\snewpage}{\newpage}

\newenvironment{commentx}{\color{magenta} }{\color{black} }
\newenvironment{commenty}{\color{blue} }{\color{black} }

}{
\usepackage{xcolor}
\newcommand{\plabel}[1]{ \label{#1}}
\newcommand{\gbibitem}[1]{ \bibitem{#1}}
\newcommand{\snewpage}{}

\specialcomment{commentx}{}{}\excludecomment{commentx}
\newenvironment{commenty}{  }{}
\specialcomment{commentz}{}{}\excludecomment{commentz}

}
\DeclareMathOperator{\sgn}{sgn}

\DeclareMathOperator{\Id}{Id}
\DeclareMathOperator{\tr}{tr}

\DeclareMathOperator{\artanh}{artanh}
\DeclareMathOperator{\arcosh}{arcosh}
\newcommand{\m}{\mathbf}
\DeclareMathOperator{\arsinh}{arsinh}
\DeclareMathOperator{\Rea}{Re}\DeclareMathOperator{\Ima}{Im}
\DeclareMathOperator{\CR}{CR}
\DeclareMathOperator{\DW}{DW}
\newcommand{\alter}{\texttt{Alternatively,} }
\theoremstyle{definition}
\newtheorem{point}{}[section]
\newtheorem{defin}[point]{Definition}

\newtheorem{disc}[point]{Discussion}
\newtheorem{remark}[point]{Remark}
\newtheorem{example}[point]{Example}
\theoremstyle{plain}

\newtheorem{lemma}[point]{Lemma}
\newtheorem{cor}[point]{Corollary}
\newtheorem{theorem}[point]{Theorem}
\newcommand{\bem}{\begin{bmatrix}}
\newcommand{\eem}{\end{bmatrix}}
\newcommand{\bo}{\boldsymbol}

\newcommand{\vvarpropto}{\,\tilde=\,}
\newenvironment{bsmallmatrix}{\left[\begin{smallmatrix}}{\end{smallmatrix}\right]}
\renewcommand{\thesubsection}{{\arabic{section}.\Alph{subsection}}}

\DeclareMathOperator{\adj}{adj}
\DeclareMathOperator{\radius}{radius}
\DeclareMathOperator{\axis}{axis}
\DeclareMathOperator{\re}{Re}
\DeclareMathOperator{\rank}{rk}
\DeclareMathOperator{\ima}{Im}
\DeclareMathOperator{\SO}{SO}
\newcommand{\eqed}{
\pushQED{\qed}
\qedhere
\popQED
}
\newcommand{\eqedexer}{
\renewcommand{\qedsymbol}{$\diamondsuit$}
\pushQED{\qed}
\qedhere
\popQED
\renewcommand{\qedsymbol}{$\Box$}
}
\newcommand{\eqedremark}{
\renewcommand{\qedsymbol}{$\triangle$}
\pushQED{\qed}
\qedhere
\popQED
\renewcommand{\qedsymbol}{$\Box$}
}
\newcommand{\qedexer}{  \renewcommand{\qedsymbol}{$\diamondsuit$} \qed \renewcommand{\qedsymbol}{$\Box$}}
\newcommand{\qedremark}{  \renewcommand{\qedsymbol}{$\triangle$} \qed \renewcommand{\qedsymbol}{$\Box$}}

\newcommand{\proofremark}[1]{
\begin{proof}[Remark] #1
\renewcommand{\qedsymbol}{}
\end{proof}
}

\newcommand{\marginextend}[1]{ \addtolength{\oddsidemargin}{-#1}  \addtolength{\evensidemargin}{-#1}
  \addtolength{\textwidth}{#1}\addtolength{\textwidth}{#1}}
\newcommand{\updownextend}[1]{ \addtolength{\topmargin}{-#1}  \addtolength{\textheight}{#1}
\addtolength{\textheight}{#1}}
\marginextend{1cm}
\updownextend{0cm}
\allowdisplaybreaks[4]
\title[
 On the elliptical range theorems
]{
On the elliptical range theorems for the Davis--Wielandt shell, the numerical range, and the conformal range
}
\author{Gyula Lakos}
\email{gyula.lakos@uni-miskolc.hu}
\address{Institute of Mathematics, Department of Analysis, University of Miskolc, H-3515 Miskolc-Egyetemváros, Hungary}
\begin{document}
\begin{abstract}
We consider the elliptical range theorems for the Davis--Wielandt shell, the numerical range, and the conformal range
 in terms of and related to their quadratic representations.
The emphasis is on exposing a variety of elementary approaches.
\end{abstract}
\maketitle
\snewpage
\section*{Introduction}
This paper is an extension to \cite{LLL} dealing with the elliptical range theorems (for $2\times2$ matrices)
 but in terms of their quadratic equations.
Firstly, we review a couple of proofs regarding the shape of these quadratic equations.
These may or may not require knowledge of more advanced general properties of the ranges.
Then, we discuss aspects of the geometry of the various ranges related to the quadratic forms in question.
In this latter matter, familiarity with \cite{LLL} is presupposed.
In general, an understanding of the basics as explained in \cite{LLL} is assumed
 (thus hyperbolic geometry will not be explained again), but formulas which are not trivial to remember will be recalled.

\textbf{The Davis--Wielandt shell, the numerical range, and the conformal range recalled.}
Assume that $\mathfrak H$ is a Hilbert space, with interior product $\langle\cdot,\cdot\rangle$
(linear in the first variable, skew-linear in the second variable); $|\m x|_2=\sqrt{\langle \m x,\m x\rangle}$.

If $A$ is a linear operator on $\mathfrak H$, then its Davis--Wielandt shell (due to
Wielandt \cite{Wie0}, \cite{Wie}; Davis \cite{D1}, \cite{D2}) is defined as
\begin{multline*}
\DW_{\mathrm{pCK}}(A)=
\left\{
\left(\frac{\Rea\langle \m y,\m x\rangle}{|\m x|_2^2} ,\frac{\Ima\langle \m y,\m x\rangle}{|\m x|_2^2} , \frac{|\m y|_2^2}{|\m x|_2^2}\right)
\,:\, A\m x=\m y, \, \m x\neq 0
\right\}\\
\subset
 \overline{H}^3_{\mathrm{pCK}}=\{(x,y,z)\in\mathbb R^3\,:\,z\geq x^2+y^2\}\cup\{\infty_{\langle0,0,1\rangle}\},
\end{multline*}
or
\begin{multline*}
\DW_{\mathrm{BCK}}(A)=
\left\{
\left(\frac{2\Rea\langle\m  y,\m x\rangle}{|\m y|_2^2+|\m x|_2^2} ,\frac{2\Ima\langle \m y,\m x\rangle}{|\m y|_2^2+|\m x|_2^2} , \frac{|\m y|_2^2-|\m x|_2^2}{|\m y|_2^2+|\m x|_2^2}\right)
\,:\, A\m x=\m y, \, \m x\neq 0
\right\}\\
\subset\overline{H}^3_{\mathrm{BCK}}=\{(x,y,z)\in\mathbb R^3\,:\,x^2+y^2+z^2\leq1\}
 \end{multline*}
(where $\infty_{\langle0,0,1\rangle}$ is the ideal point in the $\langle0,0,1\rangle$ direction).
These sets are interpreted as subsets of the asymptotically closed
Beltrami--Cayley--Klein or parabolic Cayley--Klein models of the hyperbolic space.
Between the BCK and pCK models, there are (for us) canonical correspondences given by
\begin{equation}
\bem x_{\mathrm{BCK}} \\ y_{\mathrm{BCK}}  \\ z_{\mathrm{BCK}}  \\ 1 \eem
=
\frac1{z_{\mathrm{pCK}}+1}\!
\bem 2x_{\mathrm{pCK}} \\ 2y_{\mathrm{pCK}}  \\ z_{\mathrm{pCK}}-1  \\  z_{\mathrm{pCK}}+1 \eem
%\plabel{eq:can1}
%\end{equation}%\quad
\,\,\,\text{and}\,\,
%\begin{equation}
\bem x_{\mathrm{pCK}} \\ y_{\mathrm{pCK}}  \\ z_{\mathrm{pCK}}  \\ 1 \eem
=
\frac1{1-z_{\mathrm{BCK}}}\!
\bem x_{\mathrm{BCK}} \\ y_{\mathrm{BCK}}  \\ 1+z_{\mathrm{BCK}}  \\  1-z_{\mathrm{BCK}} \eem;
\plabel{eq:can2}
\end{equation}
with $(0,0,1)$ in BCK corresponding to $\infty_{\langle0,0,1\rangle}$ in pCK
(although these latter points do not appear in the shell for linear operators).
These transcriptions, one can notice, are realized by projective transformations.
Other models of the (asymptotically closed) hyperbolic space can also be used for the shell,
however, for the purposes of linear algebra, the projective models BCK and pCK are particularly suitable.

The numerical range of $A$ (due to Toeplitz \cite{Toe}; not much later also studied by  Hausdorff \cite{Hau}) is defined as
\[\mathrm W(A)=
\left\{
\left(\frac{\Rea\langle \m y,\m x\rangle}{|\m x|_2^2} ,\frac{\Ima\langle \m y,\m x\rangle}{|\m x|_2^2} \right)
\,:\, A\m x=\m y, \, \m x\neq 0
\right\}.\]
This is a subset of $\mathbb R^2$, usually identified with $\mathbb C$.
~
One can immediately recognize that  $\mathrm W(A)$ is the projection of $\DW_{\mathrm{pCK}}(A)$ to the first two coordinates.
\snewpage

The conformal range, or in other name, the real Davis--Wieland shell
 (first studied more extensively in \cite{L2}) of $A$ is defined as
\begin{multline*}
 \DW_{\mathrm{pCK}}^{\mathbb R}(A)=
\left\{
\left(\frac{\Rea\langle \m y,\m x\rangle}{|\m x|_2^2}  , \frac{|\m y|_2^2}{|\m x|_2^2}\right)
\,:\, A\m x=\m y, \, \m x\neq 0
\right\}
\\\subset
 \overline{H}^2_{\mathrm{pCK}}=\{(x, z)\in\mathbb R^2\,:\,z\geq x^2 \}\cup\{\infty_{\langle0,1\rangle} \},
\end{multline*}
and
\begin{multline*}\DW_{\mathrm{BCK}}^{\mathbb R}(A)=
\left\{
\left(\frac{2\Rea\langle\m  y,\m x\rangle}{|\m y|_2^2+|\m x|_2^2}  ,
\frac{|\m y|_2^2-|\m x|_2^2}{|\m y|_2^2+|\m x|_2^2}\right)
\,:\, A\m x=\m y, \, \m x\neq 0
\right\}
\\\subset\overline{H}^2_{\mathrm{BCK}}=\{(x ,z)\in\mathbb R^2\,:\,x^2 +z^2\leq1\}
\end{multline*}
(where $\infty_{\langle0,1\rangle}$ is the ideal point in the $\langle0,1\rangle$ direction).
These sets are interpreted as subsets of the asymptotically closed
Beltrami--Cayley--Klein  or parabolic Cayley--Klein  models of the hyperbolic plane.
Between the BCK and pCK models, there are (for us) canonical correspondences given by
\begin{equation}
\bem x_{\mathrm{BCK}}  \\ z_{\mathrm{BCK}}  \\ 1 \eem
=
\frac1{z_{\mathrm{pCK}}+1}\!
\bem 2\,x_{\mathrm{pCK}}  \\ z_{\mathrm{pCK}}-1  \\  z_{\mathrm{pCK}}+1 \eem
%\plabel{eq:can1}
%\end{equation}%\quad
\,\,\,\text{and}\,\,
%\begin{equation}
\bem x_{\mathrm{pCK}}   \\ z_{\mathrm{pCK}}  \\ 1 \eem
=
\frac1{1-z_{\mathrm{BCK}}}\!
\bem x_{\mathrm{BCK}}  \\ 1+z_{\mathrm{BCK}}  \\  1-z_{\mathrm{BCK}} \eem.
\plabel{eq:can20}
\end{equation}
with $(0,1)$ in BCK corresponding to $\infty_{\langle0,1\rangle}$ in pCK
(although these latter points do not appear in the shell for linear operators).
These transcriptions, again, are realized by projective transformations.

The conformal range $\DW^{\mathbb R}_{\mathrm{pCK}}(A)$ or
 $\DW^{\mathbb R}_{\mathrm{BCK}}(A)$ is $\DW_{\mathrm{pCK}}(A)$ or
 $\DW_{\mathrm{BCK}}(A)$ but projected to the first and third coordinates.
Alternatively, if $A$ is a bounded linear operator, the conformal range can be derived from the numerical range as follows.
Let us define the real double as
\begin{equation}
\mathrm D^{\mathbb R}(A)=\frac{A+A^*}{2}+\mathrm i A^*A.
\plabel{eq:doub}
\end{equation}
Then it is easy to see that
\begin{equation}
\DW_{\mathrm{pCK}}^{\mathbb R}(A)=\mathrm W \left(\mathrm D^{\mathbb R}(A)\right).
\plabel{eq:doubac}
\end{equation}

In what follows, $*$ may mean either BCK or pCK.
\snewpage

\textbf{The elliptical range theorems.}
These describe the ranges of $2\times2$ complex matrices, which turn out to be, simply, of elliptical shape.
In these observations, three levels can be distinguished.
(In what follows, we assume that $A$ is $2\times2$ complex matrix.)

\textit{The qualitative elliptical range theorems:}
Regarding the Davis-Wielandt shell, it is a consequence of a more general theorem of Davis
 that $\DW_*(A)$ is a possibly degenerate ellipsoid (in Euclidean view).
This implies (cf.~Davis \cite{Dav}) that $\mathrm W(A)$ is a possibly degenerate elliptical disk; which was, of course,
 observed by Toeplitz \cite{Toe} much earlier.
It also implies that $\DW_*^{\mathbb R}(A)$ is a possibly degenerate elliptical disk; which, of course, is also implied by \eqref{eq:doubac}.

\textit{The synthetic elliptical range theorems:}
They describe the shape of elliptical ranges in terms of the synthetic geometry of the
 hyperbolic space, the Euclidean space, and the hyperbolic plane respectively.
In the case of the Davis-Wielandt shell, the qualitative theorem requires only little augmentation
 in terms of the hyperbolic radius of range; this is done in \cite{LLL}.
In the case of the numerical range, the qualitative theorem is augmented by the focal properties of boundary ellipse
 (the foci are the eigenvalues) and possibly with the lengths of its semi-axes.
On the elliptical and focal statement see
 Toeplitz \cite{Toe}, Murnaghan \cite{Mur}, Donoghue \cite{Don}, Gustafson, Rao \cite{GR}, Li \cite{Li}, and our \cite{LL}
 for various explanations.
The semi-axes are computed explicitly in Johnson \cite{Joh} (real case) and Uhlig \cite{Uhl}; since then they
 are often incorporated to the statement of the theorem, cf. Horn, Johnson \cite{HJ}.
In the case of the conformal range, synthetic interpretation of the range is provided in \cite{LLL}.
The situation is analogous to the one of the numerical range, but more complicated as it involves hyperbolic conics.
(Which is a topic older than the numerical range; see
 Story \cite{Sto} (1882), Killing \cite{Kil} (1885),
 D'Ovidio \cite{DO1}, \cite{DO2}, \cite{DO3} (1891), Barbarin \cite{Bar} (1901), Liebmann \cite{Lib0} (1902),
 Massau \cite{Mas} (1905), Liebmann \cite{Lib1} (1905), Coolidge \cite{Coo} (1909), V\"or\"os \cite{V1} / \cite{V2} (1909/10)
 for basic works of various readability and preciseness.
Regarding later works, we mention Fladt \cite{F1}, \cite{F2}, \cite{F3}, and the exposition by Izmestiev \cite{Izm}.
But only very particular hyperbolic conics occur for the conformal range.
In terms of quadratic surfaces see Barbarin \cite {Bar}, Coolidge \cite{CooP}, Bromwich \cite{Bro}, Coolidge \cite{Coo}, V\"or\"os \cite{V3};
 but only very particular ones occur for the Davis--Wieland shell.
In fact, \cite{LLL} explains all the appearing quadrics in sufficient detail.)

\textit{The elliptical range theorems computationally:}
These exhibit the (quadratic) equations for the ranges.
For the Davis--Wieland shell, this was done by Lins, Spitkovsky, Zhong \cite{LSZ}; but they do not give the computational details.
The result, however, is quite pleasant as it gives a good equation even if $A$ is normal (when the quadric is degenerate)
 as long as the solutions are restricted to base set of the model.
The situation is a bit more problematic in the case of the numerical range and the conformal range.
The problem is that when $A$ is normal, it may occur that the range is a segment
 (in Euclidean view) and a direct quadratic equation cannot capture the seqment but only its supporting line.
(But the situation is good for the dual quadric, cf. Kippenhahn \cite{Kip1} / \cite{Kip2}.)
If $A$ is non-normal, the Uhlig \cite{Uhl} gives the equation for the numerical range;
 and we will give the equation for the conformal range.
\snewpage

\textbf{On our objectives.}
In general, finding out the explicit quadratic representations of the ranges of $2\times2$ matrices
 is not very hard, but, despite the many possible many ways to do it,
 it may be somewhat frustrating and potentially time-consuming to arrange matters in a neat form.
A second point is to connect an explicitly given quadric to geometry and matrices.
Doing this in practice may also be a bit frustrating.
%Thus, again, although only elementary principles are involved, its discussion may be useful.
For this reason, I think, it is fair to discuss these topics \textit{once} despite their very elementary nature,
 in order to avoid unnecessary duplications.
We will not be content with a singe approach; intentionally, several alternative arguments will be introduced;
 so that the reader can choose a favourite one.
These approaches are  taken up instead of a more curious reader, sparing him or her from some computations.

\textbf{On the layout of this paper.}
After various reviews and preparations we will consider the quantitative properties of the
Davis--Wielandt shell,  the numerical range, and the conformal range for complex $2\times2$ matrices.
This will be done in Sections \ref{sec:DW}, \ref{sec:W}, and \ref{sec:CR}, respectively.
%Although the development allows various selections, the paper itself is intended to be read linearly.
(We review all necessary prerequisites except the ones for the discussions involving the Kippenhahnian viewpoint.
However, all possible prerequisites to this paper are contained in Sections 2, 5, 7 and Appendices A, B of \cite{L2}.)

The length of the paper is obtained partly from its expository nature,
 partly from the fact that several matrices are written out explicitly,
 and partly from the choice that several alternative arguments are also presented.
There are several formulas which are merely of computational matters, those computations are much detailed.

\textbf{On the picture obtained.}
We will not quote any particular statement here, as they involve large matrices or they are just particular computations.
However, we address the simple general picture which can be drawn from those.

In the case of the shell, the shell can be described precisely by a quadratic equation with matrix $\m Q$.
(The solution set should be restricted to the asymptotically closed hyperbolic space.)
In the dual approach, the tangent planes to the shell can equivalently be described by a quadratic equation with matrix $\m G$.
Here the two approaches are of equivalent capability; in fact, $\m Q$ and $\m G$ are closely related.
Let us, however, make some general observations here.
If we consider the shell as a \textit{covariant} object, then the quadric describing it, and the matrix $\m Q$
 representing it are \textit{contravariant} objects.
Also the set of tangent planes to the shell is a \textit{contravariant} object.
Then  the quadric describing it, and the matrix $\m G$ representing it are \textit{covariant} objects.
In this sense, $\m G$ is more closely aligned to the shell than $\m Q$.

In the case of numerical range and the conformal range  the situation is similar,
 but the corresponding matrices $\m Q$ are not necessarily
 sufficiently informative when $A$ is a normal matrix (as information loss may occur in $\m Q$).
Therefore we will start treating these ranges in the (generic) case when  $A$ is a non-normal matrix.
Then we will examine the case when $A$ is a normal matrix in order to see the extent of the information loss.
We do the same regarding $\m G$, but we will see that no information loss occurs there.
Therefore $\m G$ is completely faithful unless $\m Q$.
It is very easy to obtain $\m Q$ from $\m G$, but it is also relatively easy pass to the other direction if $A$ is non-normal.

In our simple case of $2\times2$ matrices the general structure of the matrices $\m G$ is sort of simple.
It happens that there is a ``spectral pencil'' of matrices depending only on the spectrum of $A$.
Then $\m G$ is picked out of the pencil according to the degree of non-normality of $A$.
A similar picture applies to $\m Q$ in the case of the shell.
In the case of the numerical and conformal ranges the matrix $\m Q$ is more complicated but it can
 also be decomposed relatively informatively.
\begin{commentx}
The picture sketched above is rather about the algebra of the various ranges.
Now, the matrices $\m Q$ can also be related to the geometry of ranges, augmenting some statements from \cite{LLL}.
\end{commentx}
Ultimately, we will have a couple matrices which are nor particularly difficult objects
but they are not completely trivial either.
\snewpage

\textbf{On terminology and notation.}
In general, we will adopt the terminology from \cite{L2}.
We will often consider natural transformations of the hyperbolic plane and space.
Due to their nature, depending on their appearance,
the terms `$h$-collineations', `$h$-isometries', `$h$-congruences', `conformal transformations',
and,
`M\"obius transformations', `fractional linear transformations'
(rather synonymous in the orientation-preserving case)
are used quite ecclectically.
If we do not specify that in what model we consider the Davis--Wielandt shell of $A$, then we
write $\DW_*(A)$ where $*$ might mean BCK, pCK, or Ph.
Similar comment applies to the conformal range $\DW^{\mathbb R}_*(A)$.
A possibly asymptotic hyperbolic point $\boldsymbol x$
is understood as a compatible collection of $\boldsymbol x_*$'s;
$\boldsymbol x_{\mathrm{pCK}}$ is the tuple of its coordinates in the pCK model, etc.
If a statement is meant in synthetic hyperbolic sense, then the $*$'s can be omitted altogether.
%The the Euclidean (or affine) segment between $a$ and $b$ is denoted by $[a,b]_{\mathrm e}$.
%The hyperbolic segment between $a$ and $b$ is denoted by $[a,b]_{\mathrm h}$.
The `classical adjoint' or adjugate matrix of $\m M$ will be denoted $\adj \m M$.
It will be used several times.
Some of its properties will be use multiple times, without special reference.
The connection to the inverse will be used several times.
See Appendix \ref{sec:algproj} for more.
The similarity (i.~e.~conjugacy) of matrices $\m M_1$ and $\m M_2$ will be denoted by $\m M_1\sim \m M_2$.
If $\m M_1$ and $\m M_2$ are proportional to each other by a nonzero scalar, then it is denoted by $\m M_1\vvarpropto \m M_2$.
If $a>0$, then $\frac a0=+\infty$.
However, $\frac 00=\pm\infty$, for which, in formulas,
 special evaluation rules apply, which will be indicated; typically, the values $0$, $1$, or $\infty$  apply.
Line $y$ of formula ($X$) will be referred as ($X$/$y$).

\textbf{Acknowledgements.}
The author thanks \'Akos G. Horv\'ath for help related to hyperbolic conics.
%\snewpage
\begin{commentx}
\tableofcontents
\end{commentx}
\snewpage
\section{Normality of $2\times2$ matrices and related concepts}
\plabel{sec:LogCont}

Here all matrices will be $2\times2$ complex matrices.

\subsection{Spectral type (review)}
\plabel{ssub:spectype}
~\\

We use the notation
\begin{equation}
D_A=\det\left( A-\frac{\tr A}2\Id_2\right)=(\det A)-\frac{(\tr A)^2}4.
\plabel{nt:DA}
\end{equation}
It is essentially the discriminant of $A$, as the eigenvalues of $A$ are $\frac12\tr A\pm\sqrt{-D_A}$.

\begin{commentx}
Writing the $2\times2$ matrix $A$ in skew-quaternionic form
\begin{equation}
A= \check a\Id_2 +\check b\check I+\check c\check J+\check d\check K\equiv
\check a\begin{bmatrix} 1&\\&1\end{bmatrix}
+\check b\begin{bmatrix} &-1\\1&\end{bmatrix}
+\check c\begin{bmatrix} \mathrm i&\\&-\mathrm i\end{bmatrix}
+\check d\begin{bmatrix} &\mathrm i\\\mathrm i&\end{bmatrix}.
\plabel{eq:skqform}
\end{equation}
it yields
\begin{equation}D_A=\check b^2+\check c^2+\check d^2.\end{equation}

In the special case of real $2\times2$ matrices, we use the classification

$\bullet$ elliptic case: two conjugate strictly complex eigenvalues,

$\bullet$ parabolic case: two equal real eigenvalues,

$\bullet$ hyperbolic case: two distinct real eigenvalues.

Then, for real $2\times2$ matrices,  $D_A$ measures `ellipticity/parabolicity/hiperbolicity':
If $D_A>0$, then $A$ is elliptic;
if $D_A=0$, then $A$ is parabolic;
if $D_A<0$, then $A$ is hyperbolic.
\end{commentx}
In the general complex case, there are two main categories for the spectral type: parabolic $(D_A=0)$ and non-parabolic $(D_A\neq 0)$.
~\\

\subsection{The metric discriminant (review)}
~\\

In terms of trace,
\begin{equation*}-D_A
=\frac12\tr\left(\left(A-\frac{\tr A}{2}\Id\right)^2\right)
=\frac{\tr A^2}2-\frac{(\tr A)^2}4
.\end{equation*}
Thus, a similar notion is given by the metric discriminant
\begin{equation}U_A
=\frac12\tr\left(\left(A-\frac{\tr A}{2}\Id\right)^*\left(A-\frac{\tr A}{2}\Id\right)\right)
=\frac{\tr A^*A}2-\frac{|\tr A|^2}4
.\plabel{nt:UA}\end{equation}
(We will see, however, that $U_A$ is not truly a counterpart of $-D_A$, but of $|D_A|$.)
\begin{commentx}
We remark that in form \eqref{eq:skqform},
\begin{equation}
U_A=|\check b|^2+|\check c|^2+|\check d|^2.
\plabel{eq:upos}
\end{equation}
\end{commentx}

\subsection{Canonical triangular form (review)}
\plabel{ssub:triangular}
~\\

It is elementary, but useful to keep in mind that any $2\times2$ complex matrix $A$ can be brought
by unitary conjugation into from
\begin{equation}\begin{bmatrix}
\lambda_1&\tau\\0&\lambda_2
\end{bmatrix}
\plabel{eq:canon2}\end{equation}
where $\tau\geq0$.
This form is unique, up to the order of the eigenvalues $\lambda_1$ and $\lambda_2$.

In the situation above,
\begin{equation}U_A=\frac{\tau^2}2+\left|\frac{\lambda_1-\lambda_2}2\right|^2,
\qquad\text{and}\qquad D_A=-\left(\frac{\lambda_1-\lambda_2}2\right)^2.
\plabel{eq:cansect}
\end{equation}

Conversely,
\begin{equation}
\begin{matrix}
\text{$\lambda_1,\lambda_2$\qquad are the two roots of}
\qquad
\lambda^2-(\tr A)\lambda+\det A=0;
\end{matrix}
\plabel{eq:canroots}\end{equation}
and the ``canonical off-diagonal quantity'',  as \eqref{eq:cansect} shows, is
\begin{equation}\tau=\sqrt{2(U_A-|D_A|)}.\plabel{eq:canoffdiag}\end{equation}
(One familiar with the elliptical range theorem
can recognize that $U_A$ and $|D_A|$ are quantities to play role there.
In particular, \eqref{eq:canoffdiag} is the length of the minor axis.)
\\

\subsection{The five data}
\plabel{ssub:five}~\\

For a $2\times2$ complex matrix $A$, the quantities
\begin{equation}
\Rea\tr A,\, \Ima\tr A,\, \Rea\det A,\, \Ima\det A,\, \tr (A^*A)\plabel{eq:polr}
\end{equation}
will be called as the `five data'.
Note %, in particular,
that $D_A$ and $U_A$ can be expressed by the five data.

\begin{lemma}\plabel{lem:five}
 The five data \eqref{eq:polr} characterizes the complex $2\times2$ matrix $A$ up to conjugation by unitary matrices.
\begin{proof}
It is trivial that conjugation by unitary matrices does not change the five data.
Conversely, consider the possible canonical triangular form for $A$ as in \eqref{eq:canon2}.
Then \eqref{eq:canroots} and \eqref{eq:canoffdiag} imply that the five data determines the canonical form.
\end{proof}
\end{lemma}
While it is theoretically convenient to write out formulas invariantly, in terms of five data \eqref{eq:polr};
checking those theoretical computations is much more convenient to do symbolically in the `triangularized five data'
\begin{equation}
\Rea\lambda_1,\Ima\lambda_1,\Rea\lambda_2,\Ima\lambda_2,\tau
\end{equation}
with respect to \eqref{eq:canon2}.

\subsection{Normality (review)}
\plabel{ssub:normal}

\begin{lemma}\plabel{lem:UADA}
(See \cite{LLL}.)
Suppose that $A$ is a $2\times2$ complex matrix. Then
\[U_A\geq |D_A|.\]
If $A$ is normal, then $U_A= |D_A|$.
If $A$ is not normal, then $U_A> |D_A|$.
\qed
\end{lemma}

\begin{commentx}
\begin{remark}\plabel{rem:UADA}
$U_A=-D_A$ if $A$ is self-adjoint, and
$U_A=D_A$ if $A$ is skew-adjoint.
More precisely,
$U_A=-D_A$ iff $A$ is self-adjoint plus a purely imaginary scalar matrix, and
$U_A=D_A$ iff $A$ is skew-adjoint plus a real scalar matrix,
or, phrased alternatively, iff
$A$ is a matrix of quaternionic type plus a purely imaginary scalar matrix.
\qedremark
\end{remark}
\end{commentx}

\subsection{Invariance for complex M\"obius transformations (review)}
\plabel{ssub:UDconf}
~\\

If $f:\lambda\mapsto \frac{a\lambda+b}{c\lambda+d}$, $a,b,c,d\in\mathbb C$, $ad-bc\neq0$ is a complex fractional linear
(i.~e.~complex M\"obius) transformation, and $A$ is a complex $2\times2$ matrix, and $-\frac dc$ is not in the spectrum of $A$,
then we say that $f$ is applicable to $A$. (Indeed, we can take $f(A)$.)
%\end{commentx}
\begin{lemma}\plabel{lem:UDconf}
(See \cite{LLL}.)
For complex $2\times 2$ matrices $A$, the ratio
\[U_A : |D_A|\]
is invariant under complex M\"obius transformations (when  they are applicable to $A$).
\qed
\end{lemma}

\subsection{The canonical representatives (complex M\"obius) (review)}

\begin{lemma}\plabel{lem:preDW}
(See \cite{LLL}.)
For  complex $2\times2$ matrices
up to complex M\"obius  transformations and unitary conjugations
(which, of course, commute with each other)
it is sufficient to consider the following representatives:
\begin{equation}
\m 0_2=\begin{bmatrix}0&\\&0\end{bmatrix}
\plabel{eqx:022H}
\end{equation}
(parabolic, normal) ;
\begin{equation}
S_0=\begin{bmatrix} 0&1 \\&0\end{bmatrix}
\plabel{eqx:superbolicH}
\end{equation}
(parabolic non-normal);
\begin{equation}
L_{t}=\begin{bmatrix} 1& 2t\\
&-1 \end{bmatrix} \qquad  t\geq 0
\plabel{eqx:loxodromicH}
\end{equation}
(non-parabolic case; $t=0$: normal, $t>0$: non-normal).

The representatives above are inequivalent.
\qed
\end{lemma}
\begin{example}\plabel{ex:preDW}
(a) Regarding $L_t$:
\[U_{L_t}  =1+2t^2, \qquad\text{and}\qquad |D_{L_t}| = 1.\]

(b) Regarding $S_0$:
\[U_{S_0}  =\frac12, \qquad\text{and}\qquad |D_{S_0}| = 0.\]

(c) Regarding $\m 0_2$:
\[U_{\m 0_2}  =0, \qquad\text{and}\qquad |D_{\m 0_2}| = 0.\eqedexer\]
\end{example}
\begin{cor} \plabel{cor:cconftriple}
(See \cite{LLL}.)
The ratio
\[
U_A\,\,:\,\,|D_A|
\]
is a full invariant of $2\times2$ complex matrices with respect
to equivalence by complex M\"obius transformations and unitary conjugation.
\qed
\end{cor}
~

\subsection{Norms (review)}
~\\

Recall that the operator norm of $A$ is defined by $\|A\|_2 = \sup\{|A\m x|_2/|\m x|_2\,:\,\m x\neq0\}$.
We define the co-norm of $A$ as $\|A\|_2^-= \inf\{|A\m x|_2/|\m x|_2\,:\,\m x\neq0\}$.
\begin{lemma} \plabel{lem:normcomputeD}
(Cf. \cite{L2})
Suppose that $A$ is a real or complex $2\times2$ matrix.
Then, for the norm,
\begin{align*}
\left\|A \right\|_2&=\sqrt{\frac{\tr(A^*A)}2+\sqrt{\frac{(\tr(A^*A))^2}4-|\det A|^2}}\\
&=\frac{\sqrt{\tr(A^*A)+2|\det A|}+\sqrt{\tr(A^*A)-2|\det A|} }2;
%\end{align*}
%and
%\begin{align*}
\intertext{and, for the co-norm,}
\|A\|_2^-=\left\|A^{-1} \right\|_2^{-1}
&=\sqrt{\frac{\tr(A^*A)}2-\sqrt{\frac{(\tr(A^*A))^2}4-|\det A|^2}}\\
&=\frac{\sqrt{\tr(A^*A)+2|\det A|}-\sqrt{\tr(A^*A)-2|\det A|} }2.
\end{align*}
In particular,
\begin{equation}\|A\|_2\cdot\|A\|_2^-=|\det A|.\plabel{eq:compmult}\end{equation}

In case of real matrices, the values are the same for the Hilbert spaces  $\mathbb R^2$ and $\mathbb C^2$.
\qed
\end{lemma}

\snewpage
\section{The Davis--Wielandt shell of $2\times2$ matrices}
\plabel{sec:DW}

\subsection{Complex asymptotic parametrization and $h$-collineations (review)}\plabel{ss:hypspace}
~\\

A modern review of hyperbolic geometry is contained in, for example, Berger \cite{Ber}.
But, for our purposes, Davis \cite{D1}, \cite{D2} and \cite{L2} are quite practical.
Recall that the complex numbers are embedded to the pCK and BCK models by
\[\iota_{\mathrm{pCK}}(\lambda)=\left(\Rea\lambda,\Ima\lambda,|\lambda|^2\right)
\quad\,\,\text{and}\quad\,\,
\iota_{\mathrm{BCK}}(\lambda)=\left(\frac{2\Rea\lambda}{|\lambda|^2+1},
\frac{2\Ima\lambda}{|\lambda|^2+1},\frac{|\lambda|^2-1}{|\lambda|^2+1}\right)\]
(as asymptotic points); $\iota_{\mathrm{pCK}}(\infty)=\infty_{\mathrm{pCK}}$, $\iota_{\mathrm{BCK}}(\infty)=(0,0,1)$.
If $f:\lambda\mapsto \frac{a\lambda+b}{c\lambda+d}$, $a,b,c,d\in\mathbb C$, $ad-bc\neq0$ is a complex fractional linear
(i.~e.~complex M\"obius) transformation, then
there is a single unique collineation $f_*$ of the asymptotically closed hyperbolic space such that
$f_*(\iota_*(\lambda))=\iota_*(f(\lambda))$ for $\lambda\in\mathbb C\cup\{\infty\}$.
In fact, it induces an orientation-preserving collineation of the hyperbolic space, which results
a restriction of a projective transformation in the pCK and BCK models.

Using projective coordinates in the BCK model, the matrix of the corresponding projective transformation is
\[R_{\mathrm{BCK}}(f)=\frac1{|ad-bc|}\cdot\frac12
\begin{bmatrix}
&1&1&\\&-\mathrm i&\mathrm i&\\1&&&-1\\1&&&1
\end{bmatrix}
\begin{bmatrix}
a\bar a&a\bar b&b\bar a&b\bar b\\a\bar c&a\bar d&b\bar c&b\bar d\\c\bar a&c\bar b&d\bar a&d\bar b\\c\bar c&c\bar d&d\bar c&d\bar d
\end{bmatrix}
\begin{bmatrix}
&&1&1\\1&\mathrm i&&\\1&-\mathrm i&&\\&&-1&1
\end{bmatrix}
\]
\[=\frac1{|ad-bc|}
\begin{bmatrix}
\re(\bar ad+\bar cb)&\ima(\bar ad+\bar cb)&\re(\bar ca-\bar db)&\re(\bar ca+\bar db)
\\
\ima(\bar da-\bar bc)&\re(\bar da-\bar bc)&\ima(\bar ca-\bar db)&\ima(\bar ca+\bar db)
\\
\re(\bar ab-\bar cd)&\ima(\bar ab-\bar cd)&\frac{|a|^2-|b|^2-|c|^2+|d|^2}{2}&\frac{|a|^2+|b|^2-|c|^2-|d|^2}{2}
\\
\re(\bar ab+\bar cd)&\ima(\bar ab+\bar cd)&\frac{|a|^2-|b|^2+|c|^2-|d|^2}{2}&\frac{|a|^2+|b|^2+|c|^2+|d|^2}{2}
\end{bmatrix}.\]

Thus, if \[(\tilde x_{\mathrm{BCK}},\tilde y_{\mathrm{BCK}},\tilde z_{\mathrm{BCK}})=f_{\mathrm{BCK}}\left((x_{\mathrm{BCK}},y_{\mathrm{BCK}},z_{\mathrm{BCK}})\right),\] then
\[\begin{bmatrix}\tilde x_{\mathrm{BCK}}\\\tilde y_{\mathrm{BCK}}\\\tilde z_{\mathrm{BCK}}\\1\end{bmatrix}
\quad \sim\quad R_{\mathrm{BCK}}(f)
\begin{bmatrix}x_{\mathrm{BCK}}\\ y_{\mathrm{BCK}}\\ z_{\mathrm{BCK}}\\1\end{bmatrix},
\]
where $\sim$ means proportional, yielding a projective representation (the relation extending to the asymptotic boundary).
%(Another way to see this is checking this to the asymptotic boundary, and extending from that.)
The matrices $R_{\mathrm{BCK}}(f)$ can be utilized to describe the slightly more complicated
(in the conform case: quadratic rational) transformations in the other models.
(As $R_{\mathrm{BCK}}(f)$, we obtain all matrices from $\SO^{\uparrow}(2,1)$.)

If we want to account for orientation-reversing $h$-collineations, we also have to
 consider skew-M\"obius maps $g:\lambda\mapsto f(\bar\lambda)=\frac{a\bar\lambda+b}{c\bar\lambda+d}$.
The associated matrices are  $R_{\mathrm{BCK}}(g)=R_{\mathrm{BCK}}(f)
\left[\begin{smallmatrix}1&&&\\&-1&&\\&&1&\\&&&1\end{smallmatrix}\right]$.
(By these we obtain all matrices from $ \mathrm O^{\uparrow}(2,1)$.)

~

\snewpage
\subsection{The qualitative elliptical range theorems for the shell (review)}\plabel{ss:DW}
~\\

For an introduction to the Davis--Wielandt shell, see
Wielandt \cite{Wie0}, \cite{Wie}; Davis \cite{D1}, \cite{D2}, \cite{D3}; Li, Poon, Sze \cite{LPS}.
(But especially  \cite{D1} and \cite{D2}.)

It is easy to see that $\DW_{*}(A)$ in invariant for unitary conjugation in $A$.
A less obvious symmetry principle is given by
\begin{theorem}[Wielandt, Davis]
\plabel{thm:DWtrans}
Assume that $A$ is a linear operator on a complex Hilbert space;
If $f:\lambda\mapsto \frac{a\lambda+b}{c\lambda+d}$, $a,b,c,d\in\mathbb C$, $ad-bc\neq0$ is a complex fractional linear
(i.~e.~complex M\"obius) transformation,
then for $f(A)=\frac{aA+b}{cA+d}$ (if it exists), the Davis--Wielandt shell is given as
\[\DW_*(f(A))=f_*(\DW_*(A)).\]
\end{theorem}
\begin{proof}[Proofs]
See Davis \cite{D1}, \cite{D2} for detailed arguments (or some discussions in \cite{L2}).
\end{proof}
\begin{remark}
If $A$ is bounded linear operator, and $g\mapsto \frac{a\bar\lambda+b}{c\bar\lambda+d}$,
 then the statement above extends with $g(A)=\frac{a A^*+b}{cA^*+d}$ (if it exists).
\qedremark
\end{remark}

Returning to $2\times2$ complex matrices:
As we have mentioned a theorem of Davis shows that $\DW_{\mathrm{BCK}}(A)$
 is a possibly degenerate ellipsoid in Euclidean view.
Using simple symmetry principles, it is not hard to put this statement into more specific form:
\begin{theorem}[Wielandt, Davis]
\plabel{thm:DWconc}

Suppose that $A$ is a linear operator on a $2$-dimensional complex Hilbert space.
We have the following possibilities:

(i) $A$ has a double eigenvalue $\lambda$, and $A$ is normal (thus $A=\lambda \Id$).

Then $\DW_*(A)$ contains only the point $\iota_*(\lambda)$.

(ii)  $A$ has two different eigenvalues $\lambda_1\neq\lambda_2$, and $A$ is normal.

Then $\DW_*(A)$ is the asymptotically closed $h$-line connecting $\iota_*(\lambda_1)$ and $\iota_*(\lambda_2)$.

(iii) $A$ has a double eigenvalue $\lambda$, and $A$ is not normal.

Then $\DW_*(A)$ is an asymptotically closed $h$-horosphere with  asymptotical point $\iota_*(\lambda)$.
In the $\mathrm{BCK}$ model this is an ellipsoid, whose equation is linearly
generated by the quadratic equation of the unit sphere
and the equation of the double plane tangent to unit sphere at $\iota_*(\lambda)$.

(iv) $A$ has two different eigenvalues $\lambda_1\neq\lambda_2$, and $A$ is not normal.

Then $\DW_*(A)$ is the an asymptotically closed $h$-tube around the $h$-line connecting $\iota_*(\lambda_1)$ and $\iota_*(\lambda_2)$.
In the $\mathrm{BCK}$ model this is an ellipsoid, whose equation is linearly
generated by the quadratic equation of the unit sphere
and the quadratic equation of the union of planes tangent to unit sphere at $\iota_*(\lambda_1)$ and $\iota_*(\lambda_2)$.
\end{theorem}
%\snewpage

The theorem above comprises various observations of Wielandt \cite{Wie0}, \cite{Wie} and Davis \cite{D1}, \cite{D2} and
 known facts from analytic hyperbolic geometry;
 but see it specifically explained in \cite{L2} Appendix B, or, in greater detail, in \cite{LLL}.

\snewpage

Let us make two observations regarding the shells of $2\times2$ matrices.
The first observation is that the shell as a closed compact set (in the pCK or BCK models) is continuous %in
with respect to the Hausdorff topology.
This is a simple consequence of its construction as a continuous image of the a sphere.
Consequently, limit arguments in the shell apply.
(Remark: This observation naturally extends to the numerical and conformal ranges.)
The second observation is that if we obtain an equation for the shell from the pencil construction,
say, in the case of a normal non-parabolic matrix, then the resulted equation is not just any
line ellipse for the shell but the ``infinitesimal tube'', i.~e.~degenerate $h$-\textit{circular} cylinder.
This also applies if we obtain the quadratic form of the shell as a nonzero limit
of quadratic forms from non-normal matrices.
In the case of normal parabolic matrices similar comment applies with respect to ``double tangent planes''.
(Remark: This observation does not extend to the numerical and conformal ranges.
Although there are natural point circles, they are not necessarily the non-zero limits.)

\snewpage
\subsection{Other models; distance (review)}\plabel{ss:other}
~\\

The Poincaré half-space model (Ph) and the Poincaré model (P) are the conformal versions of the
 pCK and BCK models, respectively.
(These models were also constructed by Beltrami first, but, for the sake of simplicity we use the customary terminology.)
The canonical correspondences are described as follows.
Between Ph and pCK:
\[(x_{\mathrm{Ph}},y_{\mathrm{Ph}},z_{\mathrm{Ph}})=\left(x_{\mathrm{pCK}},y_{\mathrm{pCK}},
\sqrt{z_{\mathrm{pCK}} -(x_{\mathrm{pCK}})^2-(y_{\mathrm{pCK}})^2}\right); \]
and $\infty_{\langle0,0,1\rangle}$ of pCK corresponds to $\infty$ in Ph.
Between P and BCK:
\[(x_{\mathrm{P}},y_{\mathrm{P}},z_{\mathrm{P}})=
\frac{(x_{\mathrm{BCK}},y_{\mathrm{BCK}},z_{\mathrm{BCK}})}{
1+\sqrt{
1-(x_{\mathrm{BCK}})^2-(y_{\mathrm{BCK}})^2-(z_{\mathrm{BCK}})^2
}
}.\]
For us only Ph has some limited relevance.

The natural hyperbolic metric is characterized by collineation-invariance and the property that the constant ratio between the area and the angular defect is exactly $1$.
(In hyperbolic geometry, the notions of collineation and isometry will coincide.)
%We will not go into a detailed description of the metric but we mention that
In the BCK model the distance function is given by
\begin{multline}\mathrm d^{\mathrm{BCK}}((x_{\mathrm{BCK}},y_{\mathrm{BCK}},z_{\mathrm{BCK}}),(\tilde x_{\mathrm{BCK}},\tilde y_{\mathrm{BCK}},\tilde z_{\mathrm{BCK}}))=\\
=\arcosh\left(\frac{1-x_{\mathrm{BCK}}\tilde x_{\mathrm{BCK}}- y_{\mathrm{BCK}}\tilde y_{\mathrm{BCK}} - z_{\mathrm{BCK}}\tilde z_{\mathrm{BCK}}  }{\sqrt{1-(x_{\mathrm{BCK}})^2-(y_{\mathrm{BCK}})^2-(z_{\mathrm{BCK}})^2}\sqrt{1-(\tilde x_{\mathrm{BCK}})^2-(\tilde y_{\mathrm{BCK}})^2-(\tilde z_{\mathrm{BCK}})^2}}\right)
%\\
%=\frac12\arcosh\left(\frac{
%2(1-x_{\mathrm{BCK}}\tilde x_{\mathrm{BCK}}- y_{\mathrm{BCK}}\tilde y_{\mathrm{BCK}} - z_{\mathrm{BCK}}\tilde z_{\mathrm{BCK}})^2  }{
%(1-(x_{\mathrm{BCK}})^2-(y_{\mathrm{BCK}})^2-(z_{\mathrm{BCK}})^2)
%(1-(\tilde x_{\mathrm{BCK}})^2-(\tilde y_{\mathrm{BCK}})^2-(\tilde z_{\mathrm{BCK}})^2)}-1\right)
.
\plabel{eq:dis}
 \end{multline}
\begin{commentx}
In the pCK model it is given by
\begin{multline}
\mathrm d^{\mathrm{pCK}}((x_{\mathrm{pCK}},y_{\mathrm{pCK}},z_{\mathrm{pCK}}),(\tilde x_{\mathrm{pCK}},\tilde y_{\mathrm{pCK}},\tilde z_{\mathrm{pCK}}))=\\
=\arcosh\left( \frac{  \frac12z_{\mathrm{pCK}}+\frac12\tilde z_{\mathrm{pCK}}-x_{\mathrm{pCK}}\tilde x_{\mathrm{pCK}}-y_{\mathrm{pCK}}\tilde y_{\mathrm{pCK}}
}{\sqrt{ z_{\mathrm{pCK}}- (x_{\mathrm{pCK}})^2-(y_{\mathrm{pCK}})^2}\sqrt{\tilde  z_{\mathrm{pCK}}- (\tilde x_{\mathrm{pCK}})^2-(\tilde y_{\mathrm{pCK}})^2 }}\right).
\plabel{eq:disP}
\end{multline}

In the Ph model it is given by
\begin{align*}
\mathrm d^{\mathrm{Ph}}(&(x_{\mathrm{Ph}},y_{\mathrm{Ph}},z_{\mathrm{Ph}}),(\tilde x_{\mathrm{Ph}},\tilde y_{\mathrm{Ph}},\tilde z_{\mathrm{Ph}}))=\\
&=\arcosh\left( 1+\frac{(x_{\mathrm{Ph}}-\tilde x_{\mathrm{Ph}})^2+(y_{\mathrm{Ph}}-\tilde y_{\mathrm{Ph}})^2+ (z_{\mathrm{Ph}}-\tilde z_{\mathrm{Ph}})^2}{2z_{\mathrm{Ph}}\tilde z_{\mathrm{Ph}}}\right)\\
&=2\arsinh\left( \frac{\sqrt{(x_{\mathrm{Ph}}-\tilde x_{\mathrm{Ph}})^2+ (y_{\mathrm{Ph}}-\tilde y_{\mathrm{Ph}})^2+(z^*_{\mathrm{Ph}}-\tilde z_{\mathrm{Ph}})^2}}{2\sqrt{z_{\mathrm{Ph}}\tilde z_{\mathrm{Ph}}}}\right).
\end{align*}

We remark that, in particular,
\begin{align}
\mathrm d^{\mathrm{BCK}}( (0,0,0),(s,0,0))
&=\arcosh\frac{1}{\sqrt{1-s^2}}
=\arsinh\sqrt{\frac{s^2}{1-s^2}}
=\artanh |s|
\plabel{eq:distex}
\\
\notag
&=\frac12\arcosh\frac{1+s^2}{1-s^2}
=\frac12\arsinh\frac{2|s|}{1-s^2}
=\frac12\artanh\frac{2|s|}{1+s^2}
.
\end{align}
\end{commentx}
(One can find more in \cite{L2}, Appendix A.)

\subsection{The matrix of the shell }\plabel{ss:mshell}
~\\

We start with the explicit equation (and its quadratic matrix) for the Davis--Wielandt shell of $2\times2$ matrices.
This was found by Lins, Spitkovsky, Zhong \cite{LSZ}.
The equation has sort of the nicest shape in the $\mathrm{pCK}$ model.
We give a several proofs.

\begin{theorem}[Lins, Spitkovsky, Zhong \cite{LSZ}]
\plabel{thm:LSZ}
Suppose that $A$ is linear operator on a $2$-dimensional Hilbert space.
Then $\DW_{\mathrm{pCK}}(A)$ is given by the quadratic equation
\begin{equation}
\begin{bmatrix}x_{\mathrm{pCK}}\\ y_{\mathrm{pCK}}\\ z_{\mathrm{pCK}}\\1\end{bmatrix}^\top
\m Q_{\mathrm{pCK}}(A)
\begin{bmatrix}x_{\mathrm{pCK}}\\ y_{\mathrm{pCK}}\\ z_{\mathrm{pCK}}\\1\end{bmatrix}=0
\plabel{eq:LSZ1}
\end{equation}
where
\begin{multline}
\m Q_{\mathrm{pCK}}(A)=\underbrace{U_A\,2
\begin{bmatrix}1\\&1\\&&&-\frac12\\&&-\frac12\end{bmatrix}}_{\m Q_{\mathrm{pCK}}^{\mathrm{bas}}(A):=}+\plabel{eq:LSZ2}\\\\
+\underbrace{\begin{bmatrix}
\dfrac{|\tr A|^2}2+2\Rea\det A&2\Ima\det A&{-\Rea\tr A}&{-\Rea((\det A)\overline{(\tr A)})}\\
2\Ima\det A&\dfrac{|\tr A|^2}2-2\Rea\det A&{-\Ima\tr A}&{-\Ima((\det A)\overline{(\tr A)})}\\
{-\Rea\tr A}&{-\Ima\tr A}&1&\dfrac{|\tr A|^2}4\\
{-\Rea((\det A)\overline{(\tr A)})} &{-\Ima((\det A)\overline{(\tr A)})}&\dfrac{|\tr A|^2}4&|\det A|^2
\end{bmatrix}}_{\m Q_{\mathrm{pCK}}^{\mathrm{spec}}(A):=}=
\\\\
=\begin{bmatrix}
(\tr A^*A)+2\Rea\det A&2\Ima\det A&-\Rea\tr A&-\Rea((\det A)\overline{(\tr A)})\\
2\Ima\det A&(\tr A^*A)-2\Rea\det A&-\Ima\tr A&-\Ima((\det A)\overline{(\tr A)})\\
-\Rea\tr A&-\Ima\tr A&1&\frac{|\tr A|^2-\tr(A^*A)}2\\
-\Rea((\det A)\overline{(\tr A)}) &-\Ima((\det A)\overline{(\tr A)})&\frac{|\tr A|^2-\tr(A^*A)}2&|\det A|^2
\end{bmatrix}.
\end{multline}

Here
\[\det\m Q_{\mathrm{pCK}}(A)=-4\left( (U_A)^2-|D_A|^2\right)^2.\]

In the non-normal cases, the equation yields ellipsoids in the model space.
In the normal non-parabolic case the equation yields a line ellipsoid (degenerate elliptic cylinder).
In the normal parabolic case, the equation yields a double plane tangent at the corresponding asymptotic point.
(Thus in the normal case one should restrict the solutions to the model space  $\mathrm{pCK}$.)

In the non-normal cases, the interior points $(x_{\mathrm{pCK}}, y_{\mathrm{pCK}}, z_{\mathrm{pCK}})$ of the ellipses are characterized by
\begin{commentx}
\begin{equation}
\begin{bmatrix}x_{\mathrm{pCK}}\\ y_{\mathrm{pCK}}\\ z_{\mathrm{pCK}}\\1\end{bmatrix}^\top
\m Q_{\mathrm{pCK}}(A)
\begin{bmatrix}x_{\mathrm{pCK}}\\ y_{\mathrm{pCK}}\\ z_{\mathrm{pCK}}\\1\end{bmatrix}<0.
\end{equation}
\end{commentx}
replacing `$=0$' with `$<0$' in \eqref{eq:LSZ1}.

\proofremark{
Considering the canonical correspondences between the various models (cf. \cite{L2}), we find that
in \eqref{eq:LSZ1}, replacing
\[\begin{bmatrix}x_{\mathrm{pCK}}\\ y_{\mathrm{pCK}}\\ z_{\mathrm{pCK}}\\1\end{bmatrix}
\text{ by }
\begin{bmatrix}x_{\mathrm{BCK}}\\ y_{\mathrm{BCK}}\\ 1+z_{\mathrm{BCK}}\\1-z_{\mathrm{BCK}}\end{bmatrix},
\begin{bmatrix}x_{\mathrm{P}}\\ y_{\mathrm{P}}\\\frac{(x_{\mathrm{P}})^2+(y_{\mathrm{P}})^2+(z_{\mathrm{P}}+1)^2}2,\\
\frac{(x_{\mathrm{P}})^2+(y_{\mathrm{P}})^2+(z_{\mathrm{P}}-1)^2}2\end{bmatrix},
\begin{bmatrix}x_{\mathrm{Ph}}\\ y_{\mathrm{Ph}}\\
(x_{\mathrm{Ph}})^2+(y_{\mathrm{Ph}})^2+(z_{\mathrm{Ph}})^2\\1\end{bmatrix}
\]
one obtains the corresponding equations in $\mathrm{BCK}$, $\mathrm{P}$, $\mathrm{Ph}$, respectively.

In the case of the BCK model, the transition from pCK is projective / linear (cf. \eqref{eq:can2}).
Thus another possible viewpoint is that
\begin{equation}
\begin{bmatrix}x_{\mathrm{BCK}}\\ y_{\mathrm{BCK}}\\ 1+z_{\mathrm{BCK}}\\1-z_{\mathrm{BCK}}\end{bmatrix}^\top
\m Q_{\mathrm{pCK}}(A)
\begin{bmatrix}x_{\mathrm{BCK}}\\ y_{\mathrm{BCK}}\\ 1+z_{\mathrm{BCK}}\\1-z_{\mathrm{BCK}}\end{bmatrix}
\equiv
\begin{bmatrix}x_{\mathrm{BCK}}\\ y_{\mathrm{BCK}}\\ z_{\mathrm{BCK}}\\1\end{bmatrix}^\top
\m Q_{\mathrm{BCK}}(A)
\begin{bmatrix}x_{\mathrm{BCK}}\\ y_{\mathrm{BCK}}\\ z_{\mathrm{BCK}}\\1\end{bmatrix}
;
\plabel{eq:trans30}
\end{equation}
where
\begin{equation}
\m Q_{\mathrm{BCK}}(A)
=
\begin{bmatrix}
1&&&\\&1&&\\&&1&1\\&&-1&1
\end{bmatrix}^\top
\m Q_{\mathrm{pCK}}(A)
\begin{bmatrix}
1&&&\\&1&&\\&&1&1\\&&-1&1
\end{bmatrix}
\plabel{eq:trans3}
\end{equation}
is obtained canonically.
Similarly, we adopt the scheme of \eqref{eq:trans30}/\eqref{eq:trans3} in order to transcribe
$\m Q^{\mathrm{spec}}_{\mathrm{pCK}}(A)$ to $\m Q^{\mathrm{spec}}_{\mathrm{BCK}}(A)$.
Due to projectivity / linearity, similar comments apply concerning  the shape of solution sets in the BCK model as in the pCK model.
}
\end{theorem}

There are %, moreover,
some satellite statements which are worth mentioning
in connection to Theorem \ref{thm:LSZ}.
For this reason, we set
\begin{equation}
\m Q^{0}_{\mathrm{pCK}}=
\begin{bmatrix}1&&&\\&1&&\\&&&-\frac12\\&&-\frac12&\end{bmatrix},
\qquad\text{and}\qquad
\m Q^{0}_{\mathrm{BCK}}=
\begin{bmatrix}1&&&\\&1&&\\&&1&\\&&&-1\end{bmatrix}
.\plabel{eq:biQ0}
\end{equation}
These are also connected as in \eqref{eq:trans30}/\eqref{eq:trans3}.
Consequently,
$\Bigl(\m Q^{0}_{\mathrm{pCK}}\Bigr)^{-1}\m Q_{\mathrm{pCK}}(A)$,
and
$\left(\m Q^{0}_{\mathrm{BCK}}\right)^{-1}\m Q_{\mathrm{BCK}}(A)$
are conjugates of each other.
\begin{theorem}
\plabel{thm:eigDW}
The eigenvalues of
$\Bigl(\m Q^{0}_{\mathrm{pCK}}\Bigr)^{-1}\m Q_{\mathrm{pCK}}(A)
%$,
%which are the same as the eigenvalues of
%$
\sim
\left(\m Q^{0}_{\mathrm{BCK}}\right)^{-1}\m Q_{\mathrm{BCK}}(A)$
%,
are
\[2(U_A-|D_A|) \quad\times2,\quad 2( U_A+|D_A|) \quad\times2.\]
In particular, we can recover $U_A : |D_A|$ from the ratio of the eigenvalues.
\end{theorem}

\begin{theorem}\plabel{thm:LSZdeg}
(I) Geometrically:

In \eqref{eq:LSZ2}, replacing ``$U_A$'' by ``$\left|D_A \right|$''
one obtains the ``axis'' of the $h$-tube in the complex non-parabolic case;
replacing ``$U_A$'' by ``$-\left|D_A \right|$'' one obtains the equation for the  pair of planes tangent at the asymptotic points
corresponding to the eigenvalues of $A$.
In the parabolic case $D_A=0$, one obtains the double plane tangent at the asymptotic point corresponding to the eigenvalue.
After these replacements the  associated quadratic function will be nonnegative on the points of the (asymptotically closed) $\mathrm{pCK}$ model.

(II) Numerically:
Assume that $A$ has eigenvalues $\lambda_1$, $\lambda_2$.

(a) Regarding the biplanar case,
\[\underbrace{-|D_A|2\m Q^{0}_{\mathrm{pCK}}+\m Q^{\mathrm{spec}}_{\mathrm{pCK}}(A)
}_{\m Q^{\mathrm{eig}}_{\mathrm{pCK}}(A):=}
=\frac12\left(\m v_1\m v_2^\top +\m v_2\m v_1^\top \right),\]
where
\[\m v_1=\bem -2\Rea \lambda_1\\ -2\Ima \lambda_1\\ 1\\ |\lambda_1|^2\eem
\qquad  \text{and}\qquad
\m v_2=\bem -2\Rea \lambda_2\\ -2\Ima \lambda_2\\ 1\\ |\lambda_2|^2\eem.
\]

(b) Regarding the infinitesimal tube,
\[\underbrace{|D_A|2\m Q^{0}_{\mathrm{pCK}}+\m Q^{\mathrm{spec}}_{\mathrm{pCK}}(A)
}_{\m Q^{\mathrm{ax}}_{\mathrm{pCK}}(A):=}
= \m v \m v^\top +\m w \m w^\top ,\]
where
\[\m v =\bem -(\Rea \lambda_1)-(\Rea \lambda_2)\\ -(\Ima \lambda_1)-(\Ima \lambda_2) \\ 1\\
 (\Rea \lambda_1)(\Rea \lambda_2) +(\Ima \lambda_1)(\Ima \lambda_2)  \eem
\quad  \text{and}\quad
\m w=\bem  (\Ima \lambda_1)-(\Ima \lambda_2)\\  (\Rea \lambda_2)-(\Rea \lambda_1)\\ 0\\
(\Rea \lambda_1)(\Ima \lambda_2)-(\Ima \lambda_1)(\Rea \lambda_2)
\eem.
\]
\end{theorem}

\begin{remark}\plabel{rem:LSZspec}
The matrices
 $\m Q^{\mathrm{spec}}_{\mathrm{pCK}}(A)$,
 $\m Q^{\mathrm{ax}}_{\mathrm{pCK}}(A)$,
 $\m Q^{\mathrm{eig}}_{\mathrm{pCK}}(A)$
 can equally be considered as the spectral part of $\m Q_{\mathrm{pCK}}(A)$.
The decomposition
 $\m Q_{\mathrm{pCK}}(A)=(U_A-|D_A|)2\m Q^{0}_{\mathrm{pCK}}+\m Q^{\mathrm{ax}}_{\mathrm{pCK}}(A)$
 expresses that how non-normalitity ``thickens'' the axis (infinitesimal $h$-tube) to an ellipsoid ($h$-tube).
\qedremark
\end{remark}
\snewpage
The first approach we consider is the method of  ``brute force''.
\begin{proof}[First Proof of Theorem \ref{thm:LSZ}]
Assume that $A$ is not normal.
Recall, in projective coordinates
$\DW_{\mathrm{pCK}}(\m x,A\m x)=(
\Rea\langle A\m x,\m x\rangle :
\Ima\langle A\m x,\m x\rangle :
\langle A\m x,A\m x\rangle:
\langle \m x,\m x\rangle)$.
Assume that $A=\bem a&b\\c&d\eem$ and
$\m x=\bem z_1\\z_2\eem$
%$\m x=z_1\m e_1+z_2\m e_2$
with $|z_1|^2+|z_2|^2=1$.
By direct computation, we find
\[
%\begin{multline*}
\bem
\Rea\langle A\m x,\m x\rangle \\
\Ima\langle A\m x,\m x\rangle \\
\langle A\m x,A\m x\rangle\\
\langle \m x,\m x\rangle
\eem
=%\\=
\underbrace{\bem
\Rea\frac{b+c}2&\Ima\frac{b-c}2&\Rea\frac{a-d}2&\Rea\frac{a+d}2\\
\Ima\frac{b+c}2&\Rea\frac{c-b}2&\Ima\frac{a-d}2&\Ima\frac{a+d}2\\
\Rea(\bar ab+\bar cd)&\Ima(\bar ab+\bar cd)&\frac{|a|^2-|b|^2+|c|^2-|d|^2}2&\frac{|a|^2+|b|^2+|c|^2+|d|^2}2\\
\\0&0&0&1
\eem}_{\m S_{\mathrm{pCK}}(A):=}
\bem
2\Rea( z_1\bar z_2)\\
2\Ima( z_1\bar z_2)\\
|z_1|^2-|z_2|^2\\
|z_1|^2+|z_2|^2
\eem .
\]%\end{multline*}
As the last column  parameterizes the unit sphere in affine extended coordinates (as $|z_1|^2+|z_2|^2\equiv1$,
it yields the affine extension of $\iota_{\mathrm{BCK}}(z_1/z_2)$),
and $\m S_{\mathrm{pCK}}(A)$ is itself an affine extended matrix, we see that we deal with
an affine linear image of the unit sphere.
One can check directly that
\begin{equation}
\det\m S_{\mathrm{pCK}}(A)=\frac12\left((U_A)^2-|D_A|^2\right),
\plabel{eq:simp1}
\end{equation}
which is non-zero in the non-normal case.
In that case, the image is a proper ellipsoid with matrix
\begin{equation}
\m S_{\mathrm{pCK}}(A)^{-1,\top}\bem1&&&\\&1&&\\&&1&\\&&&-1 \eem\m S_{\mathrm{pCK}}(A)^{-1}=
\frac1{(U_A)^2-|D_A|^2}
\m Q_{\mathrm{pCK}}(A).
\plabel{eq:simp2}
\end{equation}
The nonzero scalar multiplier may be omitted, so we obtain the matrix indicated.
Due to nature of the object obtained, both the determinant and the interior property is straightforward.
The case when $A$ is normal follows from limiting arguments.
(The limit matrix is nonzero as the entry in the $(3,3)$ position is $1$.)

(Remark:
This proof is very computational as the simplifications in \eqref{eq:simp1} and \eqref{eq:simp2} are cumbersome
even in the case of the canonical triangular form.
Note, however, that this method does not require any knowledge about the Davis--Wielandt shell beyond its definition;
and it can also be used effectively for any concrete example.)
\end{proof}
\begin{proof}[First Proof of Theorem \ref{thm:eigDW}]
It follows from computing the characteristic polynomial of
$\Bigl(\m Q^{0}_{\mathrm{pCK}}\Bigr)^{-1}\m Q_{\mathrm{pCK}}(A)$.
(Remark: This is long even in terms of the five data.)
\end{proof}
\begin{proof}[First Proof of Theorem \ref{thm:LSZdeg}]
From the Theorem \ref{thm:eigDW}, we see that the singular quadratic forms of the pencil generated by
$\m Q^{0}_{\mathrm{pCK}}$ and $\m Q_{\mathrm{pCK}}(A)$ are given,
up to nonzero scalar multiples, by
$\pm|D_A|2\m Q^{0}_{\mathrm{pCK}}+\m Q^{\mathrm{spec}}_{\mathrm{pCK}}(A)$.
In the non-parabolic case, bringing $A$ to canonical form by unitary conjugation
and taking off the off-diagonal term turns $U_A$ into $|D_A|$, and this results in
$|D_A|2\m Q^{0}_{\mathrm{pCK}}+\m Q^{\mathrm{spec}}_{\mathrm{pCK}}(A)$
as the matrix of the axial tube.
This leaves
$-|D_A|2\m Q^{0}_{\mathrm{pCK}}+\m Q^{\mathrm{spec}}_{\mathrm{pCK}}(A)$
as the matrix of the biplanar quadratic form.
\end{proof}

\snewpage
The second approach we consider uses the conformal invariance of the shell.
\begin{proof}[Second Proof of Theorem \ref{thm:LSZ}]
We can consider the matrices
\begin{equation}
L_t=\begin{bmatrix}1&2t\\&-1\end{bmatrix}
\plabel{nt:Lt}
\end{equation}
$(t\geq0)$.
In this case, $\pm1$ are eigenvalues, thus $\iota_{\mathrm{BCK}}(\pm1)=(\pm 1,0,0)$ will be
the asymptotic points of the $h$-tube $\DW_{\mathrm{BCK}}(L_t)$.
In Euclidean view, the shell will be an ellipsoid with rotation axis $[(-1,0,0),(1,0,0)]_{\mathrm e}$,
where the notation refers to the Euclidean segment connecting the respective points.
As
\[
\DW_{\mathrm{BCK}}\left( \begin{bmatrix}0\\1\end{bmatrix},L_t\begin{bmatrix}0\\1\end{bmatrix}\right)=
\DW_{\mathrm{BCK}}\left( \begin{bmatrix}0\\1\end{bmatrix},\begin{bmatrix}2t\\-1\end{bmatrix}\right)
=\left(\dfrac{-1}{1+t^2} , 0,\dfrac{2t^2}{1+t^2}\right),\]
it can be checked that the fitting ellipsoid, for $t>0$, is
\[(x_{\mathrm{BCK}})^2+\frac{(y_{\mathrm{BCK}})^2}{\dfrac{t^2}{1+t^2}}+ \frac{(z_{\mathrm{BCK}})^2}{\dfrac{t^2}{1+t^2}}-1=0;\]
or, in order to accommodate the limiting normal case $t=0$,
\begin{equation}
4t^2(x_{\mathrm{BCK}})^2+4(1+t^2)(y_{\mathrm{BCK}})^2+ 4(1+t^2)(z_{\mathrm{BCK}})^2-4t^2=0.
\plabel{eq:pretube}
\end{equation}
(This was the initial argument of the proof.)

Abstractly, up to conjugation by unitary matrices, this have dealt with the case of
\[\tr L=0
%\]
\qquad\text{and}\qquad
%\[
\det L=-1,\]
being
\[t^2=\frac12\left(\frac{\tr (L^*L)}2-1\right).\]
%In the general case, if

If $A$ is non-parabolic, such $L$ can be obtained by taking the  conformal transform
\[L=\frac{A-\dfrac{\tr A}2\Id}{\sqrt{-\det\left(A-\dfrac{\tr A}2\Id \right)} }.\]
Then
\begin{equation}
t^2=\frac12\left(\frac{\dfrac12\tr\,\left(A-\dfrac{\tr A}2\Id\right)^*\left(A-\dfrac{\tr A}2\Id\right)}{\left|\det\left(A-\dfrac{\tr A}2\Id\right) \right| }-1\right)
=\frac12\left(\frac{U_A}{\left|D_A\right| }-1\right).
\plabel{eq:tubet}
\end{equation}

Let us use the temporary notation
\begin{equation} T=-\frac{\tr A}2,\qquad
\Delta=\sqrt{-D_A}
\plabel{eq:tempnote}
\end{equation}
(one complex value chosen).
Then the equation \eqref{eq:pretube} for $\DW_{\mathrm{BCK}}(L)$ can be rewritten as
\begin{multline}
2(U_A+|\Delta|^2)(x_{\mathrm{BCK}})^2+2(U_A-|\Delta|^2)(y_{\mathrm{BCK}})^2\\+ 2(U_A+|\Delta|^2)(z_{\mathrm{BCK}})^2+2(-U_A+|\Delta|^2)=0.
\plabel{eq:protos}
\end{multline}
Thus the corresponding quadratic matrix, is
\begin{equation}
\m Q_{\mathrm{BCK}}^{[A]}(L)=\underbrace{2U_A\begin{bmatrix}1\\&1\\&&1\\&&&-1\end{bmatrix}}_{\m Q_{\mathrm{BCK}}^{\mathrm{bas}}(L)}
+\underbrace{2|\Delta|^2\begin{bmatrix}-1\\&1\\&&1\\&&&1\end{bmatrix}}_{\m Q_{\mathrm{BCK}}^{\mathrm{spec}}(L)}.
\plabel{eq:protos2}
\end{equation}
(The notation indicates that the scaling was chosen with respect to $A$.)

The $L$ was, however, obtained after a displacement by the conformal transformation $f$ such that
\[f(\lambda)=\frac{\lambda-\dfrac{\tr A}2}{\sqrt{-\det\left(A-\dfrac{\tr A}2\Id \right)} }\equiv \frac{\lambda+T}\Delta, \]
whose projective representation matrix with respect to the BCK model is
\[R_{\mathrm{BCK}}(f)=\frac1{|\Delta|}\cdot\frac12
\begin{bmatrix}
&1&1&\\&-\mathrm i&\mathrm i&\\1&&&-1\\1&&&1
\end{bmatrix}
\begin{bmatrix}
1&\overline{T}&T&|T|^2\\&\overline{\Delta}&&\overline{\Delta}T\\&&\Delta&\overline{T}\Delta\\&&&|\Delta|^2
\end{bmatrix}
\begin{bmatrix}
&&1&1\\1&\mathrm i&&\\1&-\mathrm i&&\\&&-1&1
\end{bmatrix}.
\]
Thus a matrix for the quadratic equation for $\DW_{\mathrm{pCK}}(A)$ should be
\begin{equation}
\m Q^{\mathrm{candidate}}_{\mathrm{pCK}}(A)=
\begin{bmatrix}
1&&&\\&1&&\\&&\frac12&-\frac12\\&&\frac12&\frac12
\end{bmatrix}^\top\cdot R_{\mathrm{BCK}}(f)^\top
\cdot\m Q_{\mathrm{BCK}}^{[A]}(L)
\cdot
R_{\mathrm{BCK}}(f)\cdot
\begin{bmatrix}
1&&&\\&1&&\\&&\frac12&-\frac12\\&&\frac12&\frac12
\end{bmatrix}.
\plabel{eq:finDW}
\end{equation}
Here the middle term is quadratic matrix for $\DW_{\mathrm{BCK}}(L)$, then $R_{\mathrm{BCK}}(f)$ takes care to the displacement in $\mathrm{BCK}$,
and the outer term translates between $\mathrm{BCK}$ and $\mathrm{pCK}$.
Expanded, it
\[= \left[ \begin {smallmatrix}
2\,T\overline{T}-{\Delta}^{2}+\overline{T}  ^{2}-  \overline{\Delta}  ^{2}+{T}^{2}+2\,U_A&
-\mathrm i  \overline{\Delta}  ^{2}+\mathrm i{\Delta}^{2}-\mathrm i{T}^{2}+\mathrm i  \overline{T}  ^{2}&
T+\overline{T}&
\overline{T}  ^{2}T+{T}^{2}\overline{T}-{\Delta}^{2}\overline{T}-  \overline{\Delta}  ^{2}T
\\
\noalign{\medskip}-\mathrm i  \overline{\Delta}  ^{2}+\mathrm i{\Delta}^{2}-\mathrm i{T}^{2}+\mathrm i  \overline{T}  ^{2}&
2\,T\overline{T}+{\Delta}^{2}-  \overline{T}  ^{2}+\overline{\Delta}  ^{2}-{T}^{2}+2\,U_A&
\mathrm i\overline{T}-\mathrm iT&
\mathrm i{\Delta}^{2}\overline{T}+\mathrm i  \overline{T}  ^{2}T-\mathrm i{T}^{2}\overline{T}-\mathrm i  \overline{\Delta}  ^{2}T
\\
\noalign{\medskip}T+\overline{T}&
\mathrm i\overline{T}-\mathrm iT&
1&
T\overline{T}-U_A
\\
\noalign{\medskip}  \overline{T}  ^{2}T+{T}^{2}
\overline{T}-{\Delta}^{2}\overline{T}-  \overline{\Delta}^{2}T&
\mathrm i{\Delta}^{2}\overline{T}+\mathrm i  \overline{T}^{2}T-\mathrm i{T}^{2}\overline{T}-\mathrm i  \overline{\Delta}^{2}T&
T\overline{T}-U_A&
-{\Delta}^{2}  \overline{T}^{2}+{T}^{2}  \overline{T}  ^{2}+{\Delta}^{2}\overline{\Delta}  ^{2}-  \overline{\Delta}
  ^{2}{T}^{2}\end {smallmatrix} \right].
 \]
Then, taking \eqref{eq:tempnote} into consideration, it yields
$\m Q^{\mathrm{candidate}}_{\mathrm{pCK}}(A)\equiv\m Q_{\mathrm{pCK}}(A)$
indeed.
From \eqref{eq:protos2} (where $|\Delta|^2=|D_A|$) and \eqref{eq:finDW} (where $\det R_{\mathrm{BCK}}(f)=1$)
the determinant follows immediately.
The characterization of the interior carries from the case of $L_t$ again.

Now, the argument above establishes the  non-parabolic case.
Then the parabolic
case follows by continuity.
(The limits %of the quadratic forms
are nontrivial, as the entries at $(3,3)$ are $1$.)
\end{proof}

\begin{proof}[Second Proof for \ref{thm:eigDW}]
Continuing the previous Second Proof of Theorem \ref{thm:LSZ},
\[\left(\m Q^{0}_{\mathrm{BCK}}\right)^{-1}\m Q_{\mathrm{BCK}}(A)=
\left( R_{\mathrm{BCK}}(f)^\top
\cdot\m Q^{0}_{\mathrm{BCK}}
\cdot R_{\mathrm{BCK}}(f)\right)^{-1}
 R_{\mathrm{BCK}}(f)^\top
\cdot\m Q_{\mathrm{BCK}}^{[A]}(L)
\cdot R_{\mathrm{BCK}}(f)\]
\[=R_{\mathrm{BCK}}(f)^{-1} \cdot\left(\m Q^{0}_{\mathrm{BCK}}\right)^{-1}\m Q_{\mathrm{BCK}}^{[A]}(L)\cdot R_{\mathrm{BCK}}(f)
\sim\left(\m Q^{0}_{\mathrm{BCK}}\right)^{-1}\m Q_{\mathrm{BCK}}^{[A]}(L).\]
From this, the conjugacy type
of $\left(\m Q^{0}_{\mathrm{BCK}}\right)^{-1}\m Q_{\mathrm{BCK}}(A)$
can be read off immediately in the non-parabolic case.
Due to the continuity of the characteristic polynomial the eigenvalues extend (but the conjugacy type does not).
\end{proof}

\begin{proof}[Second Proof of Theorem \ref{thm:LSZdeg}]
In the previous Second Proof of Theorem \ref{thm:LSZ}, we obtain the corresponding matrices by
making the appropriate replacements $\pm|D_A|$ to $U_A$
in \eqref{eq:protos}/\eqref{eq:protos2} in the non-parabolic case $|D_A|\neq 0$.
The  parabolic case $|D_A|= 0$ follows from limiting arguments.
 \end{proof}

\begin{remark}
\plabel{rem:altinit}

(a) In the initial argument of the Second Proof of Theorem \ref{thm:LSZ},
\begin{multline*}
\DW_{\mathrm{BCK}}\left( \begin{bmatrix}\sqrt{1+t^2}-t\\1\end{bmatrix},L_t\begin{bmatrix}\sqrt{1+t^2}-t\\1\end{bmatrix}\right)=
\\
=\DW_{\mathrm{BCK}}\left( \begin{bmatrix}\sqrt{1+t^2}-t\\1\end{bmatrix},\begin{bmatrix}\sqrt{1+t^2}+t\\-1\end{bmatrix}\right)
=\left(0, 0,\sqrt{\dfrac{t^2}{1+t^2}}\right)
\end{multline*}
gives immediately the top point of the shell
(although this choice for the testing point looks unmotivated).
From this, the ellipsoid of the tube can be recovered.

(b) A more geometric alternative in order to find out the top point is as follows.
Using Lemma \ref{lem:normcomputeD},
one can compute $\|L_t\|_2=t+\sqrt{1+t^2}$.
Then, by symmetry, the top point of $\DW_{\mathrm{pCK}}(L_t)$ of $\left(0,0,(t+\sqrt{1+t^2})^2\right)$.
Transcribed to the BCK model, the top point of $\DW_{\mathrm{BCK}}(L_t)$ is $\left(0,0,\sqrt{\frac{t^2}{1+t^2}}\right)$.

(c) If one finds the earlier initial arguments of the Second Proof of Theorem \ref{thm:LSZ}
too geometric, then it can be replaced by the following computational argument, which is a special case
of the brute force argument (in modified form):

Using projective coordinates, for $\m x=\bem z_1\\z_2\eem$, $|z_1|^2+|z_2|^2=1$, it yields
\begin{multline}
\DW_{\mathrm{BCK}}(\m x,L_t\m x)\sim
\bem
\Rea\langle L_t\m x,\m x\rangle \\
\Ima\langle L_t\m x,\m x\rangle \\
\frac{\langle L_t\m x,L_t\m x\rangle-\langle \m x,\m x\rangle}2\\
\frac{\langle L_t\m x,L_t\m x\rangle+\langle \m x,\m x\rangle}2
\eem
=
\underbrace{  \begin {bmatrix} t&0&1&0\\ 0&-t&0&0
\\ t&0&-{t}^{2}&{t}^{2}\\ \noalign{\medskip}t&0&-{t}
^{2}&{t}^{2}+1\end {bmatrix} }
_{\m S_{\mathrm{BCK}}(L_t):=}
\bem
2\Rea( z_1\bar z_2)\\
2\Ima( z_1\bar z_2)\\
|z_1|^2-|z_2|^2\\
|z_1|^2+|z_2|^2
\eem
=\\=
\underbrace{\sqrt{1+t^2}}_{S_0}
\underbrace{\bem
1&&&\\
&\frac{t}{\sqrt{1+t^2}}&&\\
&&\frac{t}{\sqrt{1+t^2}}&\\
&&&1
\eem}_{S_1}
\underbrace{R_{\mathrm{BCK}}\left(\lambda\mapsto\frac{\lambda+t+\sqrt{1+t^2}}{\lambda+t-\sqrt{1+t^2}} \right)}_{S_2}
\underbrace{\bem
2\Rea( z_1\bar z_2)\\
2\Ima( z_1\bar z_2)\\
|z_1|^2-|z_2|^2\\
|z_1|^2+|z_2|^2
\eem}_{S_3}.
\end{multline}
Then, $S_3$ ranges over the unit sphere (in projective coordinates),
$S_2$ leaves it invariant, $S_1$ shrinks it, and the scaling factor $S_0$ is irrelevant;
from this one obtain $\DW_{\mathrm{BCK}}(L_t)$ immediately.
In fact, we do not even have to follow through this computation completely.
It is sufficient to show that the rows of $\m S_{\mathrm{BCK}}(L_t)$
are $(1,1,1,-1)$-pseudoorthogonal, however not $(1,1,1,-1)$-pseudoorthonormal
but of square length ${1+t^2},t^2,t^2,-(1+t^2)$ respectively.
\qedremark
\end{remark}
Ultimately, the second (``conformal'') approach is quite
similar to  the first (``brute force'')  approach but the
computation is via familiar geometric elements, making the computations of the determinant
and eigenvalues particularly transparent.

\begin{example}\plabel{ex:fiveDW}
(a) For $L_t$ as in \eqref{nt:Lt},
\[\m Q_{\mathrm{pCK}}(L_t)=\begin{bmatrix}4t^2&&&\\ &4+4t^2&&\\ &&1&-1-2t^2\\ &&-1-2t^2&1 \end{bmatrix},\]
and
\[\m Q_{\mathrm{BCK}}(L_t)=\begin{bmatrix}4t^2&&&\\ &4+4t^2&&\\ &&4+4t^2&\\ &&&-4t^2 \end{bmatrix}.\]

(b)
The non-normal parabolic representative
$S_0$
yields
\[\m Q_{\mathrm{pCK}}(S_0)=\begin{bmatrix}1&&&\\ &1&&\\ &&1&-\frac12\\ &&-\frac12& \end{bmatrix},
\qquad\text{and}\qquad
\m Q_{\mathrm{BCK}}(S_0)=\begin{bmatrix}1&&&\\ &1&&\\ &&2&1\\ &&1& \end{bmatrix}
.\]

(c) The (normal, parabolic) zero matrix $\m 0_2$ yields
\[\m Q_{\mathrm{pCK}}(\m 0_2)=\begin{bmatrix}0&&&\\ &0&&\\ &&1&\\ &&&0 \end{bmatrix},
\qquad\text{and}\qquad
\m Q_{\mathrm{BCK}}(\m 0_2)=\begin{bmatrix}0&&&\\ &0&&\\ &&1&1\\ &&1&1 \end{bmatrix}
.\eqedexer\]
%Dealing with cases (b) and (c) directly  allows us to avoid the limiting arguments in the proof of Theorem \ref{thm:LSZ}.
%\qedexer
\end{example}
\snewpage

The third approach we consider is the ``pencil argument''.
This, of course, bases heavily on the geometric interpretation of the
qualitative elliptical range theorem.

\begin{lemma}
\plabel{lem:pencil}
Assume that $\lambda_1,\lambda_2\in\mathbb C$.

(a) The biplanar quadratic form tangent to the asymptotic boundary of the
$\mathrm{pCK}$ model at $\lambda_1$ and $\lambda_2$ is given, up to nonzero scalar multiples, by
\[(z-2x(\Rea \lambda_1)-2y(\Ima\lambda_1)+|\lambda_1|^2)(z-2x(\Rea \lambda_2)-2y(\Ima\lambda_2)+|\lambda_2|^2) .\]

(b) If $\lambda_1\neq\lambda_2$, then the quadratic form
of the infinitesimal $h$-tube connecting $\iota_{\mathrm{pCK}}(\lambda_1)$ and $\iota_{\mathrm{pCK}}(\lambda_2)$
 is given, up to nonzero scalar multiples, by
\begin{multline*}(z-2x(\Rea \lambda_1)-2y(\Ima\lambda_1)+|\lambda_1|^2)
(z-2x(\Rea \lambda_2)-2y(\Ima\lambda_2)+|\lambda_2|^2)
\\ - \left(((\Rea\lambda_1)-(\Rea\lambda_2))^2 + ((\Ima\lambda_1)-(\Ima\lambda_2))^2\right)(z-x^2-y^2)
\end{multline*}
\begin{multline*}
=(z-((\Rea \lambda_1)+(\Rea \lambda_2))x-((\Ima \lambda_1)+(\Ima \lambda_2))y
 +(\Rea \lambda_1)(\Rea \lambda_2)+(\Ima \lambda_1)(\Ima \lambda_2) )^2
\\
+(((\Ima \lambda_1)-(\Ima \lambda_2))x-((\Rea \lambda_1)-(\Rea \lambda_2))y+ (\Rea \lambda_1)(\Ima \lambda_2)-(\Ima \lambda_1)(\Rea \lambda_2) )^2
.
\end{multline*}

Here $x,y,z$ were understood as  $x_{\mathrm{pCK}},y_{\mathrm{pCK}},z_{\mathrm{pCK}}$, respectively, throughout.
\begin{proof}
(a) The equations of the  tangent planes are easy to recover, then one should multiply them.
(b) The previous quadratic form should be corrected by a multiple of $z-x^2-y^2$ so that it should fit to
$\frac{\iota_{\mathrm{pCK}}(\lambda_1)+\iota_{\mathrm{pCK}}(\lambda_2)}2$.
The correction coefficient is easy to determine.
The quadratic rearrangement is easy to check.
\end{proof}
\end{lemma}
%This allows us to have
\begin{proof}[Third Proof of Theorem \ref{thm:LSZdeg}]
Assume that $A$ have eigenvalues $\lambda_1$ and $\lambda_2$.
Then, using the identities
$|D_A|=\frac14|\lambda_1-\lambda_2|^2=\frac14 \left(((\Rea\lambda_1)-(\Rea\lambda_2))^2 + ((\Ima\lambda_1)-(\Ima\lambda_2))^2\right)$,
$\Rea\tr A= (\Rea\lambda_1)+(\Rea\lambda_2)$, etc., one can check that the matrices of the quadratic forms
in Lemma \ref{lem:pencil} (a) and (b) are the matrices
$-|D_A|2\m Q^{0}_{\mathrm{pCK}}+\m Q^{\mathrm{spec}}_{\mathrm{pCK}}(A)$
and
$|D_A|2\m Q^{0}_{\mathrm{pCK}}+\m Q^{\mathrm{spec}}_{\mathrm{pCK}}(A)$,
respectively.
\end{proof}

Now, we could prove Theorem \ref{thm:LSZ} by fitting an additional point, in the manner of Lemma \ref{lem:pencil}(b).
Instead, we proceed using a different argument:
\begin{lemma}
\plabel{lem:DWcenter}
 The center of $\DW_{\mathrm{pCK}}(A)$ is
$\left(\dfrac{\Rea \tr A}{2},\dfrac{\Ima \tr A}{2},\dfrac{\tr (A^*A)}{2}\right)$.
\begin{proof}
We have a possibly degenerate ellipsoid which is centrally symmetric.
In each coordinate the extremal values are given by the two-two
eigenvalues of the self-adjoint operators $\frac{A+A^*}2$, $\frac{A-A^*}{2\mathrm i}$, $A^*A$,
respectively.
The averages of those eigenvalues are given by $\frac12$ times the traces.
\end{proof}
\end{lemma}
\snewpage
\begin{proof}[Third Proof of Theorem \ref{thm:LSZ}.]
Let us assume that $A$ is not normal, i.~e.~the ellipsoid is proper.
We have to modify $-|D_A|2\m Q^{0}_{\mathrm{pCK}}+\m Q^{\mathrm{spec}}_{\mathrm{pCK}}(A)$
by a scalar multiple of $\m Q^{0}_{\mathrm{pCK}}$ in order to get the equation for the shell.
Thus the candidate is of shape  $\xi2\m Q^{0}_{\mathrm{pCK}}+\m Q^{\mathrm{spec}}_{\mathrm{pCK}}(A)$.
According to Lemma  \ref{lem:DWcenter}, this has the central symmetry
\begin{multline}
\xi2\m Q^{0}_{\mathrm{pCK}}+\m Q^{\mathrm{spec}}_{\mathrm{pCK}}(A)=\\=
\bem-1&&&\Rea \tr A\\&-1&&\Ima\tr A\\&&-1&\tr(A^*A)\\&&&1\eem^\top
\left(\xi2\m Q^{0}_{\mathrm{pCK}}+\m Q^{\mathrm{spec}}_{\mathrm{pCK}}(A)\right)
\bem-1&&&\Rea \tr A\\&-1&&\Ima\tr A\\&&-1&\tr(A^*A)\\&&&1\eem .\plabel{eq:censymm}
\end{multline}
In fact, geometrically, we could only claim equality up to a nonzero scalar multiple, but
the coefficients $1$ in positions $(3,3)$ imply equality.
Now, however, comparing positions $(3,4)$, we obtain $\xi=U_A$ immediately.
(Actually, writing \eqref{eq:censymm} into transpose-commutator form, it simplifies further.)

If $A$ is normal, then we already know that
$|D_A|2\m Q^{0}_{\mathrm{pCK}}+\m Q^{\mathrm{spec}}_{\mathrm{pCK}}(A)$
functions properly for the equation.

Having $\m Q_{\mathrm{pCK}}(A)$, its determinant can be computed
(but it is easier to recover from the following Third Proof of Theorem \ref{thm:eigDW}).
The coefficient $1$ in $(3,3)$ position of $\m Q_{\mathrm{pCK}}(A)$
informs us that how the sign relation for the interior should hold.
\end{proof}
\begin{proof}[Third Proof of Theorem \ref{thm:eigDW}]
In the non-parabolic case, on geometrical grounds, we know from the pencil that the eigenvalues come in pairs.
From the parameter of the singular members of the pencil ($U_A\rightsquigarrow\pm|D_A|$), we see that the eigenvalues are
as indicated.
\end{proof}
\begin{remark}
\plabel{rem:LPS}
In their characterization of the shell of $2\times2$ complex matrices,
Li, Poon, Sze \cite{LPS} indicates the center $\left(\dfrac{\Rea \tr A}{2},
\dfrac{\Ima \tr A}{2},\dfrac{\tr (A^*A)}{2}\right)$ of $\DW_{\mathrm{pCK}}(A)$ correctly;
however, the vertical diameter (but not principal axis) through the center is with endpoints
$\left(\dfrac{\Rea \tr A}{2},
\dfrac{\Ima \tr A}{2},\dfrac{\tr (A^*A)}{2}\pm \sqrt{(U_A)^2-|D_A|^2}\right)$.
%(They characterize the shell by the numerical range (including the lift of the foci)  and the length of the central diameter.)
\qedremark
\end{remark}
\snewpage

\subsection{The dual viewpoint}
\plabel{ss:dual}
~\\

However, the cleanest approach to get \eqref{eq:LSZ2} is via the dual approach.
As we deal with ellipsoids, some knowledge of projective geometry is sufficient.
There is no need for advanced algebra in order to pass between the ordinary and the dual picture
but a simple matrix inversion:

\begin{proof}[Fourth Proof of Theorem \ref{thm:LSZ}]
Assume that $A$ is non-normal.
We consider the projective homogeneous polynomial for the dual (i.~e.~tangent) surface:
\begin{align*}
&K^{\DW}_A(u,v,s,w)\equiv
\\&\equiv\det\left(u\frac{A+A^*}{2}+v\frac{A-A^*}{2\mathrm i}  +sA^*A+ w\Id_2\right)
\\&=\frac12\left(\tr \left( u\frac{A+A^*}{2}+v\frac{A-A^*}{2\mathrm i}  +sA^*A+ w\Id_2 \right)\right)^2
\\&\qquad-
\frac12\tr\left( \left( u\frac{A+A^*}{2}+v\frac{A-A^*}{2\mathrm i}  +sA^*A+ w\Id_2\right)^2\right)
\\&=\frac{|\tr A|^2-\tr(A^*A)}4(u^2+v^2)+\frac{\Rea\det A}2(u^2-v^2)+uv\Ima\det A
\\&\qquad+us\Rea((\det A)(\overline{\tr A}))
+vs\Ima((\det A)(\overline{\tr A}))
+s^2|\det A|^2
\\&\qquad+uw\Rea\tr A+vw\Ima\tr A + sw\tr(A^*A)+w^2
\\&=\begin{bmatrix}u\\ v \\ s \\w\end{bmatrix}^\top\m G_{\mathrm{pCK}}(A)
\begin{bmatrix}u \\ v \\ s \\ w\end{bmatrix}
,
\end{align*}
where
\begin{multline}
\m G_{\mathrm{pCK}}(A):=\\
\begin{bmatrix}
\frac{|\tr A|^2-\tr(A^*A)+2\Rea\det A}4&\frac{\Ima\det A}2&\frac {\Rea((\det A)(\overline{\tr A}))}2&\frac{\Rea \tr A}2\\\\
\frac{\Ima\det A}2&\frac{|\tr A|^2-\tr(A^*A)-2\Rea\det A}4&\frac {\Ima((\det A)(\overline{\tr A}))}2 &\frac{\Ima\tr A}2\\\\
\frac {\Rea((\det A)(\overline{\tr A}))}2&\frac {\Ima((\det A)(\overline{\tr A}))}2& |\det A|^2&\frac {\tr (A^*A)}2\\\\
\frac{\Rea\tr A}2&\frac{\Ima\tr A}2&\frac{\tr (A^*A)}2&1
\end{bmatrix}.
\plabel{eq:Gdef}
\end{multline}

By that, we have written down the matrix of the dual of the shell.
One finds that
\begin{equation}
\det\m G_{\mathrm{pCK}}(A)=-\frac14\left((U_A)^2-|D_A|^2\right)^2
,
\plabel{eq:Gfor1}
\end{equation}
and
\begin{equation}
\m Q_{\mathrm{pCK}}(A)=  -\left((U_A)^2-|D_A|^2\right)\left(\m G_{\mathrm{pCK}}(A) \right)^{-1}.
\plabel{eq:Gfor2}
\end{equation}
According to elementary projective geometry this implies that a matrix of the dual of the dual conic is
$\m Q_{\mathrm{pCK}}(A)$.

The normal case follows from limiting arguments.
\end{proof}
\begin{remark}
An advantage of the proof above is that it is entirely in terms of the five data.
The intermediate quantity $\m G_{\mathrm{pCK}}(A)$ is also of geometric significance.
We have
\begin{equation}
\m S_{\mathrm{pCK}}(A)
\bem1&&&\\&1&&\\&&1&\\&&&-1 \eem
\m S_{\mathrm{pCK}}(A)^{\top}=
\m G_{\mathrm{pCK}}(A)
\plabel{eq:simp3}
\end{equation}
in terms of the first proof.
\qedremark
\end{remark}
Then one can proceed with Theorem \ref{thm:eigDW} and Theorem \ref{thm:LSZdeg} as in the first proofs.

\begin{proof}[Alternative proof to Lemma \ref{lem:DWcenter}]
In the non-normal case, the projective coordinates  of the center of the shell can read off
from the last column of $\m G_{\mathrm{pCK}}(A)$ (or any non-zero multiple of the inverse of $\m Q_{\mathrm{pCK}}(A)$ ).
From this, the center is immediate.
The normal case follows by continuity.
\end{proof}

%In any case,
$\m G_{\mathrm{pCK}}(A)$ as defined by \eqref{eq:Gdef},
is a natural dual quantity to $\m Q_{\mathrm{pCK}}(A)$.
As a shorthand notation, we will use
\begin{equation}
\m G^{0}_{\mathrm{pCK}}=\bigl(\m Q^{0}_{\mathrm{pCK}}\bigr)^{-1}
\qquad\text{and}\qquad
\m G^{0}_{\mathrm{BCK}}=\bigl(\m Q^{0}_{\mathrm{BCK}}\bigr)^{-1}
.
\plabel{eq:G0def}
\end{equation}

We can define, compatibly to \eqref{eq:G0def},
\begin{equation}
\m G_{\mathrm{BCK}}(A)
=
\begin{bmatrix}
1&&&\\&1&&\\&&1&1\\&&-1&1
\end{bmatrix}^{-1}
\m G_{\mathrm{pCK}}(A)
\begin{bmatrix}
1&&&\\&1&&\\&&1&1\\&&-1&1
\end{bmatrix}^{-1,\top}
.
\plabel{eq:trans35}
\end{equation}
%(This is as in \eqref{eq:G0def}.)
With this convention, $\m G_{\mathrm{BCK}}(A) \m Q^{0}_{\mathrm{BCK}}$ and
$\m G_{\mathrm{pCK}}(A) \m Q^{0}_{\mathrm{pCK}}$ will be similar matrices% to each other
.
%For the sake of completeness, we present
\begin{theorem}
\plabel{thm:eigGDW}
The eigenvalues of
$\m G_{\mathrm{pCK}}(A) \m Q^{0}_{\mathrm{pCK}}
%$
%or
%$
\sim
\m G_{\mathrm{BCK}}(A) \m Q^{0}_{\mathrm{BCK}}$
are
\[ -\frac{U_A-|D_A|}2\quad\times2 \qquad{\text{and}}\qquad  -\frac{U_A+|D_A|}2\quad\times2. \]
In particular, we can recover $U_A : |D_A|$ from the ratio of the eigenvalues.
\begin{proof}
One can compare the characteristic polynomials of the matrices $\m G_{\mathrm{pCK}}(A) \m Q^{0}_{\mathrm{pCK}}$
and $\m G^{0}_{\mathrm{pCK}}\m Q_{\mathrm{pCK}}(A)$ (cf. Theorem \ref{thm:eigDW}) in order
find that they are appropriately scaled versions of each other.
\alter the statement follows from \eqref{eq:Gfor1} and \eqref{eq:Gfor2} generically (in the non-normal case),
then we can use the continuity of the characteristic polynomials.
\end{proof}
\end{theorem}
\snewpage
\begin{lemma}
\plabel{rem:Gtrans}
There is a similarity of matrices
\[ \m G^{0}_{\mathrm{pCK}} \m Q_{\mathrm{pCK}}(A)
\quad\sim\quad
-4\,\m G_{\mathrm{pCK}}(A) \m Q^{0}_{\mathrm{pCK}}.\]
\begin{proof}
Consider the invertible affine matrix
\[S=\bem
&-1&&\frac12\Rea\tr A +\frac12\Ima\tr A\\
1&&&-\frac12\Rea\tr A +\frac12\Ima\tr A\\
\Rea\tr A +\Ima\tr A&-\Rea\tr A +\Ima\tr A&-1&\\
&&&1
\eem
.\]
Then,
\[\m G^{0}_{\mathrm{pCK}} \m Q_{\mathrm{pCK}}(A)
=S\left(
-4\,\m G_{\mathrm{pCK}}(A) \m Q^{0}_{\mathrm{pCK}}
\right)S^{-1}.\qedhere\]
\end{proof}
\end{lemma}
\begin{remark}
This also shows the equivalence of Theorem \ref{thm:eigGDW} and Theorem \ref{thm:eigDW}.
The statement above (and also some other ones) would look more natural using
 $\mathcal Q_{\mathrm{pCK}}(A)=\frac12\m Q_{\mathrm{pCK}}(A)$
 and $\mathcal G_{\mathrm{pCK}}(A)=-2\m G_{\mathrm{pCK}}(A)$.
But, then we would lose the relatively convenient normalizations
 $\m Q_{\mathrm{pCK}}(A)_{33}=1$ and  $\m G_{\mathrm{pCK}}(A)_{44}=1$.
\qedremark
\end{remark}
\begin{disc}
\plabel{disc:DWinvpencil}
 $\m G_{\mathrm{pCK}}(A)$ still allows a decomposition to
basic and spectral parts:
\begin{multline*}
\m G_{\mathrm{pCK}}(A)=\underbrace{U_A\begin{bmatrix}
-\frac12&&&\\
&-\frac12&&\\
&&&1\\
&&1&
\end{bmatrix}}_{\m G_{\mathrm{pCK}}^{\mathrm{bas}}(A):= }+
\\
+
\underbrace{\begin{bmatrix}
\dfrac{|\tr A|^2+4\Rea\det A}8&\dfrac{\Ima\det A}2&\dfrac {\Rea((\det A)(\overline{\tr A}))}2&\dfrac{\Rea \tr A}2\\\\
\dfrac{\Ima\det A}2&\dfrac{|\tr A|^2-4\Rea\det A}8&\dfrac {\Ima((\det A)(\overline{\tr A}))}2 &\dfrac{\Ima\tr A}2\\\\
\dfrac {\Rea((\det A)(\overline{\tr A}))}2&\dfrac {\Ima((\det A)(\overline{\tr A}))}2& |\det A|^2&\dfrac{|\tr A|^2}4\\\\
\dfrac{\Rea\tr A}2&\dfrac{\Ima\tr A}2&\dfrac{|\tr A|^2}4&1
\end{bmatrix}}_{\m G_{\mathrm{pCK}}^{\mathrm{spec}}(A) := }.
\end{multline*}
It fits into a corresponding dual pencil, cf.
\[\left(\lambda  2\m Q^0_{\mathrm{pCK}} + \m Q^{\mathrm{spec}}_{\mathrm{pCK}}(A)\right)
\left(\lambda  \left(-\tfrac12\m G^0_{\mathrm{pCK}}\right)+ \m G^{\mathrm{spec}}_{\mathrm{pCK}}(A)\right)
=-(\lambda^2-|D_A|^2)\Id_4.\]
In fact, in terms of determinants
\[\det\left(\lambda  2\m Q^0_{\mathrm{pCK}} + \m Q^{\mathrm{spec}}_{\mathrm{pCK}}(A)\right)=-4(\lambda^2-|D_A|^2)^2,\]
and
\[\det\left(\lambda  \left(-\tfrac12\m G^0_{\mathrm{pCK}}\right) + \m G^{\mathrm{spec}}_{\mathrm{pCK}}(A)\right)=
-\frac14(\lambda^2-|D_A|^2)^2.\]
Indeed, these are easy to see from Theorems \ref{thm:eigDW} and \ref{thm:eigGDW}, respectively.

Anyway, with $\lambda=\pm|D_A|$ we obtain the singular points of the dual pencil,
\[\m G^{\mathrm{eig}}_{\mathrm{pCK}}(A)=|D_A|\left(-\frac12\m G^{0}_{\mathrm{pCK}}\right)
 +\m G^{\mathrm{spec}}_{\mathrm{pCK}}(A)\]
and
\[\m G^{\mathrm{ax}}_{\mathrm{pCK}}(A)=-|D_A|\left(-\frac12\m G^{0}_{\mathrm{pCK}}\right)
+\m G^{\mathrm{spec}}_{\mathrm{pCK}}(A).\]
If $\lambda_1$ and $\lambda_2$ are the eigenvalues of $A$, then
\[\m G^{\mathrm{eig}}_{\mathrm{pCK}}(A)=\frac12\left(\m x_1\m x_2^\top+\m x_2\m x_1^\top\right),\]
where
\[\m x_1=\bem\Rea\lambda_1\\\Ima\lambda_1\\|\lambda_1|^2\\1\eem
\qquad\text{and}\qquad
\m x_2=\bem\Rea\lambda_2\\\Ima\lambda_2\\|\lambda_2|^2\\1\eem
;
\]
and, similarly,
\[\m G^{\mathrm{ax}}_{\mathrm{pCK}}(A)=\m x \m x^\top+\m y\m y^\top,\]
where
\[\m x =\bem\frac12\left((\Rea\lambda_1)+(\Rea\lambda_2)\right)\\
\frac12\left((\Ima\lambda_1)+(\Ima\lambda_2)\right)\\
 (\Rea\lambda_1)(\Rea\lambda_2)+(\Ima\lambda_1)(\Ima\lambda_2)\\1\eem
\quad\text{and}\quad
\m y= \bem
\frac12\left((\Ima\lambda_2)-(\Ima\lambda_1)\right)\\
\frac12\left((\Rea\lambda_1)-(\Rea\lambda_2)\right)\\
 (\Rea\lambda_1)(\Ima\lambda_2)-(\Ima\lambda_1)(\Rea\lambda_2)\\0\eem
;
\]
We also remark that
\begin{equation}
\m Q^{\mathrm{bas}}_{\mathrm{pCK}}(A)=-4\m Q^0_{\mathrm{pCK}} \m G^{\mathrm{bas}}_{\mathrm{pCK}}(A)\m Q^0_{\mathrm{pCK}} ,
\end{equation}
and
\begin{equation}
\m Q^{\mathrm{spec}}_{\mathrm{pCK}}(A)=4\m Q^0_{\mathrm{pCK}} \m G^{\mathrm{spec}}_{\mathrm{pCK}}(A)\m Q^0_{\mathrm{pCK}} ,
\end{equation}
\begin{equation}
\m Q^{\mathrm{eig}}_{\mathrm{pCK}}(A)=4\m Q^0_{\mathrm{pCK}} \m G^{\mathrm{eig}}_{\mathrm{pCK}}(A)\m Q^0_{\mathrm{pCK}} ,
\plabel{eq:lueig}
\end{equation}
\begin{equation}
\m Q^{\mathrm{ax}}_{\mathrm{pCK}}(A)=4\m Q^0_{\mathrm{pCK}} \m G^{\mathrm{ax}}_{\mathrm{pCK}}(A)\m Q^0_{\mathrm{pCK}} .
\end{equation}

One can see   that not only $\m Q_{\mathrm{pCK}}(A)$ and  $\m G_{\mathrm{pCK}}(A)$
 are ``linear rearrangements'' of each other, but also the corresponding pencils.
Thus, ultimately, $\m Q_{\mathrm{pCK}}(A)$ and  $\m G_{\mathrm{pCK}}(A)$ are quite similar to each other.
They lay in pencils generated by
$\m Q^0_{\mathrm{pCK}}$ and, say, $\m Q^{\mathrm{spec}}_{\mathrm{pCK}}(A)$
and
$\m G^0_{\mathrm{pCK}}$ and, say, $\m G^{\mathrm{spec}}_{\mathrm{pCK}}(A)$,
respectively.
Although $\m Q_{\mathrm{pCK}}(A)$ and  $\m G_{\mathrm{pCK}}(A)$ are not,
but the pencils are related to each other by the pole-polar correspondence
  with respect to $\m Q^0_{\mathrm{pCK}}$ and $\m G^0_{\mathrm{pCK}}$.
\end{disc}
\snewpage
\subsection{Decompositions of the quadratic forms}

\begin{cor}
\plabel{cor:fiveDW}
For $2\times2$ complex matrices, the five data determines the Davis--Wielandt shell, and vice versa.
\begin{proof}
The five data determines the matrix up to unitary conjugation, which, of course,
determines shell, as unitary conjugation leaves it invariant.

Conversely, if $\DW_{\mathrm{pCK}}(A)$ is given, then from the asymptotic points, we can recover the eigenvalues,
and from the central vertical diameter $2\sqrt{(U_A)^2-|D_A|^2}$, the value $\tr(A^*A)$ can also be recovered.
\alter if the matrix $A$ is normal, then the endpoints determine the eigenvalues,
which is sufficient to recover the diagonal matrix up to unitary conjugation.
If the matrix $A$ is non-normal and $\m M$ is a matrix of the quadric in the pCK model, then
$\m Q_{\mathrm{pCK}}(A)=\m M/\m M_{33}$, and from this matrix
the five data can immediately be read off.
In fact, the preceding argument also works if $A$ is normal but $\m M$
is required to be an infinitesimal tube or a double tangent plane.

%The previous argument also works with $\m G_{\mathrm{pCK}}(A)$ where
%the normalized dual is obtained by  $\m G_{\mathrm{pCK}}(A)=\m M'/\m M'_{44}$.

(Remark: Cf. Theorem \ref{cor:fiveW}.)
\end{proof}
\end{cor}

We define the `core matrices'
\begin{equation}
\m Q_{\mathrm C}(A)=
\bem
\tr(A^*A)-(\Rea \tr A)^2+2\Rea\det A
&
-(\Rea \tr A)(\Ima\tr A)+2\Ima\det A
\\
-(\Rea \tr A)(\Ima\tr A)+2\Ima\det A
&
\tr(A^*A)-(\Ima \tr A)^2-2\Rea\det A
\eem,
\plabel{eq:Qcore}
\end{equation}
and
\begin{equation}
\m G_{\mathrm C}(A)=-\frac14
\bem
\tr(A^*A)-(\Ima \tr A)^2-2\Rea\det A
&
(\Rea \tr A)(\Ima\tr A)-2\Ima\det A
\\
(\Rea \tr A)(\Ima\tr A)-2\Ima\det A
&
\tr(A^*A)-(\Rea \tr A)^2+2\Rea\det A
\eem.
\plabel{eq:Gcore}
\end{equation}
One can see that
\[\m G_{\mathrm C}(A)=-\frac14\adj\m Q_{\mathrm C}(A)
\qquad\text{and}\qquad
\m Q_{\mathrm C}(A)=-4\adj\m G_{\mathrm C}(A).\]
%\begin{commentx}
(Remark:
For a $2\times2$ symmetric matrix $\m M$, $\m M$ is similar to $\adj\m M$.
In fact, for a $2\times2$ symmetric matrix $\m T$, one has
$\adj\m T=\begin{bsmallmatrix}&-1\\1&\end{bsmallmatrix}
\m T\begin{bsmallmatrix}&1\\-1&\end{bsmallmatrix}$.)
%\end{commentx}

We set
\[\m B_{\mathrm{pCK}}(A)=
\bem
1&&&\frac12\Rea\tr A\\
&1&&\frac12\Ima\tr A\\
\Rea\tr A&\Ima\tr A&1&\frac12\tr(A^*A)\\
&&&1
\eem,\]
which is an invertible affine matrix.
\begin{theorem}\plabel{thm:Qcu}
(a)
\[\m Q_{\mathrm{pCK}}(A)=
\m B_{\mathrm{pCK}}(A)^{-1,\top}
\underbrace{\bem \m Q_{\mathrm C}(A)&&\\&1&\\&&-((U_A)^2-|D_A|^2)\eem}_{Q_1}
\m B_{\mathrm{pCK}}(A)^{-1}
.
\]
(b) The eigenvalues of the $(2|1|1)$ block matrix $Q_1$ are
\[2(U_D-|D_A|),\quad 2(U_D+|D_A|),\quad |\quad 1,\quad|\quad -((U_A)^2-|D_A|^2);\]
the matrix is diagonalizable.
\end{theorem}

\begin{theorem}\plabel{thm:Gcu}
(a)
\[\m G_{\mathrm{pCK}}(A)=
\m B_{\mathrm{pCK}}(A)^{}
\underbrace{\bem \m G_{\mathrm C}(A)&&\\&-((U_A)^2-|D_A|^2)&\\&&1\eem}_{G_1}
\m B_{\mathrm{pCK}}(A)^{\top}
.
\]

(b) The eigenvalues of the $(2|1|1)$ block matrix $G_1$ are
\[-\frac12(U_D-|D_A|),\quad -\frac12(U_D+|D_A|),\quad |\quad -((U_A)^2-|D_A|^2),\quad|\quad1;\]
the matrix is diagonalizable.
\begin{proof}[Proofs]
Straightforward computations.
\end{proof}
\end{theorem}

Then it easy to see that
\[\rank \m Q_{\mathrm{pCK}} (A) =\rank \m G_{\mathrm{pCK}} (A) =
\begin{cases}
4&\text{if $A$ is non-normal,}\\
2&\text{if $A$ is normal non-parabolic,}\\
1&\text{if $A$ is normal parabolic.}
\end{cases} \]

However, we may also know this from the pencil picture.
Another possibility is to look up the canonical representatives from Example \ref{ex:fiveDW}
for $\rank \m Q_{\mathrm{pCK}}(A)$.
(M\"obius transformations may introduce only a scaling by nonzero scalar multiplier; this is also valid in the degenerate cases.
The pCK and BCK representatives can be used to determine the rank equally.)
Then Lemma \ref{rem:Gtrans} shows $\rank \m Q_{\mathrm{pCK}} (A) =\rank \m G_{\mathrm{pCK}} (A)$.

(Taking the quotient of symmetric matrix by an indefinite symmetric matrix may lead to non-trivial Jordan blocks.
Consequently, Theorems \ref{thm:eigDW} and  \ref{thm:eigGDW} cannot be used to find the ranks.
Indeed, for scalar matrices the eigenvalues there are all zero, while
$\m Q_{\mathrm{BCK}}(A)$ and $\m G_{\mathrm{BCK}}(A)$ have ranks $\geq1$ in any case, due to the entry `$1$'.)

\snewpage
\subsection{The transformation properties of the matrix}

\begin{defin}\plabel{def:LSZtransform}
Suppose that $f$ is a M\"obius transformation $f:\lambda\mapsto \frac{a\lambda+b}{c\lambda+d}$, $a,b,c,d\in\mathbb C$, $ad-bc\neq0$,
and $A$ is $2\times2$ complex matrix. Then we set
\begin{equation}
\mathcal C(f,A)= \frac1{|ad-bc|^2}
\begin{bmatrix}-\Rea(\bar cd)\\-\Ima(\bar cd) \\|d|^2 \\ |c|^2\end{bmatrix}^\top
\m Q_{\mathrm{pCK}}(A)
\begin{bmatrix}-\Rea(\bar cd)\\-\Ima(\bar cd) \\|d|^2 \\ |c|^2\end{bmatrix}
\plabel{def:qtran}
\end{equation}
(which is well-defined).
Actually, this can be rewritten as
\begin{equation}
\mathcal C(f,A)= \frac1{|ad-bc|^2}
\begin{bmatrix}2\Rea(\bar cd)\\2\Ima(\bar cd) \\|c|^2 \\ |d|^2\end{bmatrix}^\top
\m G_{\mathrm{pCK}}(A)
\begin{bmatrix}2\Rea(\bar cd)\\2\Ima(\bar cd) \\|c|^2 \\ |d|^2\end{bmatrix}.
\plabel{def:gtran}
\end{equation}
\end{defin}

\begin{lemma}\plabel{lem:LSZtransform}
Suppose that $f$ is a M\"obius transformation $f:\lambda\mapsto \frac{a\lambda+b}{c\lambda+d}$, $a,b,c,d\in\mathbb C$, $ad-bc\neq0$,
and $A$ is $2\times2$ complex matrix such that $-\frac dc$ is not an eigenvalue of $A$.

Let $R_{\mathrm{pCK}}(f)$ be a matrix of determinant $1$ representing the projective action of $f$ in $\mathrm{pCK}$.
Then, the following transformation rules hold:
\begin{equation}
\m Q_{\mathrm{pCK}}(f(A)) \cdot \mathcal C(f,A)
=
\left(R_{\mathrm{pCK}}(f)\right)^{-1,\top}\m Q_{\mathrm{pCK}}(A) \left(R_{\mathrm{pCK}}(f)\right)^{-1}
;\plabel{eq:LSZtrans1}
\end{equation}
\begin{equation}
\m G^{0}_{\mathrm{pCK}}  \m Q_{\mathrm{pCK}}(f(A)) \cdot  \mathcal C(f,A)
=\left(R_{\mathrm{pCK}}(f)\right)
\m G^{0}_{\mathrm{pCK}} \m Q_{\mathrm{pCK}}(A)
\left(R_{\mathrm{pCK}}(f)\right)^{-1}
;\plabel{eq:LSZtrans2}
\end{equation}
\begin{equation}
\m G_{\mathrm{pCK}}(f(A)) \cdot \mathcal C(f,A)
=
\left(R_{\mathrm{pCK}}(f)\right)\m G_{\mathrm{pCK}}(A) \left(R_{\mathrm{pCK}}(f)\right)^\top
;\plabel{eq:LSZtrans1g}
\end{equation}
\begin{equation}
\m G_{\mathrm{pCK}}(f(A)) \m Q^{0}_{\mathrm{pCK}}\cdot  \mathcal C(f,A)
=\left(R_{\mathrm{pCK}}(f)\right)
\m G_{\mathrm{pCK}}(A)\m Q^{0}_{\mathrm{pCK}}
\left(R_{\mathrm{pCK}}(f)\right)^{-1}
.\plabel{eq:LSZtrans2g}
\end{equation}

\begin{proof}
By Theorem  \ref{thm:LSZ} and M\"obius invariance, we already know that $\m Q_{\mathrm{pCK}}(A)$ transforms naturally.
The scaling factor in \eqref{eq:LSZtrans1} is easy to recover using the general observation $\m Q_{\mathrm{pCK}}(B)_{33}=1$.
Indeed, \eqref{def:qtran} is set up according to this.
Equivalence to \eqref{eq:LSZtrans2} follows from
$\m Q^{0}_{\mathrm{pCK}}=\left(R_{\mathrm{pCK}}(f)\right)^{-1,\top} \m Q^{0}_{\mathrm{pCK}}\left(R_{\mathrm{pCK}}(f)\right)^{-1} $ (inverted and multiplying with).
Similar argument applies regarding $\m G_{\mathrm{pCK}}(A)$.
The scaling factor in \eqref{eq:LSZtrans1g} is easy to recover using the general observation $\m G_{\mathrm{pCK}}(B)_{44}=1$.
This is in correspondence to \eqref{def:gtran}.
Equivalence to \eqref{eq:LSZtrans2g} is as previously.
\end{proof}
\end{lemma}
\begin{proof}[Alternative proof for Lemma \ref{lem:UDconf}]
We  see that axial degeneration (making $U_A\rightsquigarrow |D_A|$ while keeping the spectral part fixed)
commutes with conformal transformations.
\end{proof}
%\snewpage
\begin{lemma}\plabel{lem:sigLSZtransform}
Suppose that $f$ is a M\"obius transformation $f:\lambda\mapsto \frac{a\lambda+b}{c\lambda+d}$, $a,b,c,d\in\mathbb C$, $ad-bc\neq0$,
and $A$ is $2\times2$ complex matrix.

(a) If $-\frac dc$ is not an eigenvalue of $A$, then
\[\mathcal C(f,A)>0.\]

(b) If $-\frac dc$ is an eigenvalue of $A$, then
\[\mathcal C(f,A)=0.\]
\begin{proof}
(a) From the transformation formula \eqref{eq:LSZtrans1} of Lemma \ref{lem:LSZtransform},
$\mathcal C(f,A)\in\mathbb R\setminus\{0\}$ follows (cf. $\m Q_{\mathrm{pCK}}(A)_{33} =1$).
Due to the connectedness of complex  M\"obius transformations, the result is $>0$.
(Another argument is that the near most asymptotic points the associated quadratic form is positive.)

(b) is sufficient to check for matrices of shape $\begin{bmatrix}*&*\\&-\frac dc\end{bmatrix}$, which is straightforward.
\end{proof}
\end{lemma}
\begin{lemma}\plabel{lem:Ucomp}
It is true that
\[U_A=\frac18\tr\left(\bigl(\m Q^{0}_{\mathrm{pCK}}\bigr)^{-1}\m Q_{\mathrm{pCK}}(A)\right).\]
\begin{proof}
This follows from Theorem \ref{thm:eigDW}, or by direct computation.
\end{proof}
\end{lemma}

\begin{lemma}\plabel{lem:UDtransform}
Suppose that $f$ is a M\"obius transformation $f:\lambda\mapsto \frac{a\lambda+b}{c\lambda+d}$, $a,b,c,d\in\mathbb C$, $ad-bc\neq0$,
and $A$ is $2\times2$ complex matrix such that $-\frac dc$ is not an eigenvalue of $A$.

(a) Then,
\[U_{f(A)} \cdot \mathcal C(f,A) = U_A \]
holds.

(b) Furthermore, the  transformation formula
\[|D_{f(A)}| \cdot \mathcal C(f,A) = |D_A| \]
also holds.
\begin{proof}
Lemma \ref{lem:LSZtransform} and Lemma \ref{lem:Ucomp} imply (a).
Lemma \ref{lem:UDconf} and (a) imply (b) generically (when  $U_A>0$).
This extends by continuity.
\end{proof}
\end{lemma}
\begin{commentx}
\begin{proof}[Alternative proofs for
Lemma \ref{lem:LSZtransform}, Lemma \ref{lem:sigLSZtransform}, Lemma \ref{lem:Ucomp} ,Lemma \ref{lem:UDtransform}]
All statements can be checked by direct computations.
This is easy for Lemma \ref{lem:Ucomp}, but tedious otherwise.
\end{proof}
\end{commentx}
\begin{cor}\plabel{cor:welltrans}
(a) In the  $U_A\neq0$ case, the matrix is given by
\begin{equation}
\widetilde {\m Q}_{\mathrm{pCK}}(A) = \frac1{U_A}\m  Q_{\mathrm{pCK}}(A)
\plabel{nt:perfuu}
\end{equation}
has the transformation property
\[\widetilde {\m Q}_{\mathrm{pCK}}(f(A))  =\left(R_{\mathrm{pCK}}(f)\right)^{-1,\top}\widetilde{\m Q}_{\mathrm{pCK}}(A) \left(R_{\mathrm{pCK}}(f)\right)^{-1} .\]
whenever $f(A)$ makes sense.

(b) In the  $|D_A|\neq0$ case, similar statement holds with `$U_A$' replaced by `$|D_A|$'.
\begin{proof}
This is an immediate consequence of Lemma \ref{lem:LSZtransform} and Lemma \ref{lem:UDtransform}.
\end{proof}
\end{cor}
\begin{remark}
In the argument above we do not need to identify the value of $\mathcal C(f,A)$ exactly
but we can start with considering it as the scaling factor of $U_A$.
\qedremark
\end{remark}
By this corollary, we have quantities perfectly well-transforming under complex conformal transformation.
However, in (a) we loose the normal parabolic case (that is scalar matrices);
and in  (b) we loose the whole parabolic case (two equal eigenvalues).
\\

\subsection{The geometry of the shell}
~\\

Returning to the situation of Theorem \ref{thm:DWconc}, we can describe the picture in even more geometric terms.

\begin{theorem}\plabel{thm:tuberad}
Suppose that $A$ is a linear operator on a $2$-dimensional complex Hilbert space with two distinct eigenvalues (i.~e.~it is non-parabolic).
Then the hyperbolic radius of the (possibly degenerate) $h$-tube $\DW_*(A)$ is
\[\radius_* \DW_*(A)=\frac12\arcosh {\frac{U_A}{\left|D_A \right| }}
\begin{commentx}
=\artanh\sqrt{\frac{U_A -\left|D_A \right| }{U_A+\left|D_A \right|}}
\end{commentx}
.\]
(In the parabolic non-normal case the corresponding value can be considered to be $+\infty$.
In the parabolic normal (i.~e.~scalar matrix) case it can be taken as $0$.)
\qed
\begin{commentx}
\begin{proof}
Let us continue the argument of the Second Proof of Theorem \ref{thm:LSZ}.
After some displacement we obtained $L$. Then,
\[
\radius_* \DW_*(A)=\radius_* \DW_*(L) =
\mathrm d^{\mathrm{BCK}}\left((0,0,0),\left(0,0,\sqrt{\frac{t^2}{1+t^2}}\right) \right)
%\\&=\artanh \sqrt{\frac{t^2}{1+t^2}}
=\frac12\arcosh\left(1+2t^2\right)
\]
(cf. (\ref{eq:distex}/2)). Substituting \eqref{eq:tubet}, we obtain the result.
\end{proof}
\end{commentx}
\end{theorem}
Together with the description of asymptotic points, the result above provides a characterization
of the Davis--Wielandt shell in the complex non-parabolic (i.~e.~possibly degerate $h$-tube) case.
Also note that in obtaining Theorem \ref{thm:tuberad} we have used very little information about the shell;
merely conformal invariance and the description of $\DW^{\mathbb R}(L_t)$ (the initial argument
of the second proof of Theorem \ref{thm:LSZ}) was used.

As a consequence, we see (again) that ${U_A}/{|D_A|}$ is an invariant under complex
M\"obius transformations (as long as the spectrum allows them).
Moreover, we cannot really expect to have other (unitarily invariant) ones (except its functions in the non-parabolic case and the normal / non-normal distinction in the parabolic case).
Indeed, what makes tubes apart isometrically in hyperbolic geometry is just their radius,
and horospheres all are isometric to each other.

Theorem \ref{thm:tuberad} can be considered as the elliptical range theorem for the shell (with two distinct asymptotical foci).
However, the case of horospheres cannot be characterized by strictly focal and metric data.
The horospheric case is very special, but it requires a characterization based on other principles.

Let
\begin{equation}
Q^{\mathrm{spec},A}_{\mathrm{pCK}}(x_{\mathrm{pCK}},y_{\mathrm{pCK}},z_{\mathrm{pCK}} )
=
\begin{bmatrix}x_{\mathrm{pCK}}\\ y_{\mathrm{pCK}}\\ z_{\mathrm{pCK}}\\1\end{bmatrix}^\top
\m Q_{\mathrm{pCK}}^{\mathrm{spec}}(A)
\begin{bmatrix}x_{\mathrm{pCK}}\\ y_{\mathrm{pCK}}\\ z_{\mathrm{pCK}}\\1\end{bmatrix}
\plabel{eq:QpCKPspecform}
\end{equation}
denote the quadratic expression associated the to the matrix
$\m Q_{\mathrm{pCK}}^{\mathrm{spec}}(A)$. We also let
\begin{align}
Q^\star_{\mathrm{pCK}}(x_{\mathrm{pCK}},y_{\mathrm{pCK}},z_{\mathrm{pCK}} )
%&=(x_{\mathrm{pCK}})^2+(y_{\mathrm{pCK}})^2-z_{\mathrm{pCK}}
%\\\notag
&=
\begin{bmatrix}x_{\mathrm{pCK}}\\ y_{\mathrm{pCK}}\\ z_{\mathrm{pCK}}\\1\end{bmatrix}^\top
%\begin{bmatrix}-1&&&\\&-1&&\\&&&\frac12\\&&\frac12&\end{bmatrix}
\m Q_{\mathrm{pCK}}^{0}
\begin{bmatrix}x_{\mathrm{pCK}}\\ y_{\mathrm{pCK}}\\ z_{\mathrm{pCK}}\\1\end{bmatrix}
.\plabel{eq:QpCKP0form}
\end{align}

The same thing can be played down, with respect to
$\m Q_{\mathrm{BCK}}^{\mathrm{spec}}(A)$  and $\m Q_{\mathrm{BCK}}^{0}$
and the coordinates $x_{\mathrm{BCK}},y_{\mathrm{BCK}},z_{\mathrm{BCK}}$.
Now, taking \eqref{eq:trans30}/\eqref{eq:trans3}/\eqref{eq:biQ0} and \eqref{eq:can2} into account, we find that
$Q^{\mathrm{spec},A}_{\mathrm{pCK}}(x_{\mathrm{pCK}},y_{\mathrm{pCK}},z_{\mathrm{BCK}})
=\dfrac1{(1-z_{\mathrm{BCK}})^2}Q^{\mathrm{spec},A}_{\mathrm{BCK}}(x_{\mathrm{BCK}},y_{\mathrm{BCK}},z_{\mathrm{BCK}})$;
and similar formula holds for $Q^0$.
Thus, the ``quadratic forms'' are not equal, but for the quotients,
\[\dfrac{Q^{\mathrm{spec},A}_{\mathrm{pCK}}(x_{\mathrm{pCK}},y_{\mathrm{pCK}},z_{\mathrm{pCK}})}{
-2Q^\star_{\mathrm{pCK}}(x_{\mathrm{pCK}},y_{\mathrm{pCK}},z_{\mathrm{pCK}} )}
=
\dfrac{Q^{\mathrm{spec},A}_{\mathrm{BCK}}(x_{\mathrm{BCK}},y_{\mathrm{BCK}},z_{\mathrm{BCK}})}{
-2Q^\star_{\mathrm{BCK}}(x_{\mathrm{BCK}},y_{\mathrm{BCK}},z_{\mathrm{BCK}} )}
\]
holds. (The division by $-2$ is merely a convenient normalization factor).
These quotients are undefined for asymptotic points, but they are well-defined for interior points.

\begin{lemma}
For $(x_{\mathrm{pCK}},y_{\mathrm{pCK}},z_{\mathrm{pCK}} )\in H^3_{\mathrm{pCK}}$,
\begin{equation}
\dfrac{Q^{\mathrm{spec},A}_{\mathrm{pCK}}(x_{\mathrm{pCK}},y_{\mathrm{pCK}},z_{\mathrm{pCK}})}{
-2Q^\star_{\mathrm{pCK}}(x_{\mathrm{pCK}},y_{\mathrm{pCK}},z_{\mathrm{pCK}} )}\geq |D_A|;
\plabel{eq:quod}
\end{equation}
and
\begin{equation}
\dfrac{Q^{\mathrm{spec},A}_{\mathrm{pCK}}(x_{\mathrm{pCK}},y_{\mathrm{pCK}},z_{\mathrm{pCK}})}{
-2Q^\star_{\mathrm{pCK}}(x_{\mathrm{pCK}},y_{\mathrm{pCK}},z_{\mathrm{pCK}} )}>0.
\plabel{eq:quot}
\end{equation}

\begin{proof}
According Theorem \ref{thm:LSZdeg}(b), \eqref{eq:quod} holds.
\eqref{eq:quot} is trivial in the complex non-parabolic case $(|D_A|>0)$, and
follows from the situation regarding the tangent double plane in the parabolic case.
\end{proof}
\end{lemma}
%The most natural geometric quantities are given as follows.
\begin{lemma}
\plabel{lem:natquad}
Let
\begin{equation}
\mathcal R_{\mathrm{pCK}}^A(x_{\mathrm{pCK}},y_{\mathrm{pCK}},z_{\mathrm{pCK}})=
\frac1{U_A}\cdot
\dfrac{Q^{\mathrm{spec},A}_{\mathrm{pCK}}(x_{\mathrm{pCK}},y_{\mathrm{pCK}},z_{\mathrm{pCK}})}{
-2Q^\star_{\mathrm{pCK}}(x_{\mathrm{pCK}},y_{\mathrm{pCK}},z_{\mathrm{pCK}} )}
.
\plabel{eq:natquad}
\end{equation}
This quantity has the transformation property
\[\mathcal R_{\mathrm{pCK}}^{f(A)}(f^{\mathrm{pCK}}(\boldsymbol x_{\mathrm{pCK}}))
=
\mathcal R_{\mathrm{pCK}}^A(\boldsymbol x_{\mathrm{pCK}}),\]
whenever $f(A)$ is well-defined.
I. e., it is a truly geometrical quantity.
\begin{proof}
This is an immediate consequence of Corollary \ref{cor:welltrans}.
\end{proof}
\end{lemma}
Let $Q^A_{\mathrm{pCK}}(x_{\mathrm{pCK}},y_{\mathrm{pCK}},z_{\mathrm{pCK}} )$
be the   form associated to $\m Q_{\mathrm{pCK}}(A)$
in analogy to \eqref{eq:QpCKPspecform}.
We can note that the quantity in the Lemma \ref{lem:natquad} can be replaced by
\[\frac1{U_A}\cdot
\dfrac{Q^{A}_{\mathrm{pCK}}(x_{\mathrm{pCK}},y_{\mathrm{pCK}},z_{\mathrm{pCK}})}{
-2Q^\star_{\mathrm{pCK}}(x_{\mathrm{pCK}},y_{\mathrm{pCK}},z_{\mathrm{pCK}} )}\]
(which is \eqref{eq:natquad} and $-1$  added), or by
\begin{equation}
\frac1{|D_A|}\cdot
\dfrac{Q^{\mathrm{spec},A}_{\mathrm{pCK}}(x_{\mathrm{pCK}},y_{\mathrm{pCK}},z_{\mathrm{pCK}})}{
-2Q^\star_{\mathrm{pCK}}(x_{\mathrm{pCK}},y_{\mathrm{pCK}},z_{\mathrm{pCK}} )}
\plabel{eq:confdd}
\end{equation}
(which is \eqref{eq:natquad}  multiplied by $U_A/|D_A|$).
These are equally well-behaved under complex M\"obius transformations; moreover,
pCK can be replaced by BCK.

\begin{theorem}
\plabel{thm:axdist}
Let $A$ be a complex $2\times2$ matrix.
Then the $h$-distance between the point
$(x_{\mathrm{pCK}},y_{\mathrm{pCK}},z_{\mathrm{pCK}})\in H^3_{\mathrm{pCK}} $ and the axis of
$\DW_{\mathrm{pCK}}(A)$ is
\begin{align}
\mathrm d^{\mathrm{pCK}}&((x_{\mathrm{pCK}},y_{\mathrm{pCK}},z_{\mathrm{pCK}}),\axis_{\mathrm{pCK}} \DW_{\mathrm{pCK}}(A))=\plabel{eq:axdist}\\ \notag
&=\frac12\arcosh\left(\frac1{|D_A|}\cdot{
\dfrac{Q^{\mathrm{spec},A}_{\mathrm{pCK}}(x_{\mathrm{pCK}},y_{\mathrm{pCK}},z_{\mathrm{pCK}})}{
-2Q^\star_{\mathrm{pCK}}(x_{\mathrm{pCK}},y_{\mathrm{pCK}},z_{\mathrm{pCK}} )}}\right)
.
\end{align}
\begin{commentx}
\begin{align}\notag\phantom{aaaa}
&=\artanh\sqrt{\frac{
\dfrac{Q^{\mathrm{spec},A}_{\mathrm{pCK}}(x_{\mathrm{pCK}},y_{\mathrm{pCK}},z_{\mathrm{pCK}})}{
-2Q^\star_{\mathrm{pCK}}(x_{\mathrm{pCK}},y_{\mathrm{pCK}},z_{\mathrm{pCK}} )}
-|D_A|}{\dfrac{Q^{\mathrm{spec},A}_{\mathrm{pCK}}(x_{\mathrm{pCK}},y_{\mathrm{pCK}},z_{\mathrm{pCK}})}{
-2Q^\star_{\mathrm{pCK}}(x_{\mathrm{pCK}},y_{\mathrm{pCK}},z_{\mathrm{pCK}} )}+  |D_A|}}
.
\end{align}
\end{commentx}
(In the parabolic case the axis is just the asymptotic point, and the value above is $+\infty$.)

\begin{proof}
Consider \eqref{eq:confdd}.
Due to geometricity, it is sufficient to consider $A=L_t$ as in \eqref{nt:Lt}.
Even so, we will use the $\mathrm{BCK}$ model.
Then we find
\[
\frac1{|D_{L_t}|}\cdot
\dfrac{Q^{\mathrm{spec},L_t}_{\mathrm{BCK}}(x_{\mathrm{BCK}},y_{\mathrm{BCK}},z_{\mathrm{BCK}})}{
-2Q^\star_{\mathrm{BCK}}(x_{\mathrm{BCK}},y_{\mathrm{BCK}},z_{\mathrm{BCK}} )}
\equiv \frac11\cdot\frac22\cdot \frac{-(x_{\mathrm{BCK}})^2+(y_{\mathrm{BCK}})^2+(z_{\mathrm{BCK}})^2+1}{-(x_{\mathrm{BCK}})^2-(y_{\mathrm{BCK}})^2-(z_{\mathrm{BCK}})^2+1}
.\]

On the other hand, the distance of $(x_{\mathrm{BCK}},y_{\mathrm{BCK}},z_{\mathrm{BCK}})$ of from the axis of the tube is
\[\mathrm d^{\mathrm{BCK}}\left((\tilde x_{\mathrm{BCK}},\tilde y_{\mathrm{BCK}},\tilde z_{\mathrm{BCK}}),(\tilde x_{\mathrm{BCK}},0,0)  \right)
=\frac12\arcosh\left(\frac{1-(\tilde x_{\mathrm{BCK}})^2+(\tilde y_{\mathrm{BCK}})^2+(\tilde z_{\mathrm{BCK}})^2 }{
1-(\tilde x_{\mathrm{BCK}})^2-(\tilde y_{\mathrm{BCK}})^2-(\tilde z_{\mathrm{BCK}})^2}\right). \]
Comparing this to the RHS of the first expression, this yields the statement in the non-parabolic case.
The remark concerning the parabolic case follows from limiting arguments.
\end{proof}
\end{theorem}

As a corollary, we obtain
\begin{theorem}\plabel{thm:distquot}
Consider
\begin{equation}
\frac{1}{U_A}\cdot\dfrac{Q^{A}_{\mathrm{pCK}}(x_{\mathrm{pCK}},y_{\mathrm{pCK}},z_{\mathrm{pCK}})}{
-2Q^\star_{\mathrm{pCK}}(x_{\mathrm{pCK}},y_{\mathrm{pCK}},z_{\mathrm{pCK}} )}.
\plabel{eq:distquot}
\end{equation}
If $A$ is non-parabolic, then \eqref{eq:distquot}
\begin{equation*}
=\frac{
\cosh\left(2\cdot\mathrm d^{\mathrm{pCK}}\left((x_{\mathrm{pCK}},y_{\mathrm{pCK}},z_{\mathrm{pCK}}),\axis_{\mathrm{pCK}}  \DW_{\mathrm{pCK}}(A)\right)\right)
}{\cosh\left(2\cdot
\radius_{\mathrm{pCK}}  \DW_{\mathrm{pCK}}(A))
\right) }-1.
\end{equation*}
If $A$ is parabolic, then \eqref{eq:distquot}
\begin{equation*}
=\exp\left(2\cdot\overleftarrow{\mathrm{dis}}^{\mathrm{pCK}}\left((x_{\mathrm{pCK}},y_{\mathrm{pCK}},z_{\mathrm{pCK}}), \DW_{\mathrm{pCK}}(A)\right)\right)-1
,
\end{equation*}
where `$\overleftarrow{\mathrm{dis}}$' means oriented distance, negative inside the horosphere $\DW_{\mathrm{pCK}}(A)$,
 positive outside the horosphere $\DW_{\mathrm{pCK}}(A)$.
\begin{proof}
In the non-parabolic case, this is just Theorem \ref{thm:tuberad} and Theorem \ref{thm:axdist} combined.
In the parabolic case, the statement follows from limiting arguments via the formula
\[\lim_{u\rightarrow+\infty}\frac{\cosh 2(d+u)}{\cosh 2u}=\exp 2d.\qedhere\]
\end{proof}
\end{theorem}

\begin{theorem}
\plabel{thm:tubedist}
Let $A$ be a complex $2\times2$ matrix.
Then the signed $h$-distance between $(x_{\mathrm{pCK}},y_{\mathrm{pCK}},z_{\mathrm{pCK}})\in H^3_{\mathrm{pCK}} $ and $\DW_{\mathrm{pCK}}(A)$ is
\begin{multline*}\overleftarrow{\mathrm{dis}}_{\mathrm{pCK}}((x_{\mathrm{pCK}},y_{\mathrm{pCK}},z_{\mathrm{pCK}}) ,
\DW_{\mathrm{pCK}}(A)
)=
\\
=\artanh\frac{\sqrt{\left(\dfrac{Q^{\mathrm{spec},A}_{\mathrm{pCK}}(x_{\mathrm{pCK}},y_{\mathrm{pCK}},z_{\mathrm{pCK}})}{-2Q^\star_{\mathrm{pCK}}(x_{\mathrm{pCK}},y_{\mathrm{pCK}},z_{\mathrm{pCK}} )}\right)^2-|D_A|^2}-\sqrt{\left(U_A\right)^2-|D_A|^2}}{
\dfrac{Q^{\mathrm{spec},A}_{\mathrm{pCK}}(x_{\mathrm{pCK}},y_{\mathrm{pCK}},z_{\mathrm{pCK}})}{-2Q^\star_{\mathrm{pCK}}(x_{\mathrm{pCK}},y_{\mathrm{pCK}},z_{\mathrm{pCK}} )} +U_A},
\end{multline*}
understood so that the value is negative if the point is inside the ellipsoid, and the value is positive if it is outside.
\begin{commentx}
\proofremark{The expression given can also be written as
\[=\artanh\frac{
\dfrac{Q^{\mathrm{spec},A}_{\mathrm{pCK}}(x_{\mathrm{pCK}},y_{\mathrm{pCK}},z_{\mathrm{pCK}})}{
-2Q^\star_{\mathrm{pCK}}(x_{\mathrm{pCK}},y_{\mathrm{pCK}},z_{\mathrm{pCK}} )} -U_A}
{\sqrt{\left(\dfrac{Q^{\mathrm{spec},A}_{\mathrm{pCK}}(x_{\mathrm{pCK}},y_{\mathrm{pCK}},z_{\mathrm{pCK}})}{
-2Q^\star_{\mathrm{pCK}}(x_{\mathrm{pCK}},y_{\mathrm{pCK}},z_{\mathrm{pCK}} )}\right)^2-|D_A|^2}+\sqrt{\left(U_A\right)^2-|D_A|^2}},\]
but in this a case $\frac00=0$ should be taken when $A$ is normal $(U_A=|D_A|)$ and the investigated point is on the axis.
(If $A$ is normal, then formula \eqref{eq:axdist} can also be used.)}
\end{commentx}
\begin{proof}
First we assume that $A$ is  non-parabolic. Then the signed distance in question is
\[\mathrm d^{\mathrm{pCK}}((x_{\mathrm{pCK}},y_{\mathrm{pCK}},z_{\mathrm{pCK}}),\axis_{\mathrm{pCK}}  \DW_{\mathrm{pCK}}(A))
-\radius_* \DW_*(A).\]
Using standard addition theorems the result follows.
Then the expression obtained extends to parabolic case continuously.
\end{proof}
\end{theorem}

Let $O_{\mathrm{pCK}}=(0,0,1)\in H^3_{\mathrm{pCK}}$, $O_{\mathrm{BCK}}=(0,0,0)\in H^3_{\mathrm{BCK}}$,
$O_{\mathrm{Ph}}=(0,0,1)\in H^3_{\mathrm{Ph}}$,  $O_{\mathrm{P}}=(0,0,0)\in H^3_{\mathrm{P}}$.
This ``central point'' $O_*$ is  characterized by being the point in $\DW_*\left(\begin{bmatrix}1&\\&-1\end{bmatrix}\right)\cap \DW_*\left(\begin{bmatrix}&-1\\1&\end{bmatrix}\right)$.
Also, let
\[O_A=\dfrac{|\det A|^2}2+\dfrac{|\tr A|^2}4+\dfrac12.\]
\begin{theorem}
\plabel{thm:oridist}
Suppose that $A$ is a $2\times2$ complex matrix with eigenvalues $\lambda_1,\lambda_2$.
Then $\DW_*(A)$ is the possibly degenerate $h$-ellipsoid, which is tangent to the asymptotic sphere
at $\iota_*(\lambda_1), \iota_*(\lambda_2)$ (double tangency counts in the parabolic case)
and such that the signed distance $O_*$ from it is
\begin{equation}\overleftarrow{\mathrm{dis}}_*(O_*,\DW_*(A) )=
\artanh\frac{\sqrt{\left(O_A\right)^2-|D_A|^2}-\sqrt{\left(U_A\right)^2-|D_A|^2}}{O_A +U_A}.
\plabel{eq:oridist}
\end{equation}
\begin{proof}
This is an immediate corollary of the previous statement.
\end{proof}
\end{theorem}
This provides a characterization of $\DW_*(A)$ for $2\times2$ matrices in fairly general terms.
A similar but simpler characterization can be given in terms of the distance  from $\infty_*$.

Distance from $\infty_*$ is, of course, not taken as such, but it is replaced by a choice of a distance function from $\infty_*$.
In our case, this will be the oriented distance from the horosphere $\mathcal P^{\infty}$
given by $z_{\mathrm{Ph}}=1$ (positive outside the horosphere, negative inside the horosphere).
It is easy to see that this distance is given by
\begin{align}
\overleftarrow{\mathrm{dis}}_*((x_*,y_*,z_*), \mathcal P^{\infty}_* )
&=-\log z_{\mathrm{Ph}}\plabel{eq:horopointdist}\\\notag
&=-\log\sqrt{ z_{\mathrm{pCK}}-(x_{\mathrm{pCK}})^2 -(y_{\mathrm{pCK}})^2   }\\\notag
&=\log\frac{1-z_{\mathrm{BCK}}}{\sqrt{1-(x_{\mathrm{BCK}})^2-(y_{\mathrm{BCK}})^2-(z_{\mathrm{BCK}})^2}}.
\end{align}
This is explained in \cite{LLL} in detail.

When we consider the oriented distance of a set $M$ from the horosphere, then the distance from the horosphere
is computed for any point of $M$, then infimum is taken.

\begin{remark}
If $\mathcal P_1$ and $\mathcal P_2$ are both horospheres, then
$\overleftarrow{\mathrm{dis}}(\mathcal P_1,\mathcal P_2)=\overleftarrow{\mathrm{dis}}(\mathcal P_2,\mathcal P_1)$.
\qedremark
\end{remark}

\begin{theorem}
\plabel{thm:horodist}
Suppose that $A$ is a $2\times2$ complex matrix with eigenvalues $\lambda_1,\lambda_2$.
Then $\DW_*(A)$ is the possibly degenerate $h$-ellipsoid, which is tangent to the asymptotic sphere
at $\iota_*(\lambda_1), \iota_*(\lambda_2)$ (double tangency counts in the parabolic case)
and such that its signed distance from $\mathcal P^{\infty}_*$ is
\begin{align}
\overleftarrow{\mathrm{dis}}_*(\DW_*(A),\mathcal P^{\infty}_*  )
&=-\log \left\|A-\frac{\tr A}2\Id_2\right\|_2
\plabel{eq:horodist}
\\\notag
&=-\log\left(\sqrt{\frac{U_A+D_A}2}+\sqrt{\frac{U_A-D_A}2}\right)
\\\notag
&=-\frac12\log\left(U_A+\sqrt{(U_A)^2-(D_A)^2}\right).
\end{align}
\begin{proof}
The RHS's of \eqref{eq:horodist} are easy to see to be equivalent, thus it is sufficient
to prove equality of the LHS to any of them.

If one can visualize an $h$-tube or $h$-horosphere in the Ph model, then (\ref{eq:horodist}/1) is quite obvious from (\ref{eq:horopointdist}/1).
\alter we can continue the Second Proof of Theorem \ref{thm:LSZ}.
Then the oriented distance is $-\frac12\arcosh\frac{U_A}{D_A}$ for $L$.
As this was obtained by a displacement appearing as a linear transform with linear coefficient $\frac1{\sqrt{-D_A}}$, the oriented distance for $A$ is corrected to
$\overleftarrow{\mathrm{dis}}_*(\DW_*(A),\mathcal P^{\infty}_*  )=-\frac12\left(\arcosh\frac{U_A}{D_A}\right)+\log\frac1{\sqrt{|D_A|}}.$
By simple arithmetics, this can be brought to (\ref{eq:horodist}/3).
By continuity, this extends to the parabolic case.
A statement  is that this data, beyond the spectral data, is sufficient to characterize the shell.
This can be seen easily in either of the projective models.
(Algebraically, this is equivalent to the recovery of $U_A$ beyond the spectral data.)
\end{proof}
\end{theorem}

\snewpage
\section{Numerical range}
\plabel{sec:W}
In this section, we will be interested in the numerical range
through its relationship to the Davis--Wielandt shell.
This relationship itself is rather straightforward,
as the numerical range is the vertical projection of
the Davis--Wielandt shell in the pCK (or in the Ph) model.

If $A$ is a normal $2\times2$ matrix, then it is easy to see that
its numerical range is the segment connecting its eigenvalues.
That leaves mainly the non-normal case for study.
\begin{theorem}
\plabel{thm:numrange}
For a non-normal  complex $2\times 2$ matrix $A$, the numerical range $\mathrm W(A)$ is given by
\begin{equation}
\bem x\\y\\1\eem^\top \m Q_{\mathrm W}(A)\bem x\\y\\1\eem\leq0,
\plabel{eq:Weq}
\end{equation}
where
\begin{align*}
\m Q_{\mathrm W}(A)_{11}=
&\tr(A^*A)-(\Rea \tr A)^2+2\Rea\det A,\\
\m Q_{\mathrm W}(A)_{12}=\m Q_{\mathrm W}(A)_{21}=
&-(\Rea \tr A)(\Ima\tr A)+2\Ima\det A,\\
\m Q_{\mathrm W}(A)_{22}=
&\tr(A^*A)-(\Ima \tr A)^2-2\Rea\det A,\\
\m Q_{\mathrm W}(A)_{13}=\m Q_{\mathrm W}(A)_{31}=
&-(\Rea \tr A)\dfrac{\tr(A^*A)-|\tr A|^2}2 -\Rea ((\det A)\overline{(\tr A)}),\\
\m Q_{\mathrm W}(A)_{23}=\m Q_{\mathrm W}(A)_{32}=
&-(\Ima \tr A)\dfrac{\tr(A^*A)-|\tr A|^2}2 -\Ima ((\det A)\overline{(\tr A)}),\\
\m Q_{\mathrm W}(A)_{33}=
&|\det A|^2-\left(\dfrac{\tr(A^*A)-|\tr A|^2}2\right)^2.
\end{align*}

The matrix $\m Q_{\mathrm W}(A)|_{\{1,2\}}$ is positive definite.
\begin{proof}
Here we will use the fact that the numerical range is the vertical projection of the Davis--Wielandt in the $\mathrm{pCK}$ model.
Yet, even this simple computation can be presented in several ways:

(i) We can compute the discriminant of $Q^{A}_{\mathrm{pCK}}(x_{\mathrm{pCK}},y_{\mathrm{pCK}},z_{\mathrm{pCK}})$
in $z_{\mathrm{pCK}}$.
This yields a quadric with matrix $-4 \m Q_{\mathrm W}(A)$.
The scalar multiple $-4$, of course, can be omitted.

(ii)
Consider the equation  $Q^{A}_{\mathrm{pCK}}(x_{\mathrm{pCK}},y_{\mathrm{pCK}},z_{\mathrm{pCK}})=0$
and the projection of the solution set  by the vertical projection.
The preimage of the boundary is identified by the condition
\[\frac{\partial Q^{A}_{\mathrm{pCK}}(x_{\mathrm{pCK}},y_{\mathrm{pCK}},z_{\mathrm{pCK}})}{\partial z_{\mathrm{pCK}}}=0.\]
This allows to eliminate $z_{\mathrm{pCK}}$ from the equation in order to obtain the boundary. % (where $=0$ stands).
Indeed, elimination will manifest in the linear map corresponding to the matrix
\[\m L^{\mathrm W}(A)=\begin{bmatrix}
1&&\\&1&\\ \Rea\tr A&\Ima\tr A&\dfrac{\tr(A^*A)-|\tr A|^2}2\\&&1
\end{bmatrix}.\]
Then,
\[\m L^{\mathrm W}(A)^{\top}\m Q_{\mathrm{pCK}}(A)\m L^{\mathrm W}(A)=\m Q_{\mathrm W}(A).\]

(iii) We can use the fact from projective geometry that the projection is the dual of the restriction.
Thus we can get the matrix of the numerical range as the inverse of the restriction of the inverse of the shell.
Then, indeed,
\begin{equation}
\left(\left(\m Q_{\mathrm{pCK}}(A)\right)^{-1}|_{\{1,2,4\}}\right)^{-1}
\equiv\left(\frac{-1}{(U_A)^2-|D_A|^2} \m G_{\mathrm{pCK}}(A)|_{\{1,2,4\}}\right)^{-1}
=\m Q_{\mathrm W}(A).
\plabel{eq:eqeqW}
\end{equation}
\begin{commentx}
(The latter equality is a bit computational but cf.~\eqref{eq:rao1} and \eqref{eq:rao2}.)
\end{commentx}
(Remark: According to Subsection \ref{ss:algproj}, (i) computes the Schur reduction, yielding
\[\m Q_{\mathrm W}(A)=\m Q_{\mathrm{pCK}}(A)\bo\|^{\mathrm{Schur}}_{\{1,2,4\}}.\]
Meanwhile (ii) and (iii) compute the Schur complement.
However the Schur reduction and the Schur complement are the same as $\m Q_{\mathrm{pCK}}(A)|_{\{3\}}=[1]$ is
of determinant $1$.)

In any case, by this we have identified the equation for the boundary of the numerical range.
As the numerical range (in this non-normal case) will be a proper elliptical disk,
we know that $\m Q_{\mathrm W}(A)|_{\{1,2\}}$ will be positive or negative definite.
Checking the trace,
\[\tr \m Q_{\mathrm W}(A)|_{\{1,2\}}=4U_A\]
shows that it is positive definite.
(If approach (i) or (ii) was taken, then the positive definiteness also follows by direct means.)
This positive definiteness implies that \eqref{eq:Weq} is correct with respect to the interior of the numerical range.
\end{proof}
\end{theorem}

\begin{theorem}
\plabel{thm:Warit}
Suppose that $A$ is linear operator on a $2$-dimensional Hilbert space.
Let us take
\[\m G_{\mathrm W}(A):=\m G_{\mathrm{pCK}}(A)|_{\{1,2,4\}}.\]
I.~e.,
\[\m G_{\mathrm W}(A)=
\bem
\dfrac{\Rea\det A}2-\dfrac{\tr(A^*A)-|\tr A|^2}4&\dfrac{\Ima\det A}2&\dfrac{\Rea\tr A}2\\
\dfrac{\Ima\det A}2&-\dfrac{\Rea\det A}2-\dfrac{\tr(A^*A)-|\tr A|^2}4&\dfrac{\Ima\tr A}2\\
\dfrac{\Rea\tr A}2&\dfrac{\Ima\tr A}2&1
\eem
.\]

Then
\begin{equation}
\m Q_{\mathrm W}(A)=-4\adj \m G_{\mathrm W}(A) ,
\plabel{eq:rao1}
\end{equation}
and
\begin{equation}
\det\m G_{\mathrm W}(A)=\frac14((U_A)^2-|D_A|^2).
\plabel{eq:rao2}
\end{equation}
Furthermore,
\begin{equation}
\det \m Q_{\mathrm W}(A)=-4((U_A)^2-|D_A|^2)^2,
\plabel{eq:rao3}
\end{equation}
and
\begin{equation}
\adj \m Q_{\mathrm W}(A)=4((U_A)^2-|D_A|^2) \m G_{\mathrm W}(A).
\plabel{eq:rao4}
\end{equation}
\begin{proof}
Equations \eqref{eq:rao1} and \eqref{eq:rao2} can be checked arithmetically.
Then \eqref{eq:rao3} and \eqref{eq:rao4} follow from of properties of $\adj$.
\end{proof}
\begin{proof}[Alternative proof]
In conjunction to the proof of Theorem \ref{thm:numrange} we can proceed as follows.
We can argue that after having method (ii) carried out, method (iii) must give \textit{exactly} the same matrix, thus
equality in \eqref{eq:eqeqW}  must hold.
Now, as $\m Q_{\mathrm W}(A)\vvarpropto  \m G_{\mathrm W}(A)^{-1}\vvarpropto \adj\m G_{\mathrm W}(A)$,
equation \eqref{eq:rao1} is sufficient to check out for position $(3,3)$ only, which is simple.
Taking the validity of \eqref{eq:eqeqW} into consideration again, this proves \eqref{eq:rao2}.
Then \eqref{eq:rao3} and \eqref{eq:rao4} are  also implied.
Altogether, this proves the statement in the non-normal case.
It extends to the normal case by continuity.
\end{proof}
\end{theorem}

\begin{remark}
(a) The simplest way to present the matrix $\m Q_{\mathrm W}(A)$ itself is \eqref{eq:rao1}.

(b) We can note that matrix of the Kippenhahnian quadric
\[K_A^{\mathrm W}(u,v,w)\equiv\det\left(u\frac{A+A^*}{2}+v\frac{A-A^*}{2\mathrm i}  + w\Id_2\right)\]
is $\m G^{A}_{\mathrm{pCK}}(A)|_{\{1,2,4\}}\equiv\m G_{\mathrm W}(A)$.
Thus, the direct Kippenhahnian proof of Theorem \ref{thm:numrange} is represented by \eqref{eq:rao1} again.

(c) The center of $\mathrm W(A)$ can be read off from the last column of $\m G_{\mathrm W}(A)$.
By the continuity of the numerical range as a set (or, by direct arguments), this is also valid in the normal case.
\qedremark
\end{remark}

Let us recall the matrices $\m Q_{\mathrm C}(A)$ and $\m G_{\mathrm C}(A)$
 defined in \eqref{eq:Qcore} and \eqref{eq:Gcore}.
It is easy to see that
\[\m Q_{\mathrm C}(A)=\m Q_{\mathrm W}(A)|_{\{1,2\}}\]
and
\[\m G_{\mathrm C}(A)=\m G_{\mathrm W}(A)\bo\|_{\{1,2\}}^{\mathrm{Schur}}.\]
(It is practical to take  $\m Q_{\mathrm C}(A)=\m Q_{\mathrm W}(A)|_{\{1,2\}}$
and $\m G_{\mathrm C}(A)=-\frac14\adj\m Q_{\mathrm W}(A)|_{\{1,2\}} $.)
\begin{theorem}[Uhlig, \cite{Uhl}]
\plabel{thm:uhlig}
(a) %In fact,
\begin{multline*}
\m Q_{\mathrm W}(A)=\\=
\bem1&&\dfrac{\Rea\tr A}2\\&1&\dfrac{\Ima\tr A}2\\&&1\eem^{-1,\top}
\underbrace{\bem\m Q_{\mathrm C}(A)\\&&-((U_A)^2-|D_A|^2)\eem}_{\m Q_{\mathrm W0}(A):=}
\bem1&&\dfrac{\Rea\tr A}2\\&1&\dfrac{\Ima\tr A}2\\&&1\eem^{-1}
.
\end{multline*}

(b) The eigenvalues of the $(2|1)$-blockdiagonal matrix $\m Q_{\mathrm W0}(A)$ are
\[ 2(U_A-|D_A|),\qquad 2(U_A+|D_A|),\quad|\quad -((U_A)^2-|D_A|^2).\]
\end{theorem}
\begin{theorem}
\plabel{thm:couhlig}
(a) %In fact,
\begin{equation*}
\m G_{\mathrm W}(A)=
\bem1&&\dfrac{\Rea\tr A}2\\&1&\dfrac{\Ima\tr A}2\\&&1\eem
\underbrace{\bem\m G_{\mathrm C}(A)\\&&1\eem}_{\m G_{\mathrm W0}(A):=}
\bem1&&\dfrac{\Rea\tr A}2\\&1&\dfrac{\Ima\tr A}2\\&&1\eem^\top
.
\end{equation*}

(b) The eigenvalues of the $(2|1)$-blockdiagonal matrix $\m G_{\mathrm W0}(A)$ are
\[ -\frac12(U_A-|D_A|),\qquad -\frac12(U_A+|D_A|),\quad|\quad 1.\]
\end{theorem}
\begin{proof}[Proofs]
(a) can easily be checked arithmetically. (b) follows easily, as the fact that the determinant and the trace
matches $\m Q_{\mathrm W}(A)|_{\{1,2\} }$ can be checked easily.
\end{proof}

Curiously, the matrix $\mathrm Q_{\mathrm W}(A)$ of  Theorem \ref{thm:numrange} as such does not much appear
in the literature, but  $\mathrm Q_{\mathrm W}(A)$ of Theorem \ref{thm:uhlig} appears in Uhlig \cite{Uhl}.
%\snewpage
\begin{theorem}[Toeplitz \cite{Toe}; Johnson \cite{Joh}, Uhlig \cite{Uhl}]
\plabel{thm:ellrange}
If $A$ is a $2\times2$ complex matrix, then
$\mathrm W(A)$ is the disk of the possibly degenerate ellipse, whose foci
are the eigenvalues of $A$ ($\mathbb R^2\simeq\mathbb C$), the focal distance
is $2\sqrt{|D_A|}$; and the major and minor semi-axes are
\[s^+=\sqrt{\dfrac{U_A+|D_A|}2}\qquad\text{and}\qquad s^-=\sqrt{\dfrac{U_A-|D_A|}2},\]
respectively.
\begin{proof}
Firstly, we examine the non-normal case.
Then the semi-axes come as the
square root of ($(-1)$ times the third of eigenvalue of $\m Q_{\mathrm W0}(A)$ divided by an other one).
The foci are %a bit
trickier:
It is known that if a non-degenerate quadric has matrix $\m M$, then its foci are
are given as $(\Rea f_i,\Ima f_i)$ where the $f_i$ are the roots of the quadratic equation
\[\bem 1\\\mathrm i\\ -f \eem^\top\m M^{-1} \bem 1\\\mathrm i\\ -f \eem=0.\]
Cf. Sommerville \cite{Som} for details; but it was already known to Pl\"ucker \cite{Pl},
inspiring the curve-theoretic foci.
In the present case, $\m M^{-1}$ can be replaced by $\m G_{\mathrm W}(A)$, and
then the focal equation reduces to
\[f^2-f\tr A+\det A=0,\]
which is the characteristic equation.
(This focal computation may look trivial in the Kippenhahnian view,
but it involves the fact that the curve-theoretic foci are indeed generalizations of the ordinary foci.)

The normal case extends by continuity, or can easily be checked directly.
\end{proof}
\end{theorem}

Theorem \ref{thm:ellrange} is the well-known elliptical range theorem.
The proof is, of course, not the possibly simplest one.

\snewpage
\begin{theorem}
\plabel{cor:fiveW}
For $2\times2$ complex matrices, the five data determines the numerical range, and vice versa.
(In particular, then, there is no loss of information in passing from the Davis--Wielandt shell to the numerical range.)
\begin{proof}
The five data determines the matrix up to unitary conjugation, thus the numerical range.
Conversely, if $\mathrm W(A)$ is given, then one recover the spectral data from the foci,
and using either the major or minor semi-axes, one can get $U_A$, and hence $\tr A^*A$.
\end{proof}
\end{theorem}

From Theorem \ref{thm:uhlig} and Theorem \ref{thm:couhlig}, one can conclude immediately that
\[\rank \m G_{\mathrm W}(A) =
\begin{cases}
3&\text{if $A$ is non-normal,}\\
2&\text{if $A$ is normal non-parabolic,}\\
1&\text{if $A$ is normal parabolic;}
\end{cases} \]
and
\[\rank \m Q_{\mathrm W}(A) =
\begin{cases}
3&\text{if $A$ is non-normal,}\\
1&\text{if $A$ is normal non-parabolic,}\\
0&\text{if $A$ is normal parabolic.}
\end{cases} \]

\begin{lemma}
\plabel{lem:Wspec}
Assume that the $2\times2$ complex matrix $A$ is non-normal, and
$\m M$ is the matrix of the quadric of the boundary of $\mathrm W(A)$.
Then

(a)
\begin{equation}
\m G_{\mathrm W}(A)=\frac{\m M^{-1}}{(\m M^{-1})_{33}}=\frac{\adj\m M}{\det\m M|_{\{1,2\}}};
\plabel{eq:WGrec}
\end{equation}

(b)
\[\m Q_{\mathrm W}(A)=\frac{-4}{(\det\m M)((\m M^{-1})_{33})^2}\cdot\m M
=\frac{-4\det\m M}{ (\det\m M|_{\{1,2\}})^2} \cdot\m M
=\frac{-4\adj\,\adj\m M}{\det\m M|_{\{1,2\}})^2}
.\]
\begin{proof}
(a) follows from $\m G_{\mathrm W}(A)\vvarpropto\m M^{-1}$ and $\m G_{\mathrm W}(A)_{33}=1$.
Then \eqref{eq:rao1} implies (b).
\end{proof}
\end{lemma}
\snewpage
\begin{disc}
\plabel{disc:GW}
One can decompose
\[\m G_{\mathrm W}(A)=-\frac12U_A\bem1\\&1\\&&0\eem+
\underbrace{\bem
\dfrac{\Rea\det A}2+\dfrac{|\tr A|^2}8&\dfrac{\Ima\det A}2&\dfrac{\Rea\tr A}2\\
\dfrac{\Ima\det A}2&-\dfrac{\Rea\det A}2+\dfrac{|\tr A|^2}8&\dfrac{\Ima\tr A}2\\
\dfrac{\Rea\tr A}2&\dfrac{\Ima\tr A}2&1
\eem}_{\m G^{\mathrm{spec}}_{\mathrm W}(A)}.\]
Then we can consider
\[\m G_{\mathrm W}^{\mathrm{eig}}(A)=\frac12|D_A|\bem1\\&1\\&&0\eem+\m G^{\mathrm{spec}}_{\mathrm W}(A),\]
and
\[\m G_{\mathrm W}^{\mathrm{ceig}}(A)=-\frac12|D_A|\bem1\\&1\\&&0\eem+\m G^{\mathrm{spec}}_{\mathrm W}(A).\]

Then one can see that
\[\m G_{\mathrm W}^{\mathrm{eig}}(A)=\frac12\left(\hat{\m x}_1\hat{\m x}_2^\top+ \hat{\m x}_2\hat{\m x}_1^\top\right),\]
where
\[\hat{\m x}_1=\bem\Rea \lambda_1\\\Ima\lambda_1\\1\eem
\qquad\text{and}\qquad
\hat{\m x}_2=\bem\Rea \lambda_2\\\Ima\lambda_2\\1\eem.\]
Geometrically, this corresponds to a pair dual lines, i.~e.~a pair of ordinary points,
 corresponding to the eigenvalues of $A$.

Furthermore
\[\m G_{\mathrm W}^{\mathrm{ceig}}(A)= \hat{\m x} \hat{\m x} ^\top+ \hat{\m y}\hat{\m y}^\top, \]
where
\[\hat{\m x} =\bem\frac12\left((\Rea \lambda_1)+(\Rea \lambda_2)\right)\\
 \frac12\left((\Ima \lambda_1)+(\Ima \lambda_2)\right)
\\1\eem
\qquad\text{and}\qquad
\hat{\m y}=\bem\frac12\left((\Ima \lambda_1)-(\Ima \lambda_2)\right)\\
\frac12\left((\Rea \lambda_2)-(\Rea \lambda_1) \right)
\\0\eem .\]
If $A$ is non-parabolic then this is a dual pointellipse, which vanishes on the dual of the line
$((\Rea \lambda_1)-(\Rea \lambda_2))x+((\Ima \lambda_1)-(\Ima \lambda_2))y-\frac12(|\lambda_1|^2-|\lambda_2|^2|)=0$,
which line is the minor axis of the numerical range.
(The information contained in $\m G_{\mathrm W}^{\mathrm{ceig}}(A)$ is, however, greater.
It is the pair of the eigenvalues of $A$ again.)
In the parabolic case,  $\m G_{\mathrm W}^{\mathrm{ceig}}(A)$ is just the double dual line
 $\m G_{\mathrm W}^{\mathrm{eig}}(A)$, encoding the center (double eigenvalue) again.
\end{disc}
%\snewpage
\begin{disc}
\plabel{disc:QW}
Let
\[\m Q_{\mathrm W}^{\mathrm{majax}} (A)=-4\adj \m G_{\mathrm W}^{\mathrm{eig}}(A);\]
\[\m Q_{\mathrm W}^{\mathrm{minax}} (A)=-4\adj \m G_{\mathrm W}^{\mathrm{ceig}}(A);\]
and
\[\m Q_{\mathrm W}^{\mathrm{cx}} (A)
 =\bem1&&-\frac12\Rea\tr A\\&1&-\frac12\Ima\tr A\\-\frac12\Rea\tr A&-\frac12\Ima\tr A&\frac14|\tr A|^2\eem.\]

Then one can see that
\begin{equation}
\m Q_{\mathrm W} (A)=((U_A)^2-|D_A|^2)\bem0\\&0\\&&-1\eem
+ 2(U_A-|D_A|)\m Q_{\mathrm W}^{\mathrm{cx}} (A)
+\m Q_{\mathrm W}^{\mathrm{majax}} (A),
\plabel{eq:anas}
\end{equation}
and
\begin{equation*}
\m Q_{\mathrm W} (A)=((U_A)^2-|D_A|^2)\bem0\\&0\\&&-1\eem
+ 2(U_A+|D_A|)\m Q_{\mathrm W}^{\mathrm{cx}} (A)
+\m Q_{\mathrm W}^{\mathrm{minax}} (A).
\end{equation*}

Then
\[ \m Q_{\mathrm W}^{\mathrm{majax}} (A)
=\hat{\m v}\hat{\m v}^\top,\]
where
\[
\hat{\m v}=\bem \Ima\lambda_1-\Ima\lambda_2\\\Rea\lambda_2-\Rea\lambda_1\\(\Rea\lambda_1)(\Ima\lambda_2)-(\Rea\lambda_2)(\Ima\lambda_1) \eem
=\bem\Rea \lambda_1\\\Ima\lambda_1\\1\eem\times \bem\Rea \lambda_2\\\Ima\lambda_2\\1\eem \]
(with $\times$ meaning the ordinary vector product of Gibbs).
~
Similarly,
\[ \m Q_{\mathrm W}^{\mathrm{minax}} (A)
=-\hat{\m w}\hat{\m w}^\top,\]
where
\[
\hat{\m w}
=\bem(\Rea \lambda_1)-(\Rea \lambda_2)\\(\Ima \lambda_1)-(\Ima \lambda_2)\\
 \frac12\left(|\lambda_2|^2-|\lambda_1|^2 \right)
 \eem
 .\]
\end{disc}

%\snewpage
\begin{cor}\plabel{lem:Wdeg}
Assume that $A$ is a normal complex $2\times2$ matrix with eigenvalues $\lambda_1,\lambda_2$.

(a)
\begin{equation}
\m G_{\mathrm W}(A)=
\frac12\left(\bem\Rea \lambda_1\\\Ima\lambda_1\\1\eem\bem\Rea \lambda_2\\\Ima\lambda_2\\1\eem^\top
+
\bem\Rea \lambda_2\\\Ima\lambda_2\\1\eem\bem\Rea \lambda_1\\\Ima\lambda_1\\1\eem^\top\right)
.
\plabel{eq:WGdegrec}
\end{equation}
I. e. the quadratic form of $\m G_{\mathrm W}(A)$ in $u,v,w$ is
\[( (\Rea \lambda_1)u+(\Ima \lambda_1)v+w)( (\Rea \lambda_2)u+(\Ima \lambda_2)v+w).\]

(b)
\[\m Q_{\mathrm W}(A)
=\hat{\m v}\hat{\m v}^\top,\]
where
\[
\hat{\m v}=\bem \Ima\lambda_1-\Ima\lambda_2\\\Rea\lambda_2-\Rea\lambda_1\\(\Rea\lambda_1)(\Ima\lambda_2)-(\Rea\lambda_2)(\Ima\lambda_1) \eem
=\bem\Rea \lambda_1\\\Ima\lambda_1\\1\eem\times \bem\Rea \lambda_2\\\Ima\lambda_2\\1\eem .\eqed \]

\begin{commentx}
\begin{proof}
We can compute the data from $A=\bem\lambda_1&\\&\lambda_2\eem$.
\end{proof}
\end{commentx}
\end{cor}
\begin{commentx}
\begin{proof}[Alternative proof for Theorem \ref{cor:fiveW}]
In the $\Leftarrow$ direction:
$\m G_{\mathrm W}(A)$ can be recovered from \eqref{eq:WGrec} / \eqref{eq:WGdegrec},
and from this matrix the five data can be read off immediately.
\end{proof}
\end{commentx}

As we have seen, $\m G_{\mathrm W}(A)$ involves no loss of data relative to the five data, not even in the normal case.
Regarding $\m Q_{\mathrm W}(A)$:

\begin{lemma}\plabel{lem:Wpar}
Assume that $A$ is complex $2\times2$ matrix.

(a) If $A$ is normal  with two distinct eigenvalues, then
$\m Q_{\mathrm W}(A)$ is a matrix of rank strictly $1$.
Geometrically, it corresponds to the (double) line containing the segment which is the numerical range.
Beyond that, $\m Q_{\mathrm W}(A)$, as it is, also contains the information of the distance of the eigenvalues $2\sqrt{|D_A|}$,
but essentially not more.

(b) If $A$ is normal with two equal eigenvalues, then
$\m Q_{\mathrm W}(A)=0 $.
\begin{proof}
Let us consider the description in Lemma \ref{lem:Wdeg}(b) with $\hat{\m v}\equiv \m v_{\mathrm W}(\lambda_1,\lambda_2)$.

(b) If $\lambda_1$ and $\lambda_2$ are equal, then
$(\Rea\lambda_1,\Ima\lambda_1)=(\Rea\lambda_2,\Ima\lambda_2)$, and $\m v_{\mathrm W}(\lambda_1,\lambda_2)$ vanishes.
This leads to the statement.

(a) If $\lambda_1$ and $\lambda_2$ are not equal, then
$(\Rea\lambda_1,\Ima\lambda_1)\neq(\Rea\lambda_2,\Ima\lambda_2)$, and $\m v_{\mathrm W}(\lambda_1,\lambda_2)\neq0$.
This leads to a double line for $\m Q^{\mathrm{W}}(A)$.
Then, by the construction of $\m v_{\mathrm W}(\lambda_1,\lambda_2)$,
it is easy to see that the line fits to the distinct points
$(\Rea\lambda_1,\Ima\lambda_1)$ and $(\Rea\lambda_2,\Ima\lambda_2)$.
From $\m Q^{\mathrm{W}}(A)$ one can reconstruct $\pm\m v_{\mathrm W}(\lambda_1,\lambda_2)$.
If $\m v_{\mathrm W}(\lambda_1,\lambda_2)=\pm(u_1,u_2,u_3)$, then
$2\sqrt{|D_A|}=|\lambda_1-\lambda_2|=\sqrt{(u_1)^2+(u_2)^2}$,
and this information is specific with respect to the scaling of $\pm\m v_{\mathrm W}(\lambda_1,\lambda_2)$.
(In particular, in the normal setting, we have the formula
$2\sqrt{|D_A|}=\sqrt{\tr\left(\bem 1&&\\&1&\\&&0\eem\m Q^{\mathrm{W}}(A)\right)}$.)
\end{proof}
\end{lemma}
\snewpage
\begin{remark}
\plabel{rem:Wscalarmat}
(a) For non-scalar matrices $A$, we can consider the assignment, say,
\begin{commentx}
\[\widehat{\m Q}_{\mathrm W}(A)=\frac1{2U_A}\m Q_{\mathrm W}(A)\]
\end{commentx}
\[\widehat{\m Q}_{\mathrm W}(A)=\frac1{2(U_A+|D_A|)}\m Q_{\mathrm W}(A).\]
(Reasonable variations are possible.)
~
If $A$ is not normal, then  $\widehat{\m Q}_{\mathrm W}(A)$ represents the numerical range faithfully.
(There is no information loss relative to the numerical range.)
~
If $A$ is normal and non-parabolic, then $\widehat{\m Q}_{\mathrm W}(A)$ contains no more
 information than the line of the eigenvalues.
(So, there is information loss relative to the numerical range.)

(b) If $A=\lambda\Id_2$ is a scalar matrix, then it is  reasonable to define
\[\widehat{\m Q}_{\mathrm W}(\lambda\Id_2)=
\bem1&0&-\Rea\lambda\\0&1&-\Ima\lambda\\-\Rea\lambda&-\Ima\lambda&|\lambda|^2\eem.\]
(Cf. $\m Q_{\mathrm W}^{\mathrm{cx}}(\lambda\Id_2)$.)
~
Here, when $A$ is normal and parabolic,    $\widehat{\m Q}_{\mathrm W}(A)$ represents the numerical range faithfully.
(There is no information loss relative to the numerical range.)
~
This yields a natural but non-continuous extension of the original assignment.

(c) Then, in general,
\begin{commentx}
\[\adj\widehat{\m Q}_{\mathrm W}(A)=\left(1-\left(\frac{|D_A|}{U_A}\right)^2\right)\m G_{\mathrm W}(A)\]
holds with $\frac00=0$.
\end{commentx}
\[\adj\widehat{\m Q}_{\mathrm W}(A)= \frac{U_A-|D_A|}{U_A+|D_A| } \m G_{\mathrm W}(A)\]
holds with $\frac00=1$.
\qedremark
\end{remark}
\snewpage
\section{Spectral invariants of $2\times2$ complex matrices}
\plabel{sec:spec22}

\subsection{Spectral type (review)}
\plabel{ssub:spectype2}
~\\

We extend the elliptic/parabolic/hyperbolic classification
of real $2\times2$ matrices to complex $2\times2$ matrices as follows.
We distinguish the following classes:

$\bullet$ real-elliptic case: two conjugate, strictly complex eigenvalues, %i.~e.
%\[\Ima\tr A=0,\qquad D_A>0;\]

$\bullet$ real-parabolic case: two equal real eigenvalues, %i.~e.
%\[\Ima\tr A=0,\qquad D_A=0;\]

$\bullet$ real-hyperbolic case: two distinct real eigenvalues, %i.~e.
%\[\Ima\tr A=0,\qquad D_A<0;\]

$\bullet$ non-real parabolic case: two equal strictly complex  eigenvalues, %i.~e.
%\[|\Ima\tr A|>0,\qquad D_A=0;\]

$\bullet$ semi-real  case: a real and a strictly complex eigenvalue, %i.~e.
%\[|\Ima \tr A|= \sqrt{2(|D_A|+(\Rea D_A))},\qquad \Ima \tr A\neq0;\]

$\bullet$ quasielliptic case:  two non-conjugate
%strictly complex
eigenvalues $\lambda_1,\lambda_2$ with $(\Ima \lambda_1)(\Ima\lambda_2)<0$,  %i.~e.
%\[|\Ima \tr A|< \sqrt{2(|D_A|+(\Rea D_A))},\qquad (D_A \ngtr 0\quad\text{or}\quad\Ima \tr A\neq0);\]

$\bullet$ quasihyperbolic case:  two distinct
%strictly complex
eigenvalues $\lambda_1,\lambda_2$ with $(\Ima \lambda_1)(\Ima\lambda_2)>0$.  %i.~e.
%\[|\Ima \tr A|>\sqrt{2(|D_A|+(\Rea D_A))},\qquad D_A\neq0.\]

One can see that the classes above are closed for conjugation by unitary matrices
and for real fractional linear (i.~e.~M\"obius) transformations, whenever they are applicable.

%\snewpage
\subsection{Spectral quantities (partly review)}\plabel{ssub:squant}
~\\

A quantity corresponding to $|D_A|$ is
\[E_A=\left(\frac{(\Ima\tr A)}2\right)^2+\frac{|D_A|-\Rea D_A}2.\]
Indeed, if $\lambda_1$ and $\lambda_2$ are the eigenvalues of $A$,
then it is easy to see that
\[|D_A|=\left|\frac{\lambda_1-\lambda_2}2\right|^2\]
and
\[E_A=\left|\frac{\lambda_1-\bar \lambda_2}2\right|^2;\]
i.~e.~one corresponds to the other after conjugating an eigenvalue.

Related quantities are
\[H_A=E_A-|D_A|=\left(\frac{(\Ima\tr A)}2\right)^2-\frac{|D_A|+\Rea D_A}2=(\Ima \lambda_1)(\Ima \lambda_2),\]
and
\[K_A=E_A+\Rea D_A=\left(\frac{(\Ima\tr A)}2\right)^2+\frac{|D_A|+\Rea D_A}2=\frac{(\Ima \lambda_1)^2+(\Ima \lambda_2)^2}2.\]

Then one can characterize the $2\times2$ complex spectral classes in terms of
the %vanishing and
signs of the quantities $|D_A|$, $E_A$, $H_A=E_A-|D_A|$, and $K_A=E_A+\Rea D_A$:
\[\begin{tabular}{l|cccc}
type &$|D_A|$&$E_A$&$H_A$&$K_A$\\
\hline
real-elliptic     &+&0&$-$&+\\
real-parabolic    &0&0&0&0\\
real-hyperbolic   &+&+&0&0\\
non-real parabolic&0&+&+&+\\
semi-real         &+&+&0&+\\
quasielliptic     &+&+&$-$&+\\
quasihyperbolic   &+&+&+&+\\
\end{tabular}\]

\begin{commentx}
(For the purposes of classification, the columns $|D_A|$ and $E_A$ can be replaced by a column of $|D_A|E_A$.)
\end{commentx}

As the spectral classes are closed to real fractional transformations (whenever they are applicable), the signs of
$|D_A|$, $E_A$, $K_A$, and $H_A$ are also invariant.
A finer invariant is shown by
\begin{lemma}
\plabel{lem:EDinvar}
The ratio
\[{E_A}:{|D_A|}=\left|{\lambda_1-\bar\lambda_2}|^2:|{\lambda_1-\lambda_2}\right|^2\]
is invariant under real fractional transformations (whenever they are applicable to $A$).
\qed
\end{lemma}

In fact, one can see that the quantity
$\frac{E_A-|D_A|}{E_A+|D_A|}$
will act as an improved, real M\"obius invariant version of the classifier $(\sgn)H_A$.

\begin{commentx}
\begin{remark}
(a) If $A$ is in canonical form \eqref{eq:canon2}, then the five data and $U_A,$ $|D_A|,$ $E_A,$ $K_A,$ $H_A$ are all
polynomial in $\Rea\lambda_1,\Ima\lambda_1,\Rea\lambda_2,\Ima\lambda_2,\tau$; thus computation with all this data
is very simple in that way.

(b) $U_A,$ $|D_A|^2,$ $2E_A-|D_A|,$ $2H_A+|D_A|,$ $2K_A-|D_A|$ are polynomial in the five data.
Thus, in terms of the more invariant  five data, computation is a bit less simple.
\qedremark
\end{remark}
\end{commentx}

The following construction will be mentioned only in Remark \ref{rem:CR}.
If $A$ is \textit{quasielliptic}, then $D_A\nleq0$, and $E_A\neq0$; we can define
\[B_A^0=\frac1{\sqrt{E_A}}\left(\frac{\Ima\tr A}{2}+\mathrm i\frac{\Ima D_A}{\sqrt{2(|D_A|+\Rea D_A})}\right).\]
Then $B_A^0$ is a unit vector.
In fact, if $A$ has eigenvalues $\lambda_1$ and $\lambda_2$ such that $\Ima\lambda_1>0>\Ima\lambda_2$, then
\[B_A^0=\frac1{\mathrm i}\frac{\lambda_1-\bar\lambda_2}{|\lambda_1-\bar\lambda_2|}.\]
In this case $B_{A^*}=-B_A$ holds.
\snewpage
\subsection{The canonical representatives (real M\"obius) (review)}
\plabel{ssub:preCR}

\begin{lemma}\plabel{lem:preCR}
For  complex $2\times2$ matrices
up to real M\"obius transformations and unitary conjugation
it is sufficient to consider the following cases:
\begin{equation}
\m 0_2=\begin{bmatrix}0&\\&0\end{bmatrix}
\plabel{eqx:022}
\end{equation}
(real-parabolic, normal) ;
\begin{equation}
S_\beta=\begin{bmatrix} 0&\cos\beta \\&\mathrm i\sin\beta\end{bmatrix} \qquad \beta\in\left[0,\frac\pi2\right]
\plabel{eqx:superbolic}
\end{equation}
($\beta=0$: real-parabolic non-normal, $0<\beta<\pi/2$: semi-real non-normal, $\beta=\pi/2$: semi-real normal);
\begin{equation}
L_{\alpha,t}^{\pm}=\begin{bmatrix} \cos\alpha+\mathrm i\sin\alpha& 2t\\
&-\cos\alpha\pm\mathrm i\sin\alpha \end{bmatrix} \qquad \alpha\in\left[0,\frac\pi2\right], t\geq 0
\plabel{eqx:loxodromic}
\end{equation}
\[\text{but the cases  $L_{0,t}^+$ and $L_{0,t}^-$ are identical}\]
($\alpha=0$:   real-hyperbolic case,  $0<\alpha<\pi/2$: quasielliptic [$-$] / quasihiperbolic case [$+$],
$\alpha=\pi/2$: real-elliptic [$-$] / non-real parabolic case [$+$];
$t=0$: normal, $t>0$ non-normal).

Apart from the degeneracy  for $\alpha=0$, the representatives above are inequivalent.
\qed
\end{lemma}
%\snewpage
\begin{example}\plabel{ex:confrep}
(a) Regarding $L_{\alpha,t}^\pm$:
\[U_{L^{+}_{\alpha,t}}=2t^2+(\cos\alpha)^2,
\quad%\]\[
\left|D_{L^{+}_{\alpha,t}}\right|=(\cos\alpha)^2,
\quad%\\]\[
E_{L^{+}_{\alpha,t}}=1,
\quad%\\]\[
E_{L^{+}_{\alpha,t}}\geq \left|D_{L^{+}_{\alpha,t}}\right|;\]
and
\[U_{L^{-}_{\alpha,t}}=2t^2+1,
\quad%\\]\[
\left|D_{L^{-}_{\alpha,t}}\right|=1,
\quad%\\]\[
E_{L^{-}_{\alpha,t}}=(\cos\alpha)^2,
\quad%\\]\[
E_{L^{-}_{\alpha,t}}\leq\left|D_{L^{-}_{\alpha,t}}\right|.\]

(b) Regarding $S_{\beta}$:
\[U_{S_\beta}=\frac14\left(1+(\cos\beta)^2\right),
\quad
\left|D_{S_\beta}\right|=E_{S_\beta}=\frac14\left(1-(\cos\beta)^2\right).
\]

(c) In case of $\m 0_2$:
\[U_{\m 0_2}=\left|D_{\m 0_2}\right|=E_{\m 0_2}=0.\eqedexer\]
\end{example}
\begin{cor} \plabel{cor:rconftriple}
The triple ratio (i.~e.~ ordered triple up to nonzero scalar multipliers)
\begin{equation}
U_A\,\,:\,\,|D_A|\,\,:\,\,E_A
\plabel{eq:preconfrep}
\end{equation}
and (when $|D_A|=E_A>0$) the possible choice of
\[\text{real-hyperbolic / semi-real type}\]
together form a full invariant of $2\times2$ complex matrices with respect
to equivalence by real M\"obius transformations and unitary conjugation.
\begin{proof}
\eqref{eq:preconfrep} almost distinguishes the canonical representatives of Lemma \ref{lem:preCR}.
The only ambiguity is between $L_t\equiv L_{0,t}$ and $S_\beta$ when $\cos\beta=\frac{t}{\sqrt{1+t^2}}$.
%(This is the general case $|D_A|=E_A>0$.)
\end{proof}
\end{cor}
\snewpage
\begin{commenty}
The content of the following three subsections will be considered only in Remark \ref{rem:CR}.
\subsection{The square root}\plabel{ssub:sq}~\\

Let $\widetilde{\sgn}$ denote the half-sided sign function, for which
\[\widetilde{\sgn}(x)=
\begin{cases}
1&\text{if }x\geq0,\\
-1&\text{if }x<0.
\end{cases}\]
Then, according to our conventions, the canonical square root we use for complex numbers,
 i.~e.~the square root cut along $(-\infty,0)-\mathrm i\epsilon$, is given as
\begin{equation}
\sqrt z=\sqrt{\frac{|z|+\Rea z}2}+ \mathrm i\cdot\widetilde{\sgn}(\Ima z)\sqrt{\frac{|z|-\Rea z}2}.
\plabel{eq:sqrtdef}
\end{equation}

We remark that two trivial but sometimes useful identities are
\[\Rea(a\sqrt z)\Rea(b\sqrt z)=\Ima(a\sqrt{- z})\Ima(b\sqrt{- z}),\]
and
\[\Rea(a\sqrt z)\Ima(b\sqrt z)=-\Ima(a\sqrt{- z})\Rea(b\sqrt{- z}).\]
\begin{commentx}
(These are actually valid for any consistent choice of square roots.)
\end{commentx}
~\\

\subsection{The spectrum}
~\\

Using \eqref{eq:sqrtdef},
the points of the spectrum of the $2\times 2$ matrix $A$ can be described as
\begin{equation*}
\frac{\tr A}2\pm\sqrt{-D_A}=
\frac{\tr A}2\pm\left(\sqrt{\frac{|D_A|-\Rea D_A}{2}}+
\widetilde{\sgn}(-\Ima D_A)
\mathrm i \sqrt{\frac{|D_A|+\Rea D_A}{2}}\right).
\end{equation*}
\begin{commentx}
In general, the complex $2\times2$ matrix $A$ has a real eigenvalue if and only if
\[|\Ima \tr A|= \sqrt{2(|D_A|+(\Rea D_A))}.\]
(As $|\Ima \sqrt z|=\sqrt{\frac{|z|-\Rea z}2}$ holds for complex numbers in general.)
\end{commentx}
(One can use
\[\Rea D_A=(\Rea\det A)-\frac{|\tr A|^2}4\]
and
\[\Ima D_A=(\Ima\det A)-\frac{(\Rea \tr A)(\Ima\tr A)}2\]
in terms of the five data.)
\snewpage
\subsection{Spectral quantities (one more)}\plabel{ssub:bquant}
~\\

Let us set
\[B_A=\frac{\Ima\tr A}2+\mathrm i \Ima\sqrt{D_A}.\]

It is easy to that $|B_A|^2=E_A$, but $B_A$ is less naturally defined than $E_A$.
($B_A$ is an analogue of $\sqrt{D_A}$.)
\begin{lemma}
\plabel{lem:Bellip}
Assume that $A$ has eigenvalues $\lambda_1$ and $\lambda_2$ such that $\Ima\lambda_1>0>\Ima\lambda_2$.
Then $D_A\nleq0$ and
\[B_A=\frac{\lambda_1-\overline{\lambda_2}}{2\mathrm i}.\]
In this case $B_{A^*}=-B_A$ also holds.
\begin{proof}This follows from the formulas
\[\sqrt{D_A}=\sqrt{\frac{|D_A|+\Rea D_A}{2}}+\widetilde{\sgn}(\Ima D_A) \mathrm i\sqrt{\frac{|D_A|-\Rea D_A}{2}},\]
\[\lambda_1=\frac12\tr A+\left(\widetilde{\sgn}(-\Ima D_A)\sqrt{\frac{|D_A|-\Rea D_A}{2}}+
\mathrm i \sqrt{\frac{|D_A|+\Rea D_A}{2}}\right),\]
\[\lambda_2=\frac12\tr A-\left(\widetilde{\sgn}(-\Ima D_A)\sqrt{\frac{|D_A|-\Rea D_A}{2}}+
\mathrm i \sqrt{\frac{|D_A|+\Rea D_A}{2}}\right);\]
and the fact that here $\Ima D_A=0$ implies $D_A>0$.
\end{proof}
\end{lemma}
\end{commenty}

\snewpage
\section{The conformal range of $2\times2$ complex matrices }
\plabel{sec:CR}

\subsection{On the hyperbolic plane
(review)
}
~\\

In term of the models, BCK, pCK, let
$\pi^{[2]}_{*}:\overline{H^3_{*}}\rightarrow \overline{H^2_{*}}$ be the canonical projection
from the hyperbolic $3$-space to the hyperbolic $2$-space
by the elimination of the second coordinate, or, in other terms, setting it to $0$.
(It is slightly more complicated in the conformal models.)
Geometrically, this is an $h$-orthogonal projection from the $h$-space $\overline{H^3_{*}}$ to the canonically embedded $h$-plane $i^{[2]}(\overline{H^2_{*}})$ (asymptotically closed).

The Poincar\'e half-plane model offers some conveniences for the conformal range, thus we also reflect on it.
The canonical transcription from the projective plane models to the Poincar\'e half-plane model $\mathrm{Ph}$ is given by
\begin{align*}
(x_{\mathrm{Ph}} ,z_{\mathrm{Ph}})
&=\frac{(x_{\mathrm{BCK}},\sqrt{1-(x_{\mathrm{BCK}})^2-(z_{\mathrm{BCK}})^2 })}{1-z_{\mathrm{BCK}}}          =\left(x_{\mathrm{pCK}},\sqrt{z_{\mathrm{pCK}}-(x_{\mathrm{pCK}})^2 }   \right).
\end{align*}

Let $\iota^{[2]}_*=\pi^{[2]}_*\circ\iota_*$, the mapping of the Riemann-sphere to the asymptotically closed hyperbolic plane.
Then
\[\iota_{\mathrm{pCK}}^{[2]}(\lambda)=\left(\Rea\lambda,|\lambda|^2\right)
\quad\,\,\text{and}\quad\,\,
\iota_{\mathrm{BCK}}^{[2]}(\lambda)=\left(\frac{2\Rea\lambda}{|\lambda|^2+1},
\frac{|\lambda|^2-1}{|\lambda|^2+1}\right).\]

In terms of Ph this yields
\[\iota_{\mathrm{Ph}}^{[2]}(\lambda)=\left(\Rea\lambda,|\Ima\lambda| \right).\]

 \textit{All} collineations (i.~e.~congruences, i.~e.~isometries) of the hyperbolic plane are induced by real fractional linear
transformations $f:\lambda\mapsto \frac{ax+b}{cx+d}$ where $a,b,c,d\in\mathbb R$, $ad-bc\neq0$;
such that $f_*$ restricts to our canonical $h$-plane.
In that restriction it is useful to know that for real M\"obius transformations $f$
the effects $f_*$ commute with  $\pi^{[2]}_*$.

Projective representations of the isometries can be also be obtained by restriction.
In terms of the BCK model,
\[R_{\mathrm{BCK}}^{[2]}(f)=R_{\mathrm{BCK}}(f)|_{\{1,3,4\}\times \{1,3,4\}}\]
can be taken. (Remember, this applies in the case when $a,b,c,d\in\mathbb R$.)
This yields
\begin{equation}
R_{\mathrm{BCK}}^{[2]}(f)=\frac1{|ad-bc|}
\begin{bmatrix}
 ad+ cb& ca- db& ca+ db
\\
 ab- cd&\frac{a^2-b^2-c^2+d^2}{2}&\frac{a^2+b^2-c^2-d^2}{2}
\\
 ab+ cd&\frac{a^2-b^2+c^2-d^2}{2}&\frac{a^2+b^2+c^2+d^2}{2}
\end{bmatrix}.
\plabel{eq:projrepBCK}
\end{equation}
As the only nontrivial omitted term was $R_{\mathrm{BCK}}(f)_{22}=\frac{ad-bc}{|ad-bc|}$,
it is easy to see that $\det R_{\mathrm{BCK}}^{[2]}(f)=\frac{ad-bc}{|ad-bc|} $.
We obtain all elements of $\mathrm O^\uparrow(2,1)$ in this way.

Their effect for $w=x_{\mathrm{Ph}}+\mathrm iz_{\mathrm{Ph}}$ is given by
\[ f_{\mathrm{Ph}}^{[2]}: w\mapsto \left(\frac{aw+b}{cw+d}\right)^{\text{ conjugated if } ad-bc<0}.\]

\snewpage
\begin{commentx}
In the Ph model, the natural distance function  is given by
\begin{align*}
\mathrm d^{\mathrm{Ph}}((x_{\mathrm{Ph}} ,z_{\mathrm{Ph}}),(\tilde x_{\mathrm{Ph}} ,\tilde z_{\mathrm{Ph}}))&
=\arcosh\left( 1+\frac{(x_{\mathrm{Ph}}-\tilde x_{\mathrm{Ph}})^2 + (z_{\mathrm{Ph}}-\tilde z_{\mathrm{Ph}})^2}{2z_{\mathrm{Ph}}\tilde z_{\mathrm{Ph}}}\right)\\
&=2\arsinh\left( \frac{\sqrt{(x_{\mathrm{Ph}}-\tilde x_{\mathrm{Ph}})^2 +(z^*_{\mathrm{Ph}}-\tilde z_{\mathrm{Ph}})^2}}{2\sqrt{z_{\mathrm{Ph}}\tilde z_{\mathrm{Ph}}}}\right).
\end{align*}

In the pCK model, the natural distance function   is given by
\begin{multline}
\mathrm d^{\mathrm{pCK}}((x_{\mathrm{pCK}}, z_{\mathrm{pCK}}),(\tilde x_{\mathrm{pCK}},\tilde z_{\mathrm{pCK}}))=\\
=\arcosh\left( \frac{  \frac12z_{\mathrm{pCK}}+\frac12\tilde z_{\mathrm{pCK}}-x_{\mathrm{pCK}}\tilde x_{\mathrm{pCK}}
}{\sqrt{ z_{\mathrm{pCK}}- (x_{\mathrm{pCK}})^2}\sqrt{\tilde  z_{\mathrm{pCK}}- (\tilde x_{\mathrm{pCK}})^2 }}\right).
\plabel{eq:disP2}
\end{multline}
\end{commentx}
In the BCK model,  the natural distance function  is given by
\begin{multline}\mathrm d^{\mathrm{BCK}}((x_{\mathrm{BCK}}, z_{\mathrm{BCK}}),(\tilde x_{\mathrm{BCK}}, \tilde z_{\mathrm{BCK}}))=\\
=\arcosh\left(\frac{1-x_{\mathrm{BCK}}\tilde x_{\mathrm{BCK}} - z_{\mathrm{BCK}}\tilde z_{\mathrm{BCK}}  }{\sqrt{1-(x_{\mathrm{BCK}})^2-(z_{\mathrm{BCK}})^2}\sqrt{1-(\tilde x_{\mathrm{BCK}})^2-(\tilde z_{\mathrm{BCK}})^2}}\right).
\plabel{eq:dis2}
 \end{multline}

In particular,
\begin{align}
\mathrm d^{\mathrm{BCK}}( (0,0),(s,0))
&=\arcosh\frac{1}{\sqrt{1-s^2}}
=\arsinh\sqrt{\frac{s^2}{1-s^2}}
=\artanh |s|
\plabel{eq:distex2}
\\
\notag
&=\frac12\arcosh\frac{1+s^2}{1-s^2}
=\frac12\arsinh\frac{2|s|}{1-s^2}
=\frac12\artanh\frac{2|s|}{1+s^2}
.
\end{align}

\snewpage
\begin{disc}
In order not to interrupt the continuity of the presentation later, there is a point which we mention here.
In certain situations it is useful to consider asymptotic distances, which are regularized distances from an asymptotic point.
In fact, these are just oriented distance from horocycles with the given asymptotical point.

If the asymptotic point is $\iota^{[2]}(\infty)$, then a notable asymptotic distance is given by
\begin{align}
\mathrm d_{[\infty]}^*(x_*,z_*)
&=-\log z_{\mathrm{Ph}}\plabel{eq:allius}\\
&=-\log\sqrt{z_{\mathrm{pCK}}-(x_{\mathrm{pCK}})^2}\notag\\
&=\log\left(\frac{1-z_{\mathrm{BCK}} }{\sqrt{1-(x_{\mathrm{BCK}})^2-(z_{\mathrm{BCK}})^2}}\right).\notag
\end{align}
(This is the oriented distance from the horocycle $z_{\mathrm{Ph}}=1$, as it is explained in \cite{LLL}.)

Assume that $\lambda_0\in\mathbb R$.
Then an asymptotical distance from $\iota^{[2]}(\lambda_0)$ is given by
\begin{align}
\mathrm d_{[\lambda_0]}^*(x_*,z_*) &=\log\frac{(x_{\mathrm{Ph}}-\lambda_0)^2+(z_{\mathrm{Ph}})^2}{z_{\mathrm{Ph}}}\plabel{eq:salius}\\
&=\log\frac{z_{\mathrm{pCK}}-2\lambda_0 x_{\mathrm{pCK}}+\lambda_0^2 }{\sqrt{z_{\mathrm{pCK}}-(x_{\mathrm{pCK}})^2 }  }\notag\\
&=\log\frac{1-2\lambda_0 x_{\mathrm{BCK}}+(1-\lambda_0^2)z_{\mathrm{BCK}}+\lambda_0^2 }{\sqrt{1-(x_{\mathrm{BCK}})^2-(z_{\mathrm{BCK}})^2 }  }.\notag
\end{align}

(Indeed \eqref{eq:salius} can be obtained from \eqref{eq:allius}
applying involutive isometry
\[(x_{\mathrm{Ph}},z_{\mathrm{Ph}})\mapsto\frac{(x_{\mathrm{Ph}}-\lambda_0,z_{\mathrm{Ph}})}{(x_{\mathrm{Ph}}-\lambda_0)^2+
(z_{\mathrm{Ph}})^2}+(\lambda_0,0)
\] the linear fractional transformation $x\mapsto\frac1{x-\lambda_0}+\lambda_0$,
interchanging the roles of $\iota^{[2]}(\lambda_0)$ and $\iota^{[2]}(\infty)$
; plus systematic transcription to other models.)
\end{disc}
\snewpage
~

\subsection{On some basic properties of the conformal range (review)}
~\\

As the relationship of the conformal range and the other ranges has been already discussed, we
mention only one general property.

From the natural properties of the canonical projection,
it is easy to see that the analogy of Theorem \ref{thm:DWtrans} holds:
If $f:\lambda\mapsto \frac{a\lambda+b}{c\lambda+d}$, $a,b,c,d\in\mathbb R$, $ad-bc\neq0$ is a real fractional linear
(i.~e.~real M\"obius) transformation,
then for $f(A)=\frac{aA+b}{cA+d}$, the Davis--Wielandt shell is given as
\[\DW_*^{\mathbb R}(f(A))=f_*(\DW_*^{\mathbb R}(A)).\]

In what follows, we deal exclusively with case when $A$ is a $2\times 2$ complex matrix.

\subsection{Some observations}\plabel{ssub:genobs}
~\\

In the subsection we motivate the use of the `reduced five data' based on observations from \cite{L2},
but otherwise this subsection can be omitted.

Complex $2\times2$ matrixes are determined, up to unitary conjugation, by the `five data'
\eqref{eq:polr}.
Thus, these determine the Davis-Wielandt shell, and, in particular, the  conformal range.
Now, in the case of the conformal range, slightly less is sufficient.
\begin{lemma}\plabel{lem:CRcompar}
Suppose that $A_1$ and $A_2$ are two complex $2\times2$ matrices.
We claim that
\[\CR(A_1)=\CR(A_2)\]
holds if and only if for all $\lambda\in\mathbb R$,
\begin{equation}
|\det\, (A_1-\lambda\Id_2)|^2=|\det\,(A_2-\lambda\Id_2)|^2
,\plabel{eq:det2pol}
\end{equation}
and
\begin{equation}
\tr \,(A_1-\lambda\Id_2)^*(A_1-\lambda\Id_2)= \tr\, (A_2-\lambda\Id_2)^*(A_2-\lambda\Id_2)
.\plabel{eq:tr2pol}
\end{equation}
hold.
\end{lemma}
\begin{proof}
According to the dual viewpoint explained in \cite{L2}, the information $\CR(A)$
is equivalent to the information $\lambda\in\mathbb R\mapsto (\|A-\lambda\Id_2\|_2,\|A-\lambda\Id_2\|_2^- )$.
By %Now, according to
Lem\-ma \ref{lem:normcomputeD}, this is equivalent to the information
$\lambda\in\mathbb R\mapsto ( \tr (A-\lambda\Id_2)^*(A-\lambda\Id_2),$ $|\det(A-\lambda\Id_2)|^2 )$.
\end{proof}

It is easy to see, that if $A$ is a complex $2\times2$ matrix, $\lambda\in\mathbb R$, then
\begin{multline*}|\det\, (A-\lambda\Id_2)|^2=\\
=\lambda^4-(2\Rea\tr A)\lambda^3+(|\tr A|^2+2\Rea \det A)\lambda^2
-2(\Rea((\det A)\overline{\tr A})  )\lambda+|\det A|^2
,\end{multline*}
and
\[\tr \,(A-\lambda\Id_2)^*(A-\lambda\Id_2)= 2\lambda^2-(2\Rea\tr A)\lambda +\tr(A^*A).\]
\begin{theorem}\plabel{thm:redfive}
For complex $2\times2$ matrices $A$,
the information contained conformal range is in bijective correspondence
to the `reduced five data'
\[ \Rea \tr A,\quad|\tr A|^2+2\Rea \det A,\quad\Rea((\det A)\overline{\tr A}),\quad |\det A|^2,\quad  \tr(A^*A).\]
\begin{proof}
Considering the polynomials \eqref{eq:det2pol} and \eqref{eq:tr2pol} in $\lambda$,
the information contained in them is the same as in their list of coefficients.
\end{proof}
\end{theorem}
~
\snewpage
\subsection{The matrix of the conformal range}
~\\

If $A$ is normal, then $\mathbb C^2$ is the orthogonal direct sum of its complex eigenspaces.
Thus $\DW_*^{\mathbb R}(A)$ is the (possibly degenerate) $h$-segment connecting the points corresponding to the eigenvalues of $A$.
In the non-normal case:
\begin{theorem}
\plabel{thm:CR}
Suppose that $A$ is linear operator on a $2$-dimensional Hilbert space.
Assume that $A$ is not normal.
Then $\DW^{\mathbb R}_{\mathrm{pCK}}(A)$ is given by the quadratic inequality
\begin{equation}
\begin{bmatrix}x_{\mathrm{pCK}}\\ z_{\mathrm{pCK}}\\1\end{bmatrix}^\top
\m Q_{\mathrm{pCK}}^{\mathbb R}(A)
\begin{bmatrix}x_{\mathrm{pCK}}\\ z_{\mathrm{pCK}}\\1\end{bmatrix}\leq0
\plabel{eq:LSZ3}
\end{equation}
such that
\[\m Q_{\mathrm{pCK}}^{\mathbb R}(A)=
\begin{bmatrix}
Z^2-4Y&&2X-VZ&&2VY-XZ\\\\
2X-VZ&&Z-W+V^2&&\dfrac{WZ}2-VX -\dfrac{Z^2}2 \\\\
2VY-XZ&&\dfrac{WZ}2-VX -\dfrac{Z^2}2&&X^2-YW+YZ
\end{bmatrix},
\]
where
\begin{align}
V&\equiv \Rea \tr A,
\plabel{eq:reducedfive}\\
W&\equiv |\tr A|^2+2(\Rea\det A),
\notag\\
X&\equiv \Rea ((\det A)(\overline{\tr A})),
\notag\\
Y&\equiv |\det A|^2,\qquad
\notag\\
Z&\equiv \tr A^*A
\notag
\end{align}
is the `reduced' five data.

In this non-normal case, $\m Q_{\mathrm{pCK}}^{\mathbb R}(A)|_{\{1,2\}}$ is positive definite.
\proofremark{
In \eqref{eq:LSZ3}, replacing
\[\begin{bmatrix}x_{\mathrm{pCK}}\\ z_{\mathrm{pCK}}\\1\end{bmatrix}
\text{ by }
\begin{bmatrix}x_{\mathrm{BCK}}\\ 1+z_{\mathrm{BCK}}\\1-z_{\mathrm{BCK}}\end{bmatrix},
\begin{bmatrix}x_{\mathrm{P}}\\\frac{(x_{\mathrm{P}})^2+(z_{\mathrm{P}}+1)^2}2,\\
\frac{(x_{\mathrm{P}})^2+(z_{\mathrm{P}}-1)^2}2\end{bmatrix}
\begin{bmatrix}x_{\mathrm{Ph}}\\
(x_{\mathrm{Ph}})^2+(z_{\mathrm{Ph}})^2\\1\end{bmatrix},
\]
one obtains the corresponding equations in $\mathrm{BCK}$, $\mathrm{P}$, $\mathrm{Ph}$, respectively.

In the case of the BCK model, this allows to take the viewpoint that
\begin{equation}
\begin{bmatrix}x_{\mathrm{BCK}}\\ 1+z_{\mathrm{BCK}}\\1-z_{\mathrm{BCK}}\end{bmatrix}^\top
\m Q_{\mathrm{pCK}}^{\mathbb R}(A)
\begin{bmatrix}x_{\mathrm{BCK}}\\ 1+z_{\mathrm{BCK}}\\1-z_{\mathrm{BCK}}\end{bmatrix}
\equiv
\begin{bmatrix}x_{\mathrm{BCK}}\\ z_{\mathrm{BCK}}\\1\end{bmatrix}^\top
\m Q_{\mathrm{BCK}}^{\mathbb R}(A)
\begin{bmatrix}x_{\mathrm{BCK}}\\ z_{\mathrm{BCK}}\\1\end{bmatrix}
\plabel{eq:trans20}
\end{equation}
where
\begin{equation}
\m Q_{\mathrm{BCK}}^{\mathbb R}(A)
=
\begin{bmatrix}
1&&\\&1&1\\&-1&1
\end{bmatrix}^\top
\m Q_{\mathrm{BCK}}^{\mathbb R}(A)
\begin{bmatrix}
1&&\\&1&1\\&-1&1
\end{bmatrix}
\plabel{eq:trans2}
\end{equation}
is obtained canonically.
This is similar to the case of the Davis--Wielandt shell.
}
\end{theorem}

\snewpage
\begin{proof}[A proof of Theorem \ref{thm:CR} via the projection approach.]
Here we will use the fact the conformal range is a projection of the Davis-Wielandt shell.
Again, this projection can be obtained in various ways:

(i) We can take the discriminant of $Q^{A}_{\mathrm{pCK}}(x_{\mathrm{pCK}},y_{\mathrm{pCK}},z_{\mathrm{pCK}})$
in variable $y_{\mathrm{pCK}}$.
Then obtain the quadric with matrix $-4 \m Q_{\mathrm{pCK}}^{\mathbb R}(A)$.

(ii) Consider the equation  $Q^{A}_{\mathrm{pCK}}(x_{\mathrm{pCK}},y_{\mathrm{pCK}},z_{\mathrm{pCK}})=0$
and the projection of the solution set  by $\pi^{[2]}_{\mathrm{pCK}}$.
The preimage of the boundary is identified by the condition
\[\frac{\partial Q^{A}_{\mathrm{pCK}}(x_{\mathrm{pCK}},y_{\mathrm{pCK}},z_{\mathrm{pCK}})}{\partial y_{\mathrm{pCK}}}=0.\]
This allows to eliminate $y_{\mathrm{pCK}}$ from the equation in order to obtain the boundary. % (where $=0$ stands).
Indeed, elimination will manifest in the linear map corresponding the matrix
\[\m L^{\mathbb R}_{\mathrm{pCK}}(A)=\begin{bmatrix}
1&&\\\dfrac{-2\Ima\det A}{\tr(A^*A)-2\Rea\det A}&\dfrac{\Ima \tr A}{\tr(A^*A)-2\Rea\det A}&
\dfrac{\Rea((\det A)\overline{(\tr A}))}{\tr(A^*A)-2\Rea\det A}\\&1&\\&&1
\end{bmatrix}.\]
(Note  that $\tr(A^*A)-2\Rea\det A>0$ if $A$ is non-normal.)
Then,  we obtain
\begin{equation}
(\m L^{\mathbb R}_{\mathrm{pCK}}(A))^\top
\m Q_{\mathrm{pCK}}(A)\m L^{\mathbb R}_{\mathrm{pCK}}(A)
=
\frac1{ {\tr(A^*A)-2\Rea\det A} }\cdot
\m Q_{\mathrm{pCK}}^{\mathbb R}(A)
.
\plabel{eq:mCR}
\end{equation}
After that, the scalar factor is omitted.

(iii) Again,
we can use the fact from projective geometry that the projection is the dual of the restriction.
Thus we can get the matrix of the conformal range as the inverse of the restriction of the inverse of the matrix of the shell.
Then, indeed,
\begin{align}
\left(\left(\m Q^{\mathbb R}_{\mathrm{pCK}}(A)\right)^{-1}|_{\{1,3,4\}}\right)^{-1}
&\equiv\left(\frac{-1}{(U_A)^2-|D_A|^2} \m G_{\mathrm{pCK}}(A)|_{\{1,3,4\}}\right)^{-1}
\plabel{eq:eqeqCR}
\\&=
\frac{1}{\tr(A^*A)-2\Rea\det A}\m Q_{\mathrm{pCK}}^{\mathbb R}(A).
\notag
\end{align}
\begin{commentx}
(This is somewhat computational but not more than \eqref{eq:Gadj} and \eqref{eq:detGR22} together.)
\end{commentx}
(Remark: Again, according to Subsection \ref{ss:algproj},  (i) computes the Schur reduction, yielding
\[\m Q_{\mathrm{pCK}}^{\mathbb R}(A)=\m Q_{\mathrm{pCK}}(A)\bo\|^{\mathrm{Schur}}_{\{1,3,4\}}.\]
Meanwhile (ii) and (iii) compute the Schur complement.
Here Schur reduction and the Schur complement are not the same as
$\m Q_{\mathrm{pCK}}(A)|_{\{2\}}=[\tr(A^*A)-2\Rea\det A ]=[2(U_A-|D_A|+2E_A)]$ introduces a factor.)

In this non-normal case we know that the conformal range is a non-degenerate ellipse,
thus $\m Q_{\mathrm{pCK}}^{\mathbb R}(A)|_{\{1,2\}\times \{1,2\}}$ is either positive definite or negative definite.
However, it is easy to see that $\m Q_{\mathrm{pCK}}^{\mathbb R}(A)_{11}=Z-4Y=\tr(A^*A)^2-4|\det A|^2\geq0$,
thus it will be positive definite.
In particular, the specification of the interior in \eqref{eq:LSZ3} is also correct.
(In cases (i) and (ii) this latter fact is quite transparent anyway.)
\end{proof}

\snewpage
%The statement of the following auxiliary theorem is independent from the geometric content of Theorem \ref{thm:CR},
%only the definition of $\m Q_{\mathrm{pCK}}^{\mathbb R}(A)$ is referred.
\begin{theorem}
\plabel{thm:CRarit}
Suppose that $A$ is linear operator on a $2$-dimensional Hilbert space.
Let us take
\[\m G^{\mathbb R}_{\mathrm{pCK}}(A):=\m G_{\mathrm{pCK}}(A)|_{\{1,3,4\}}.\]
I. e., in terms of \eqref{eq:reducedfive}, we define
\begin{equation}
\m G_{\mathrm{pCK}}^{\mathbb R}(A)=
\begin{bmatrix}\dfrac{W-Z}4&\dfrac X2&\dfrac V2\\\\\dfrac X2& Y&\dfrac Z2\\\\\dfrac V2&\dfrac Z2&1\end{bmatrix}.
\plabel{eq:G2ad}
\end{equation}

Then
\begin{equation}
\m Q^{\mathbb R}_{\mathrm{pCK}}(A) =-4\adj \m G^{\mathbb R}_{\mathrm{pCK}}(A)
.
\plabel{eq:Gadj}
\end{equation}

Here
\begin{align}
\det\m G_{\mathrm{pCK}}^{\mathbb R}(A)&
={\frac1{16} { \left(   {Z}^{3}-W{Z}^{2}+4\,VXZ-4\,YZ-4\,{V}^{2}Y+4\,WY-4\,{X}^{2} \right) }}
\plabel{eq:detGR22}
\\\notag
&=\frac14\left((\tr A^*A)-2(\Rea\det A)\right) \left((U_A)^2-|D_A|^2\right)
\\\notag
&=\frac12\left(U_A-|D_A|+2E_A  \right) \left((U_A)^2-|D_A|^2\right).
\end{align}
This is strictly positive if $A$ is non-normal, and it is zero if $A$ is normal.

Furthermore,
\begin{align}
\det\m Q_{\mathrm{pCK}}^{\mathbb R}(A)
&=-64\left( \det \m G_{\mathrm{pCK}}^{\mathbb R}(A)\right)^2
\plabel{eq:detCR}
\\\notag
&=-{\frac1{4} { \left(   {Z}^{3}-W{Z}^{2}+4\,VXZ-4\,YZ-4\,{V}^{2}Y+4\,WY-4\,{X}^{2} \right) ^{2}}}
\\\notag
&=-4\left((\tr A^*A)-2(\Rea\det A)\right)^2 \left((U_A)^2-|D_A|^2\right)^2
\\\notag
&=-16\left(U_A-|D_A|+2E_A  \right)^2 \left((U_A)^2-|D_A|^2\right)^2.
\end{align}
This is strictly negative if $A$ is non-normal, and it is zero if $A$ is normal.

Moreover,
\begin{equation}
\adj \m Q_{\mathrm{pCK}}^{\mathbb R}(A)
=\left(16\det\m G_{\mathrm{pCK}}^{\mathbb R}(A)\right)\cdot\m G_{\mathrm{pCK}}^{\mathbb R}(A).
\plabel{eq:chr1}
\end{equation}
\end{theorem}

\begin{proof}[Proof for Theorem \ref{thm:CRarit}]
Equation \eqref{eq:Gadj} is simple to compute, and so is the first line of \eqref{eq:detGR22}.
The rest of $\eqref{eq:detGR22}$ is somewhat computational but straightforward:
The point is that both sides can be transcribed to polynomials in the five data.
Indeed,
\begin{equation}
\begin{matrix}
\text{ $U_A-|D_A|+2E_A$, $(U_A+|D_A|)+(U_A-|D_A|)$, and $(U_A+|D_A|)(U_A-|D_A|)$},
\\
\text{are all polynomials in the five data.}
\end{matrix}
\plabel{eq:polfive}
\end{equation}
Taking the determinant for \eqref{eq:Gadj}  implies (\ref{eq:detCR}/1), which leads to the rest of \eqref{eq:detCR}.
In the determinants the critical factors are $U_A-|D_A|+2E_A, U_A+|D_A|,U_A-|D_A|$.
These are all greater or equal to $U_A-|D_A|\geq0$, which vanishes exactly in the normal case.
Applying $\adj$ to \eqref{eq:Gadj}, we obtain \eqref{eq:chr1}.
\end{proof}
\snewpage
\begin{proof}[Alternative proof for Theorem \ref{thm:CRarit}]
For a relatively computation-free argument, let us continue the previous proof of Theorem \ref{thm:CR}
by projections (for the non-normal case).
After having method (ii) carried out, method (iii) must give \textit{exactly} the same matrix
(cf. Appendix \ref{sec:algproj}), thus equality in \eqref{eq:eqeqCR}  must hold.
By the relation to the inverse, $\m Q^{\mathbb R}_{\mathrm{pCK}}(A) $ and
$\adj \m G^{\mathbb R}_{\mathrm{pCK}}(A) $ are proportional.
After checking for position $(3,3)$, equality \eqref{eq:Gadj} follows.
Taking determinant, this implies (\ref{eq:detCR}/1).
Now, (\ref{eq:detCR}/1) and the equality in \eqref{eq:eqeqCR} imply (\ref{eq:detGR22}/2).
Then (\ref{eq:detGR22}/3) is just a minor variant, while (\ref{eq:detGR22}/1) follows from a simple direct computation.
Equations \eqref{eq:detGR22} and (\ref{eq:detCR}/1)  imply the rest of \eqref{eq:detCR}.
Applying $\adj$ to \eqref{eq:Gadj}, we obtain \eqref{eq:chr1}.
Non-degeneracy and (\ref{eq:detCR}/1)  imply $\det \m Q^{\mathbb R}_{\mathrm{pCK}}(A)<0$.
This and equality in \eqref{eq:eqeqCR} imply $\det \m G^{\mathbb R}_{\mathrm{pCK}}(A)>0$.
In the normal case the equalities extend by continuity; the quadratic forms must be degenerate with determinants $0$.
\end{proof}

Next we will consider other arguments for the proof of Theorem \ref{thm:CR}.
As there is a general argument for the location the interior,
we will mostly restrict only to arguments identifying the boundary of conformal range.

\begin{proof}[A proof of Theorem \ref{thm:CR} via the numerical range.]
Let \[\mathrm D^{\mathbb R}(A)=\frac{A+A^*}2+\mathrm iA^*A.\]
Then, one can see that
\[\DW^{\mathbb R}_{\mathrm{pCK}}(A)=\mathrm W\left(\mathrm D^{\mathbb R}(A)\right).\]
Now, one can see
\begin{align*}
\Rea\tr\mathrm D^{\mathbb R}(A)&=V , \\
\Ima\tr\mathrm D^{\mathbb R}(A)&=Z , \\
\Rea\det\mathrm D^{\mathbb R}(A)&=-\frac Z4+\frac W4-Y , \\
\Ima\det\mathrm D^{\mathbb R}(A)&=X, \\
\tr\mathrm D^{\mathbb R}(A)^*\mathrm D^{\mathbb R}(A)&=V^2+ Z^2+\frac Z2-\frac W2-2Y.
\end{align*}
These can be plugged into Theorem \ref{thm:numrange}, yielding the statement.
\end{proof}
\begin{remark}
In fact, we find explicitly,
\[\m Q^{\mathbb R}_{\mathrm{pCK}}(A)=\m Q_{\mathrm W}\left( \mathrm D^{\mathbb R}(A)\right) \]
and
\[\m G^{\mathbb R}_{\mathrm{pCK}}(A)=\m G_{\mathrm W}\left( \mathrm D^{\mathbb R}(A)\right). \]

 Through
$\left(U_{\mathrm D^{\mathbb R}(A)}\right)^2-\left|D_{\mathrm D^{\mathbb R}(A)}\right|^2=2
\left(U_A-|D_A|\right)\left(U_A+|D_A|\right)\left(U_A-|D_A|+2|E_A|\right)$,
(\ref{eq:Gadj}--\ref{eq:chr1}) corresponds to (\ref{eq:rao1}--\ref{eq:rao4}).
\qedremark
\end{remark}
\snewpage

The following arguments will use convexity, with and without differential geometry.

\begin{proof}[A proof of Theorem \ref{thm:CR} via the standard enveloping construction.]
One can use the enveloping construction as described in \cite{L2} in order to show that
\begin{equation}
\begin{bmatrix}x_{\mathrm{pCK}}\\ z_{\mathrm{pCK}}\\1\end{bmatrix}^\top
\m Q_{\mathrm{pCK}}^{\mathbb R}(A)
\begin{bmatrix}x_{\mathrm{pCK}}\\ z_{\mathrm{pCK}}\\1\end{bmatrix}=0
\end{equation}
does indeed describe the boundary.
This goes as follows.
One can see that, regarding the norm and the co-norm,
\begin{equation}
\left(\|A-\lambda\Id_2\|_2^\pm\right)^2=\lambda^2-\lambda V+\frac Z2
\pm
\frac12\sqrt{
\begin{bmatrix}1\\2\lambda\end{bmatrix}^\top
\begin{bmatrix}{Z^2-4Y}&{2X-VZ}\\
{2X-VZ}&{V^2-W+Z}\end{bmatrix}
\begin{bmatrix}1\\2\lambda\end{bmatrix}
}
\plabel{eq:gtran}
\end{equation}
(sign $+$ for the  norm branch, sign $-$ for the co-norm branch).
Let us use the abbreviation `$N$' for \eqref{eq:gtran}.
For the enveloping curve, written in column vector form, it yields
\begin{align*}
E_{\mathrm{pCK}}^A(\lambda)&=
\bem
\lambda-\dfrac12 \dfrac{\mathrm d N}{\mathrm d\lambda}
\\\\
\lambda^2-\lambda\dfrac{\mathrm d N}{\mathrm d\lambda}+ N
\eem
\\&
=\frac12\left(
\begin{bmatrix} V\\\\ Z\end{bmatrix}
\pm
\frac{
\begin{bmatrix}&-1\\1&\end{bmatrix}
\begin{bmatrix}{Z^2-4Y}&{2X-VZ}\\
{2X-VZ}&{V^2-W+Z}\end{bmatrix}
\begin{bmatrix}1\\2\lambda\end{bmatrix}
}{
\sqrt{
\begin{bmatrix}1\\2\lambda\end{bmatrix}^\top
\begin{bmatrix}{Z^2-4Y}&{2X-VZ}\\
{2X-VZ}&{V^2-W+Z}\end{bmatrix}
\begin{bmatrix}1\\2\lambda\end{bmatrix}
}
}\right).
\end{align*}
According to standard facts regarding quadrics,  this traces out (two open halves of) the
quadric  with matrix
\[\begin{bmatrix}1&&\frac V2\\&1&\frac Z2\\&&1\end{bmatrix}^{-1,\top}
\begin{bmatrix}{Z^2-4Y}&{2X-VZ}\\
{2X-VZ}&{V^2-W+Z}
\\&&
-\frac14\det\left[\begin{smallmatrix}{Z^2-4Y}&{2X-VZ}\\
{2X-VZ}&{V^2-W+Z}\end{smallmatrix}\right]
\end{bmatrix}
\begin{bmatrix}1&&\frac V2\\&1&\frac Z2\\&&1\end{bmatrix}^{-1}.
\]
Expanded, this gives the matrix $\m Q_{\mathrm{pCK}}^{\mathbb R}(A)$ indicated.
\end{proof}

\snewpage
\begin{proof}[A proof of Theorem \ref{thm:CR} via the rotational enveloping construction.]
The rotational enveloping construction in this context works out as follows.
Let us use the abbreviation
\[\m S=\left[ \begin {array}{cc} {Z}^{2}-4\,Y&-2\,YV-VZ+XZ+2\,X
\\ \noalign{\medskip}-2\,YV-VZ+XZ+2\,X&{V}^{2}+2\,VX-WY-WZ+{X}^{2}+YZ+
{Z}^{2}-W+Z\end {array} \right].
\]
Then we find
%\begin{equation}
\begin{multline}
\left(\left\|\dfrac{(\cos\frac\omega2)A-(\sin\frac\omega2)\Id}{
(\sin\frac\omega2)A+(\cos\frac\omega2)\Id}\right\|_2\right)^2
=\\=
\frac{1+Y+Z
+
\bem Y-1\\-V-X\eem^\top\bem\cos\omega\\\sin\omega\eem
+
\sqrt{
\begin{bmatrix}\cos\omega\\\sin\omega\end{bmatrix}
\m S
\begin{bmatrix}\cos\omega\\\sin\omega\end{bmatrix}
}
}{
1+Y+Z-\bem Y-1\\-V-X\eem^\top\bem\cos\omega\\\sin\omega\eem
-
\sqrt{
\begin{bmatrix}\cos\omega\\\sin\omega\end{bmatrix}
\m S
\begin{bmatrix}\cos\omega\\\sin\omega\end{bmatrix}
}
}.
\plabel{eq:ntran}
%\end{equation}
\end{multline}
(In contrast to \eqref{eq:gtran}, it is hard to compute \eqref{eq:ntran}.
%this is much harder to compute.
Canonical triangular form is advised.)
Let use the abbreviation  `$\widehat N$' for \eqref{eq:ntran}.
Then, the enveloping curve giving the boundary of the conformal range is
\begin{align*}
\widehat E^A_{\mathrm{BCK}}(\omega)&=\frac1{(1+\widehat N)^2}
\bem-2\dfrac{\mathrm d\widehat N}{\mathrm d\omega}&1-\widehat N^2\\-1+\widehat N^2&-2\dfrac{\mathrm d\widehat N}{\mathrm d\omega}\eem
\bem\cos\omega\\\sin\omega\eem
\\
&=\frac1{1+Y+Z}\left(
\begin{bmatrix}
{V+X}\\\\{Y-1}
\end{bmatrix}+
\frac{{
\begin{bmatrix}&-1\\1&\end{bmatrix}
\m S
\begin{bmatrix}\cos\omega\\\sin\omega\end{bmatrix}
}}{\sqrt{
\begin{bmatrix}\cos\omega\\\sin\omega\end{bmatrix}
\m S
\begin{bmatrix}\cos\omega\\\sin\omega\end{bmatrix}
}}
\right)
\end{align*}
(written in column vector form).
This traces the quadric with matrix
\[\begin{bmatrix}1&&\frac{V+X}{1+Y+Z}\\&1&\frac{Y-1}{1+Y+Z}\\&&1\end{bmatrix}^{-1,\top}
\begin{bmatrix}\m S&\\&-\frac1{(1+Y+Z)^2}\det\m S\end{bmatrix}
\begin{bmatrix}1&&\frac{V+X}{1+Y+Z}\\&1&\frac{Y-1}{1+Y+Z}\\&&1\end{bmatrix}^{-1}.
\]
Expanded, this yields the matrix $\m Q_{\mathrm{BCK}}^{\mathbb R}(A)$.
\end{proof}
Note that the previous two arguments are close variants of each other.
Yet, the second one is much more computational than the first one.
\snewpage
\begin{proof}[A proof of Theorem \ref{thm:CR} via the algebraic enveloping construction.]
Consider
\begin{multline*}
F_A^{\CR}(\lambda,\nu)\equiv\det(\nu\Id_2-(A-\lambda\Id_2)^*(A-\lambda\Id_2))=
\\=\nu^2-(2\lambda^2-2V\lambda+Z)\nu+(\lambda^4-2V\lambda^3+W\lambda^2-2X\lambda+Y).
\end{multline*}
This is irreducible, as otherwise it would decompose to two linear factors in $\nu$ contradicting to non-normality.
Then
\begin{multline*}
F_A^{\CR}(\lambda,\lambda^2-2x\lambda+z)=\\= \left(4\,{x}^{2}-4\,Vx+W-Z \right) {\lambda}^{2}
+ \left(4\,xz+2\,Zx+ 2\,Vz-2\,X \right) \lambda
+ \left({z}^{2}-Zz+Y\right).
\end{multline*}
(Here $x$ and $z$ should be understood as $x_{\mathrm{pCK}}$ and $z_{\mathrm{pCK}}$.)
Then one directly finds that the discriminant of this polynomial is
\begin{equation*}
4\cdot\begin{bmatrix}x_{\mathrm{pCK}}\\ z_{\mathrm{pCK}}\\1\end{bmatrix}^\top
\m Q_{\mathrm{pCK}}^{\mathbb R}(A)
\begin{bmatrix}x_{\mathrm{pCK}}\\ z_{\mathrm{pCK}}\\1\end{bmatrix}.
\qedhere
\end{equation*}
\end{proof}

\begin{proof}[A proof of Theorem \ref{thm:CR}
in a more Kippenhahnian formalism.]
We consider the projective homogeneous polynomial for the dual curve:
\begin{align*}
K^{\CR}_A&(u,s,w)\equiv\\
&\equiv\det\left(u\frac{A+A^*}{2}+sA^*A+ w\Id\right)=
\\&=\frac12\left(\tr \left(u\frac{A+A^*}{2}+sA^*A+ w\Id\right)\right)^2-
\frac12\tr\left( \left(u\frac{A+A^*}{2}+sA^*A+ w\Id\right)^2\right)
\\&=\frac{W-Z}4u^2+Xus+Vuw+Ys^2+Zsw+w^2
\\&\equiv\begin{bmatrix}u\\s \\ w\end{bmatrix}^\top
\m G_{\mathrm{pCK}}^{\mathbb R}(A)
\begin{bmatrix}u \\ s \\ w\end{bmatrix}.
\end{align*}
(Cf. $\m G_{\mathrm{pCK}}^{\mathbb R}(A)=\m G_{\mathrm{pCK}}(A)|_{\{1,3,4\}}$.)
As it is known from  projective geometry, the matrix of the dual of the conic is given by the adjugate matrix.
In this case this is
 $\mathrm{adj}\,\,\m G_{\mathrm{pCK}}^{\mathbb R}(A)=-\frac14 \m Q_{\mathrm{pCK}}^{\mathbb R}(A)$.
The scalar factor, of course, can be modified.
\end{proof}
From \cite{L2}, we know that the previous two arguments are essentially synonymous.
Yet, in practice, they look quite different.
The first one looks like the algebraic variant of the differential geometric argument
using discriminant instead of differentiation.
The second one is the restriction of the Kippenhahnian argument from the case of the shell,
or version (iii) of the projection argument; invoking only a moderate amount of projective geometry.
(And, by this, we have constructed a connected graph of essential equivalences between the various proofs above.)

\snewpage

Further observations
in the line of Theorem \ref{thm:CRarit}
%regarding $\m Q_{\mathrm{pCK}}^{\mathbb R}(A)$ and $\m G_{\mathrm{pCK}}^{\mathbb R}(A)$
are as follows.
Compatible to the scheme of \eqref{eq:trans2}, we set
\[\m Q^{\mathbb R,0}_{\mathrm{pCK}}=
\begin{bmatrix}1&&\\&&-\frac12\\&-\frac12&\end{bmatrix},
\qquad\text{and}\qquad
\m Q^{\mathbb R,0}_{\mathrm{BCK}}=
\begin{bmatrix}1&&\\&1&\\&&-1\end{bmatrix}.\]

On the other hand, we define
\begin{equation}
\m G_{\mathrm{BCK}}^{\mathbb R}(A)
=
\begin{bmatrix}
1&&\\&1&1\\&-1&1
\end{bmatrix}^{-1}
\m G_{\mathrm{pCK}}^{\mathbb R}(A)
\begin{bmatrix}
1&&\\&1&1\\&-1&1
\end{bmatrix}^{-1,\top}
.
\plabel{eq:trans355}
\end{equation}

We also use the shorthand notation
\[\m G^{\mathbb R,0}_{\mathrm{pCK}}=\bigl(\m Q^{\mathbb R,0}_{\mathrm{pCK}}\bigr)^{-1},
\qquad\text{and}\qquad
\m G^{\mathbb R,0}_{\mathrm{BCK}}=\bigl(\m Q^{\mathbb R,0}_{\mathrm{BCK}}\bigr)^{-1}.
\]

Then, it is easy to see that
$\m G^{\mathbb R,0}_{\mathrm{pCK}} \m Q^{\mathbb R}_{\mathrm{pCK}}(A) $
$\sim$
$\m G^{\mathbb R,0}_{\mathrm{BCK}} \m Q^{\mathbb R}_{\mathrm{BCK}}(A) $
and
$\m G^{\mathbb R}_{\mathrm{BCK}}(A) \m Q^{\mathbb R,0}_{\mathrm{BCK}}$
$\sim$
$\m G^{\mathbb R}_{\mathrm{pCK}}(A) \m Q^{\mathbb R,0}_{\mathrm{pCK}}$.
Applying \eqref{eq:adj1} to \eqref{eq:trans355}, we obtain as a variant of \eqref{eq:Gadj},
\[\m Q^{\mathbb R}_{\mathrm{BCK}}(A) =-16\adj\m G^{\mathbb R}_{\mathrm{BCK}}(A);\]
moreover,
\begin{equation}
\m G^{\mathbb R,0}_{\mathrm{pCK}} \m Q^{\mathbb R}_{\mathrm{pCK}}(A)=
16\adj\left(\m G^{\mathbb R}_{\mathrm{pCK}}(A)\m Q^{\mathbb R,0}_{\mathrm{pCK}}\right),
\plabel{eq:GCadj}
\end{equation}
and
\[\m G^{\mathbb R,0}_{\mathrm{BCK}} \m Q^{\mathbb R}_{\mathrm{BCK}}(A)=
16\adj\left(\m G^{\mathbb R}_{\mathrm{BCK}}(A)\m Q^{\mathbb R,0}_{\mathrm{BCK}}\right).\]

\begin{theorem}\plabel{thm:CReigen}

(a)
The eigenvalues of $\m G^{\mathbb R}_{\mathrm{pCK}}(A)\m Q^{\mathbb R,0}_{\mathrm{pCK}}$
or  $\m G^{\mathbb R}_{\mathrm{BCK}}(A)\m Q^{\mathbb R,0}_{\mathrm{BCK}}$ are
\[
-\frac12(U_A-|D_A|),\qquad
-\frac12(U_A+|D_A|),\qquad
-\frac12(U_A-|D_A|+2E_A).
\]

(b) The eigenvalues of $\m G^{\mathbb R,0}_{\mathrm{pCK}} \m Q^{\mathbb R}_{\mathrm{pCK}}(A) $
or $\m G^{\mathbb R,0}_{\mathrm{BCK}} \m Q^{\mathbb R}_{\mathrm{BCK}}(A) $,
are
\begin{multline*}
4(U_A-|D_A|)(U_A+|D_A|),
\\
 4(U_A-|D_A|)(U_A-|D_A|+2E_A),
\\
 4(U_A+|D_A|)(U_A-|D_A|+2E_A).
\end{multline*}

\begin{proof}
(a) Direct computation, cf. \eqref{eq:polfive}.
(b) follows from \eqref{eq:GCadj}.
\end{proof}
\end{theorem}

\snewpage
\begin{example}\plabel{ex:checkGQ}
(a) Regarding $L_{\alpha,t}^\pm$:
\[\m G^{\mathbb R}_{\mathrm{BCK}}(L_{\alpha,t}^\pm )\m Q^{\mathbb R,0}_{\mathrm{BCK}}=
\begin{bmatrix}
-((\cos\alpha)^2+t^2)&&\\
&-t^2&\\
&&-(1+t^2)
\end{bmatrix}
,
\]
and
\[\m G^{\mathbb R,0}_{\mathrm{BCK}} \m Q^{\mathbb R}_{\mathrm{BCK}}(L_{\alpha,t}^\pm )=
\begin{bmatrix}
16t^2(1+t^2)&&\\
&16((\cos\alpha)^2+t^2)(1+t^2)&\\
&&16t^2((\cos\alpha)^2+t^2)
\end{bmatrix}.
\]

(b) Regarding $S_{\beta}$:
\[\m G^{\mathbb R}_{\mathrm{pCK}}(S_\beta )\m Q^{\mathbb R,0}_{\mathrm{pCK}}=
\begin{bmatrix}
-\frac14(\cos\beta)^2&&\\
&-\frac14&\\
&-\frac12&-\frac14
\end{bmatrix}
,
\]
and
\[\m G^{\mathbb R,0}_{\mathrm{pCK}}\m Q^{\mathbb R}_{\mathrm{pCK}}(S_\beta )=
\begin{bmatrix}
1&&\\
&(\cos\beta)^2&\\
&-2(\cos\beta)^2&(\cos\beta)^2
\end{bmatrix}
.
\]

(c) In case of $\m 0_2$:
\[\m G^{\mathbb R}_{\mathrm{pCK}}(\m 0_2 )\m Q^{\mathbb R,0}_{\mathrm{pCK}}=
\begin{bmatrix}
0&&\\
&0&\\
&-\frac12&0
\end{bmatrix}
,
\]
and
\[\m G^{\mathbb R,0}_{\mathrm{pCK}}\m Q^{\mathbb R}_{\mathrm{pCK}}(\m 0_2 )=
\m 0_3
.
\]

For all the matrices above, the eigenvalues can be read off from the diagonals.
\qedexer
\end{example}
\snewpage
\begin{commenty}
\begin{point}
\plabel{disc:UU}
Using the notation \eqref{eq:reducedfive}, we can define
\[U_1(A)={3\,Z-W},\]
\[U_2(A)={3\,{Z}^{2}-2\,WZ+4\,VX-4\,Y},\]
\[U_3(A)={{Z}^{3}-W{Z}^{2}+4\,VXZ-4\,YZ-4\,{V}^{2}Y+4\,WY-4\,{X}^{2}}.\]
Here, the expressions are related to
\[
\det\left(\lambda\Id_3- \m G^{\mathbb R}_{\mathrm{pCK}}(A) \m Q^{\mathbb R,0}_{\mathrm{pCK}}  \right)
=\lambda^3+\frac{U_1(A)}4\lambda^2+\frac{U_2(A)}{16} \lambda+\frac{U_3(A)}{64},\]
and
\[
\det\left(\lambda\Id_3- \m G^{\mathbb R,0}_{\mathrm{pCK}}  \m Q^{\mathbb R}_{\mathrm{pCK}}(A)  \right)
=\lambda^3-U_2(A)\lambda^2+U_1(A)U_3(A)\lambda-(U_3(A))^2.\]

Then $U_1(A)$, $U_2(A)$, $U_3(A)$ are elementary symmetric polynomials of
$2(U_A-|D_A|)$, $2(U_A+|D_A|)$, $2(U_A-|D_A|+2E_A)$.
It is easy to see that
\begin{equation}
U_1(A)=
6(U_A-|D_A|)+4|D_A|+4E_A\geq 0,
\plabel{eq:UU1}
\end{equation}
with equality if and only if $A$ is normal and real-parabolic, i. e.
it is a real scalar matrix.
Similarly,
\begin{equation}
U_2(A)=
12(U_A-|D_A|)^2 +16(U_A-|D_A|)\left(|D_A|+E_A\right)+16|D_A|E_A\geq 0,
\plabel{eq:UU2}
\end{equation}
with equality if and only if $A$ is normal with two equal or conjugate eigenvalues.
Moreover,
\begin{equation}
U_3(A)=8
(U_A-|D_A|)(U_A+|D_A|)(U_A-|D_A|+2E_A)\geq 0,
\plabel{eq:UU3}
\end{equation}
with equality if and only if $A$ is normal.
\end{point}
\end{commenty}
\begin{commentx}
\begin{example}\plabel{ex:UU}
(a) Regarding $L_{\alpha,t}^\pm$:
\[U_1(L_{\alpha,t}^\pm)=4\left(1+(\cos\alpha)^2 +3t^2\right),\]
\[U_2(L_{\alpha,t}^\pm)=16\left((\cos\alpha)^2+2t^2 +2(\cos\alpha)^2t^2+3t^4 \right),\]
\[U_3(L_{\alpha,t}^\pm)=64t^2(1+t^2)\left( (\cos\alpha)^2 +t^2\right).\]

(b) Regarding $S_{\beta}$:
\[U_1(S_{\beta})={2+(\cos\beta )^2},\]
\[U_2(S_{\beta})={1+2(\cos\beta )^2},\]
\[U_3(S_{\beta})={(\cos\beta )^2}.\]

(c) In case of $\m 0_2$, everything is $0$.
\qedexer
\end{example}
\end{commentx}

\snewpage
\subsection{Decomposions of the quadratic forms}

\begin{lemma}
\plabel{lem:CRcore0}
Consider
\[\m Q_{\mathrm{pCK}}^{\mathbb R}(A)|_{\{1,2\} }=
\begin{bmatrix}{Z^2-4Y}&{2X-VZ}\\
{2X-VZ}&{V^2-W+Z}\end{bmatrix}. \]

Then
\begin{align}
\det\m Q_{\mathrm{pCK}}^{\mathbb R}(A)|_{\{1,2\} } &=
16 \det\m G_{\mathrm{pCK}}^{\mathbb R}(A)
\plabel{eq:core0}
\\\notag
&=  {Z}^{3}-W{Z}^{2}+4\,VXZ-4\,YZ-4\,{V}^{2}Y+4\,WY-4\,{X}^{2}
\\\notag
&=8\left(U_A-|D_A|+2E_A  \right) \left((U_A)^2-|D_A|^2\right).
\end{align}

Regarding its rank,
\[\rank \m Q_{\mathrm{pCK}}^{\mathbb R}(A)|_{\{1,2\} }=
\begin{cases}
2&\text{if $A$ is non-normal,}\\
1&\text{if $A$ is normal without equal or conjugate eigenvalues,}\\
0&\text{if $A$ is normal with equal or conjugate eigenvalues.}
\end{cases} \]
\begin{proof}
Specifying \eqref{eq:chr1} to entry $(3,3)$, we obtain (\ref{eq:core0}/1),
or by direct computation (\ref{eq:core0}/2).
The rest of \eqref{eq:core0} follows from \eqref{eq:detGR22}.
This determinant is non-vanishing if and only if $A$ is non-normal.
In the normal case the rank can be $0$ or $1$.
If the canonical from \eqref{eq:canon2} is considered, then
\[Z^2-4Y=t^4 +2t^2(|\lambda_1|^2+|\lambda_1|^2)+(|\lambda_1|^2-|\lambda_1|^2)^2,\]
\[2X-VZ=-t^2(\Rea\lambda_1+\Rea\lambda_2)-(\Rea\lambda_1-\Rea\lambda_2)(|\lambda_1|^2-|\lambda_2|^2),\]
\[V^2-W+Z=t^2+(\Rea \lambda_1-\Rea \lambda_2)^2.\]
This shows that rank $0$ is achieved if and only if $t=0$ (normality) and $\Rea\lambda_1-\Rea\lambda_2=0$
and $|\lambda_1|^2-|\lambda_2|^2=0$ holds.
The latter two conditions together mean equal or conjugate eigenvalues.
\end{proof}
\end{lemma}
\begin{commentx}
\begin{remark}\plabel{rem:CRcore0}
One can find a matrix $S$ of determinant $1$ such that
$S^{\top}\m Q_{\mathrm{pCK}}^{\mathbb R}(A)|_{\{1,2\}} S$
has eigenvalues $2\left(U_A-|D_A|  \right) $ and $4(U_A+|D_A|)(U_A-|D_A|+2|E_A|)$.
But, as the trace shows, then $S$ is not polynomial in the five data.
\qedremark
\end{remark}
\end{commentx}
\snewpage
\begin{lemma}
\plabel{lem:CRcore1}
(a)
\[\m G_{\mathrm{pCK}}^{\mathbb R}(A)=
\bem1&&\frac12 V\\&1&\frac12 Z\\&&1\eem
\bem-\frac14\adj\left(\m Q_{\mathrm{pCK}}^{\mathbb R}(A)|_{\{1,2\}}\right)&\\&1\eem
\bem1&&\frac12 V\\&1&\frac12 Z\\&&1\eem^\top.
\]

(b)
\[\rank \m G_{\mathrm{pCK}}^{\mathbb R}(A) =
\begin{cases}
3&\text{if $A$ is non-normal,}\\
2&\text{if $A$ is normal without equal or conjugate eigenvalues,}\\
1&\text{if $A$ is normal with equal or conjugate eigenvalues.}
\end{cases} \]
\end{lemma}
\begin{lemma}
\plabel{lem:CRcore2}
(a)
\begin{multline*}
\m Q_{\mathrm{pCK}}^{\mathbb R}(A)=\\=
\bem1&&\frac12 V\\&1&\frac12 Z\\&&1\eem^{-1,\top}
\bem\m Q_{\mathrm{pCK}}^{\mathbb R}(A)|_{\{1,2\}}&\\&
-\frac14\det \m Q_{\mathrm{pCK}}^{\mathbb R}(A)|_{\{1,2\}}  \eem
\bem1&&\frac12 V\\&1&\frac12 Z\\&&1\eem^{-1}.
\end{multline*}

(b)
\[\rank \m Q_{\mathrm{pCK}}^{\mathbb R}(A) =
\begin{cases}
3&\text{if $A$ is non-normal,}\\
1&\text{if $A$ is normal without equal or conjugate eigenvalues,}\\
0&\text{if $A$ is normal with equal or conjugate eigenvalues.}
\end{cases} \]
\end{lemma}
\begin{proof}[Proofs]
(a) These are simple computations.
(b) These follow from Lemma \ref{lem:CRcore0}.
\end{proof}
\begin{lemma}
\plabel{lem:hypfocal}
\[\bem-2\lambda\\ 1\\ \lambda^2\eem^\top
\m G_{\mathrm{pCK}}^{\mathbb R}(A)
\bem-2\lambda\\ 1\\ \lambda^2\eem=\det(\lambda\Id_2-A)\cdot\det(\lambda\Id_2-A^*).\]
\begin{proof}
Direct computation. (See \cite{L2} for greater generality.)
\end{proof}
\end{lemma}
%\snewpage
\begin{lemma}
\plabel{lem:CRspec}
Assume that the $2\times2$ complex matrix $A$ is non-normal, and
$\m M$ is the matrix of the quadric of the boundary of $\DW_{\mathrm{pCK}}^{\mathbb R}(A)$.
Then

(a)
\begin{equation}
\m G_{\mathrm{pCK}}^{\mathbb R}(A)=\frac{\m M^{-1}}{(\m M^{-1})_{33}}=\frac{\adj\m M}{\det\m M|_{\{1,2\}}};
\plabel{eq:CRGrec}
\end{equation}

(b)
\[\m Q_{\mathrm{pCK}}^{\mathbb R}(A)=\frac{-4}{(\det\m M)((\m M^{-1})_{33})^2}\cdot\m M
=\frac{-4\det\m M}{ (\det\m M|_{\{1,2\}})^2} \cdot\m M
=\frac{-4\adj\,\adj\m M}{\det\m M|_{\{1,2\}})^2}
.\]
\begin{proof}
(a) follows from $\m G_{\mathrm{pCK}}^{\mathbb R}(A)\vvarpropto\m M^{-1}$ and $\m G_{\mathrm{pCK}}^{\mathbb R}(A)_{33}=1$.
Then \eqref{eq:Gadj} implies (b).
\end{proof}
\end{lemma}
\snewpage
The following discussions are about connecting the conformal range to geometric data.
First (in terms of hyperbolic conics) we relate to pairs of real and imaginary foci.
(The foci are duals of some pair of lines.  Point ellipses are pairs of imaginary lines.)
If the reader is not inclined to think in terms of this very old viewpoint
(see Story \cite{Sto}), then these discussions can be skipped.
\begin{disc}
\plabel{disc:GR}
Let $\m G_{\mathrm{pCK}}^{\mathbb R,\mathrm{spec}}(A)=
 \m G_{\mathrm{pCK}}^{\mathrm{spec}}(A)|_{\{1,3,4\}\times\{1,3,4\}}$.
Then
\[ \m G_{\mathrm{pCK}}^{\mathbb R}(A)=
 U_A\left(-\frac12 \m G_{\mathrm{pCK}}^{\mathbb R,0}\right)+\m G_{\mathrm{pCK}}^{\mathbb R,\mathrm{spec}}(A).\]
By Theorem \ref{thm:CReigen}(a), the singular quadrics in the pencil of generated by
$\left(-\frac12 \m G_{\mathrm{pCK}}^{\mathbb R,0}\right)$
and $ \m G_{\mathrm{pCK}}^{\mathbb R}(A)$ are
\[ \m G_{\mathrm{pCK}}^{\mathbb R,\mathrm{eig}}(A)=
 |D_A|\left(-\frac12 \m G_{\mathrm{pCK}}^{\mathbb R,0}\right)+\m G_{\mathrm{pCK}}^{\mathbb R,\mathrm{spec}}(A),\]
\[ \m G_{\mathrm{pCK}}^{\mathbb R,\mathrm{ceig}}(A)=
 (|D_A|-2\min(|D_A|,E_A))\left(-\frac12 \m G_{\mathrm{pCK}}^{\mathbb R,0}\right)
 +\m G_{\mathrm{pCK}}^{\mathbb R,\mathrm{spec}}(A),\]
\[ \m G_{\mathrm{pCK}}^{\mathbb R,\mathrm{cx}}(A)=
 (|D_A|-2\max(|D_A|,E_A))\left(-\frac12 \m G_{\mathrm{pCK}}^{\mathbb R,0}\right)
 +\m G_{\mathrm{pCK}}^{\mathbb R,\mathrm{spec}}(A).\]

Assume that $A$ has eigenvalues $\lambda_1,\lambda_2$.
Then, after some computation, we find that
\[\m G_{\mathrm{pCK}}^{\mathbb R,\mathrm{eig}}(A)=\frac12\left(
\tilde{\m x}_1\tilde{\m x}_2^\top+\tilde{\m x}_2\tilde{\m x}_1^\top\right),\]
where
\[\tilde{\m x}_1=\bem\Rea \lambda_1\\|\lambda_1|^2\\1\eem
\qquad\text{and}\qquad
\tilde{\m x}_2=\bem\Rea \lambda_2\\|\lambda_2|^2\\1\eem.\]
If $\lambda_1=\lambda_2$ or  $\lambda_1=\bar\lambda_2$  then this is double line on the dual projective space;
 otherwise it is pair of distinct line on the dual projective space.
~
Similarly,
\[\m G_{\mathrm{pCK}}^{\mathbb R,\mathrm{ceig}}(A)
=\frac12\left(\tilde{\m x}_g\tilde{\m x}_g^\top+\tilde{\m y}_g\tilde{\m y}_g^\top\right),\]
where
\[\tilde{\m x}_g=\bem\frac12((\Rea \lambda_1)+(\Rea \lambda_2))\\ (\Rea \lambda_1)(\Rea \lambda_2)+|\Ima \lambda_1|\,|\Ima \lambda_2|\\1\eem
\quad\text{and}\quad
\tilde{\m y}_g=\bem\frac12(|\Ima\lambda_1|-|\Ima\lambda_2|)\\|\Ima \lambda_1|(\Rea \lambda_2)-(\Rea \lambda_1)|\Ima \lambda_2|\\0\eem.\]
If $\lambda_1=\lambda_2$ or $\lambda_1=\bar\lambda_2$ or  $\Ima\lambda_1=\Ima\lambda_2=0$ then this is double line on the dual projective space;
 otherwise it is a point ellipse on the dual projective space.
~
Furthermore,
\[\m G_{\mathrm{pCK}}^{\mathbb R,\mathrm{cx}}(A)
=\frac12\left(\tilde{\m x}_n\tilde{\m x}_n^\top+\tilde{\m y}_n\tilde{\m y}_n^\top\right),\]
where
\[\tilde{\m x}_n=\bem\frac12((\Rea \lambda_1)+(\Rea \lambda_2))\\ (\Rea \lambda_1)(\Rea \lambda_2)-|\Ima \lambda_1|\,|\Ima \lambda_2|\\1\eem
\quad\text{and}\quad
\tilde{\m y}_n=\bem\frac12(|\Ima\lambda_1|+|\Ima\lambda_2|)\\|\Ima \lambda_1|(\Rea \lambda_2)+(\Rea \lambda_1)|\Ima \lambda_2|\\0\eem.\]
If  $\Ima\lambda_1=\Ima\lambda_2=0$ then this is double line on the dual projective space;
 otherwise it is point ellipse on the dual projective space.
\end{disc}
%\snewpage
Next, axes and centers appear.
\begin{disc}
\plabel{disc:QR}
Let us define
\[\m Q_{\mathrm{pCK}}^{\mathbb R,\mathrm{majax}}(A)=-4\adj\m G_{\mathrm{pCK}}^{\mathbb R,\mathrm{eig}}(A),\]
\[\m Q_{\mathrm{pCK}}^{\mathbb R,\mathrm{minax}}(A)=-4\adj\m G_{\mathrm{pCK}}^{\mathbb R,\mathrm{ceig}}(A),\]
\[\m Q_{\mathrm{pCK}}^{\mathbb R,\mathrm{cen}}(A)=-4\adj\m G_{\mathrm{pCK}}^{\mathbb R,\mathrm{cx}}(A).\]
Let $A$ have eigenvalues $\lambda_1$, $\lambda_2$.
Then
\[\m Q_{\mathrm{pCK}}^{\mathbb R,\mathrm{majax}}(A)=\tilde{\m v}_e\tilde{\m v}_e^\top\]
where
\[\tilde{\m v}_e=\bem  |\lambda_1|^2-|\lambda_2|^2\\(\Rea\lambda_1)-(\Rea\lambda_2)\\
 (\Rea\lambda_1)|\lambda_2|^2-(\Rea\lambda_2)|\lambda_1|^2\eem
 =\bem\Rea \lambda_1\\|\lambda_1|^2\\1\eem\times\bem\Rea \lambda_2\\|\lambda_2|^2\\1\eem
 ;\]
 $\tilde{\m v}_e$ vanishes if $\lambda_1$ and $\lambda_2$ are equal or conjugate.
Furthermore,
\[\m Q_{\mathrm{pCK}}^{\mathbb R,\mathrm{minax}}(A)=-\tilde{\m v}_g\tilde{\m v}_g^\top\]
where
\[\tilde{\m v}_g=\bem  2(\Rea\lambda_1)|\Ima\lambda_2|-2(\Rea\lambda_2)|\Ima\lambda_1|\\ |\Ima\lambda_1|- |\Ima\lambda_2|\\
 |\Ima\lambda_1|\,|\lambda_2|^2-|\Ima\lambda_2|\,|\lambda_1|^2\eem
 =
 |\Ima\lambda_1|\bem    -2\Rea\lambda_2\\1\\|\lambda_2|^2\eem
 -|\Ima\lambda_2|\bem    -2\Rea\lambda_1\\1\\|\lambda_1|^2\eem;\]
 $\tilde{\m v}_g$ vanishes if $\lambda_1$ and $\lambda_2$ are equal or conjugate or both real.
Moreover,
\[\m Q_{\mathrm{pCK}}^{\mathbb R,\mathrm{cen}}(A)=-\tilde{\m v}_n\tilde{\m v}_n^\top\]
where
\[\tilde{\m v}_n=\bem  -2(\Rea\lambda_1)|\Ima\lambda_2|-2(\Rea\lambda_2)|\Ima\lambda_1|\\ |\Ima\lambda_1|+ |\Ima\lambda_2|\\
 |\Ima\lambda_1|\,|\lambda_2|^2+|\Ima\lambda_2|\,|\lambda_1|^2\eem
 =
 |\Ima\lambda_2|\bem    -2\Rea\lambda_1\\1\\|\lambda_1|^2\eem
 +|\Ima\lambda_1|\bem    -2\Rea\lambda_2\\1\\|\lambda_2|^2\eem.\]
 $\tilde{\m v}_n$ vanishes if $\lambda_1$ and $\lambda_2$ are both real.

Then
\[\tilde{\m v}_e^\top \m G_{\mathrm{pCK}}^{\mathbb R,0}(A)\tilde{\m v}_e=
\tr\m G_{\mathrm{pCK}}^{\mathbb R,0}\m Q_{\mathrm{pCK}}^{\mathbb R,\mathrm{majax}}(A)
=16|D_A|E_A\geq0,\]
\[\tilde{\m v}_g^\top \m G_{\mathrm{pCK}}^{\mathbb R,0}(A)\tilde{\m v}_g=
-\tr\m G_{\mathrm{pCK}}^{\mathbb R,0}\m Q_{\mathrm{pCK}}^{\mathbb R,\mathrm{minax}}(A)
=16|\Ima\lambda_1|\,|\Ima\lambda_2|\,\min(|D_A|,E_A)\geq0,\]
\[\tilde{\m v}_n^\top \m G_{\mathrm{pCK}}^{\mathbb R,0}(A)\tilde{\m v}_n=
-\tr\m G_{\mathrm{pCK}}^{\mathbb R,0}\m Q_{\mathrm{pCK}}^{\mathbb R,\mathrm{cen}}(A)
=-16|\Ima\lambda_1|\,|\Ima\lambda_2|\,\max(|D_A|,E_A)\leq0.\]

Whenever the representing vectors are non-vanishing, one can see that
  $[\tilde{\m v}_e]$ is the line connecting the $h$-eigenpoints
 $\iota^{[2]}(\lambda_1)$ and $\iota^{[2]}(\lambda_2)$;
 $[\tilde{\m v}_g]$ is the $h$-perpendicular bisector of the segment connecting
 $\iota^{[2]}(\lambda_1)$ and $\iota^{[2]}(\lambda_2)$ (degenerating into the asymptotic tangent line if one the eigenvalues is real);
and
\[\tilde{\m z}_n=-\frac12\m G_{\mathrm{pCK}}^{\mathbb R,0}\tilde{\m v}_n
%\begin{commentx}
=\bem   (\Rea\lambda_1)|\Ima\lambda_2|+(\Rea\lambda_2)|\Ima\lambda_1|\\
 |\Ima\lambda_1|\,|\lambda_2|^2+|\Ima\lambda_2|\,|\lambda_1|^2\\
 |\Ima\lambda_1|+ |\Ima\lambda_2|\eem
%\end{commentx}
 =
 |\Ima\lambda_2|\bem    \Rea\lambda_1\\|\lambda_1|^2\\ 1\eem
 +|\Ima\lambda_1|\bem    \Rea\lambda_2\\|\lambda_2|^2\\ 1\eem
 \]
has the property that $[\tilde{\m z}_n]$ (that is the pole of $[\tilde{\m v}_n]$)
is the $h$-midpoint of the segment connecting
 $\iota^{[2]}(\lambda_1)$ and $\iota^{[2]}(\lambda_2)$
 (degenerating into the asymptotic point if one the eigenvalues is real).
\snewpage

Another bunch of matrixes is given by
\[\m Q_{\mathrm{pCK}}^{\mathbb R,\mathrm{eig}}(A):=\m Q_{\mathrm{pCK}}^{\mathrm{eig}}(A)|_{\{1,3,4\}\times\{1,3,4\}}
  \equiv-|D_A|2\m Q^{\mathbb R,0}_{\mathrm{pCK}}+\m Q_{\mathrm{pCK}}^{\mathrm{spec}}(A)|_{\{1,3,4\}\times\{1,3,4\}},\]
\[\m Q_{\mathrm{pCK}}^{\mathbb R,\mathrm{ceig}}(A):=
  (\min(E_A,|D_A|)-|D_A|)2\m Q^{\mathbb R,0}_{\mathrm{pCK}}+\m Q_{\mathrm{pCK}}^{\mathrm{spec}}(A)|_{\{1,3,4\}\times\{1,3,4\}},\]
\[\m Q_{\mathrm{pCK}}^{\mathbb R,\mathrm{cx}}(A):=
  (\max(E_A,|D_A|)-|D_A|)2\m Q^{\mathbb R,0}_{\mathrm{pCK}}+\m Q_{\mathrm{pCK}}^{\mathrm{spec}}(A)|_{\{1,3,4\}\times\{1,3,4\}}.\]

Then
\[\m Q_{\mathrm{pCK}}^{\mathbb R,\mathrm{eig}}(A)
 =\frac12\left(\tilde{\m v}_1\tilde{\m v}_2^\top+\tilde{\m v}_2\tilde{\m v}_1^\top\right) ,\]
where
\[\tilde{\m v}_1=\bem -2(\Rea\lambda_1)\\ 1\\ |\lambda_1|^2\eem
\qquad\text{and}\qquad
\tilde{\m v}_2=\bem -2(\Rea\lambda_2)\\ 1\\ |\lambda_2|^2\eem.\]
Furthermore,
 \[\m Q_{\mathrm{pCK}}^{\mathbb R,\mathrm{ceig}}(A)
 = \tilde{\m v_g} \tilde{\m v_g} ^\top+\tilde{\m w_g} \tilde{\m w_g} ^\top  ,\]
where
\[\tilde{\m v}_g =\bem -(\Rea\lambda_1)-(\Rea\lambda_2)\\ 1\\ (\Rea\lambda_1)(\Rea\lambda_2)-|\Ima\lambda_1|\,|\Ima\lambda_2|\eem
\quad\text{and}\quad
\tilde{\m w}_g=\bem - |\Ima\lambda_1|-|\Ima\lambda_2| \\ 0\\  (\Rea\lambda_1)|\Ima\lambda_2|+(\Rea\lambda_2)|\Ima\lambda_1| \eem.\]
Moreover,
  \[\m Q_{\mathrm{pCK}}^{\mathbb R,\mathrm{cx}}(A)
 = \tilde{\m v_n} \tilde{\m v_n} ^\top+\tilde{\m w_n} \tilde{\m w_n} ^\top ,\]
where
\[\tilde{\m v}_n =\bem -(\Rea\lambda_1)-(\Rea\lambda_2)\\ 1\\ (\Rea\lambda_1)(\Rea\lambda_2)+|\Ima\lambda_1|\,|\Ima\lambda_2|\eem
\quad\text{and}\quad
\tilde{\m w}_n=\bem  |\Ima\lambda_1|-|\Ima\lambda_2| \\ 0\\  (\Rea\lambda_1)|\Ima\lambda_2|-(\Rea\lambda_2)|\Ima\lambda_1| \eem.\]

These are related to the previously given matrices by
\[4\min(E_A,|D_A|)\m Q_{\mathrm{pCK}}^{\mathbb R,\mathrm{cx}}(A)
  =\m Q_{\mathrm{pCK}}^{\mathbb R,\mathrm{majax}}(A)-\m Q_{\mathrm{pCK}}^{\mathbb R,\mathrm{minax}}(A),\]
\[4\max(E_A,|D_A|)\m Q_{\mathrm{pCK}}^{\mathbb R,\mathrm{ceig}}(A)
  =\m Q_{\mathrm{pCK}}^{\mathbb R,\mathrm{majax}}(A)-\m Q_{\mathrm{pCK}}^{\mathbb R,\mathrm{cen}}(A)  ,\]
\[4|\Ima\lambda_1|\,|\Ima\lambda_2|\m Q_{\mathrm{pCK}}^{\mathbb R,\mathrm{eig}}(A)
  =\m Q_{\mathrm{pCK}}^{\mathbb R,\mathrm{minax}}(A)-\m Q_{\mathrm{pCK}}^{\mathbb R,\mathrm{cen}}(A)  .\]

Then one can check
\begin{multline*}
\m Q_{\mathrm{pCK}}^{\mathbb R}(A)
=4(U_A-|D_A|)(U_A+|D_A|+2E_A)\m Q_{\mathrm{pCK}}^{\mathbb R,0} +\\
+2(U_A-|D_A|)\m Q_{\mathrm{pCK}}^{\mathbb R,\mathrm{eig}}(A)+\m Q_{\mathrm{pCK}}^{\mathbb R,\mathrm{majax}}(A);
\end{multline*}
\begin{multline*}
\m Q_{\mathrm{pCK}}^{\mathbb R}(A)
=4(U_A-|D_A|)(U_A-|D_A|+2\min(E_A,|D_A|))\m Q_{\mathrm{pCK}}^{\mathbb R,0} +\\
+2(U_A-|D_A|+2\min(E_A,|D_A|))\m Q_{\mathrm{pCK}}^{\mathbb R,\mathrm{cx}}(A)+\m Q_{\mathrm{pCK}}^{\mathbb R,\mathrm{minax}}(A);
\end{multline*}
\begin{multline*}
\m Q_{\mathrm{pCK}}^{\mathbb R}(A)
=4(U_A-|D_A|)(U_A-|D_A|+2\max(E_A,|D_A|))\m Q_{\mathrm{pCK}}^{\mathbb R,0} +\\
+2(U_A-|D_A|+2\max(E_A,|D_A|))\m Q_{\mathrm{pCK}}^{\mathbb R,\mathrm{ceig}}(A)+\m Q_{\mathrm{pCK}}^{\mathbb R,\mathrm{cen}}(A).
\end{multline*}
\end{disc}
%\snewpage
As a corollary, or directly, we have
\begin{lemma}\plabel{lem:CRdeg}
Assume that $A$ is a normal complex $2\times2$ matrix with eigenvalues $\lambda_1,\lambda_2$.
Thus the endpoints of $\DW^{\mathbb R}_{\mathrm{pCK}}(A)$
are  $\iota^{[2]}_{\mathrm{pCK}}(\lambda_1)=(\Rea\lambda_1,|\lambda_1|^2)$ and
$\iota^{[2]}_{\mathrm{pCK}}(\lambda_2)=(\Rea\lambda_2,|\lambda_2|^2)$.
Then, $\m G^{\mathbb R}_{\mathrm{pCK}}(A)$ and $\m Q^{\mathbb R}_{\mathrm{pCK}}(A)$
can be expressed in terms of the endpoints as follows:

(a)
\begin{equation}
\m G^{\mathbb R}_{\mathrm{pCK}}(A)=
\frac12\left(
\bem\Rea \lambda_1\\|\lambda_1|^2\\1\eem\bem\Rea \lambda_2\\|\lambda_2|^2\\1\eem^\top
+
\bem\Rea \lambda_2\\|\lambda_2|^2\\1\eem\bem\Rea \lambda_1\\|\lambda_1|^2\\1\eem^\top
\right)
.
\plabel{eq:CRGdegrec}
\end{equation}

(b)
\[\m Q^{\mathbb R}_{\mathrm{pCK}}(A)
=\tilde{\m v}\tilde{\m v}^\top\]
where
\[
\tilde{\m v}=
\bem |\lambda_1|^2-|\lambda_2|^2\\\Rea\lambda_2-\Rea\lambda_1\\(\Rea\lambda_1)|\lambda_2|^2-(\Rea\lambda_2)|\lambda_1|^2 \eem
=\bem\Rea \lambda_1\\|\lambda_1|^2\\1\eem\times \bem\Rea \lambda_2\\|\lambda_2|^2\\1\eem\]
(with $\times$ meaning the ordinary vector product of Gibbs).
\begin{proof}
We can compute the data from $A=\bem\lambda_1&\\&\lambda_2\eem$.
\end{proof}
\end{lemma}
\begin{proof}[Alternative proof for Theorem \ref{thm:redfive}]
%\plabel{rem:CRmat}
$\m G_{\mathrm{pCK}}^{\mathbb R}(A)$ can be obtained by \eqref{eq:CRGrec} / \eqref{eq:CRGdegrec};
from which $V,W,X,Y,Z$ can be recovered.
\end{proof}
%\snewpage
As we know, $\m G^{\mathbb R}_{\mathrm{pCK}}(A)$ has no data loss relative to the reduced five data of $A$,
not even in the normal case.
Regarding $\m Q^{\mathbb R}_{\mathrm{pCK}}(A)$:
\begin{lemma}\plabel{lem:CRpar}
Assume that $A$ is complex $2\times2$ matrix.

(a) If $A$ is normal  with two distinct, non-conjugate eigenvalues, then
$\m Q^{\mathbb R}_{\mathrm{pCK}}(A) $ is a matrix of rank strictly $1$.
Geometrically, it corresponds to the (double) line containing
the segment which is the conformal range.
Beyond that, $\m Q^{\mathbb R}_{\mathrm{pCK}}(A)$, as it is, also contains the data
``squared reduced eigendistance'' $4\sqrt{|D_A|}\sqrt{E_A}$ but essentially not more.

(b) If $A$ is normal with two equal or conjugate eigenvalues, then
$\m Q^{\mathbb R}_{\mathrm{pCK}}(A)=0 $.
\begin{proof}
Let us consider the description in Lemma \ref{lem:CRdeg}(b) with $\tilde{\m v}\equiv\m v_{\mathrm{pCK}}(\lambda_1,\lambda_2)$.

(b) If $\lambda_1$ and $\lambda_2$ are equal or conjugates to other, then
$(\Rea\lambda_1,|\lambda_1|^2)=(\Rea\lambda_2,|\lambda_2|^2)$, and $\m v_{\mathrm{pCK}}(\lambda_1,\lambda_2)$ vanishes.
This leads to the statement.

(a) If $\lambda_1$ and $\lambda_2$ are neither equal nor conjugates to other, then
$(\Rea\lambda_1,|\lambda_1|^2)\neq(\Rea\lambda_2,|\lambda_2|^2)$, and $\m v_{\mathrm{pCK}}(\lambda_1,\lambda_2)\neq0$.
This leads to a double line for $\m Q^{\mathbb R}_{\mathrm{pCK}}(A)$.
Then, by the construction of $\m v_{\mathrm{pCK}}(\lambda_1,\lambda_2)$,
it is easy to see that the line fits to the distinct points
$\iota_{\mathrm{pCK}}^{[2]}(\lambda_1)=(\Rea\lambda_1,|\lambda_1|^2)$
and
$\iota_{\mathrm{pCK}}^{[2]}(\lambda_2)=(\Rea\lambda_2,|\lambda_2|^2)$.
From $\m Q^{\mathbb R}_{\mathrm{pCK}}(A)$ one can reconstruct $\pm\m v_{\mathrm{pCK}}(\lambda_1,\lambda_2)$.
If $\m v_{\mathrm{pCK}}(\lambda_1,\lambda_2)=\pm(u_1,u_2,u_3)$, then
$4\sqrt{|D_A|}\sqrt{E_A}=|\lambda_1-\lambda_2|\cdot|\lambda_1-\bar\lambda_2|=\sqrt{(u_1)^2-4u_2u_3}$,
and this information is specific with respect to the scaling of $\pm\m v_{\mathrm{pCK}}(\lambda_1,\lambda_2)$.
(In particular, in the normal setting,
we have the formula
 $4\sqrt{|D_A|}\sqrt{E_A}=
\sqrt{\tr((\m Q^{\mathbb R,0}_{\mathrm{pCK}})^{-1}\m Q^{\mathbb R}_{\mathrm{pCK}}(A))}$.)
\end{proof}
\end{lemma}
\snewpage

\begin{remark}
\plabel{rem:preCRscalar}
(a) If $A$ is not normal, or its eigenvalues are neither equal nor conjugate, then we may define, say,
\begin{commentx}
\[\acute{\m Q}_{\mathrm{pCK}}^{\mathbb R}(A)=\frac1{2(U_A-|D_A|+2\min(E_A,|D_A|))} {\m Q}_{\mathrm{pCK}}^{\mathbb R}(A).\]
\end{commentx}
\[\breve{\m Q}_{\mathrm{pCK}}^{\mathbb R}(A)=\frac1{2(U_A-|D_A|+2 \sqrt{|D_A|} \sqrt{E_A} )} {\m Q}_{\mathrm{pCK}}^{\mathbb R}(A) \]
(reasonable variations are possible).
~
If $A$ is not normal, then $\breve{\m Q}_{\mathrm{pCK}}^{\mathbb R}(A)$ represents the conformal range faithfully.
(There is no information loss relative to the conformal range.)
~
If $A$ is normal but its eigenvalues are neither equal nor conjugate, then $\breve{\m Q}_{\mathrm{pCK}}^{\mathbb R}(A)$
contains only the information of the line of the conformal range, i.~e.~the line of $h$-eigenvalues.
(There is information loss relative to the conformal range.)

(b)
Assume that $A$ is normal with two equal or conjugate eigenvalues; say, the eigenvalues are $\lambda$ or $\bar \lambda$.
Then  it is reasonable to define
\[\breve{\m Q}_{\mathrm{pCK}}^{\mathbb R}(A)=
\begin{bmatrix}
4|\lambda|^2&-2\Rea \lambda&-2|\lambda|^2\Rea \lambda\\
-2\Rea \lambda&1&2(\Rea \lambda)^2-|\lambda|^2 \\
-2|\lambda|^2\Rea \lambda&2(\Rea \lambda)^2-|\lambda|^2&|\lambda|^4
\end{bmatrix}.\]
\begin{commentx}
This is natural in the sense that in this case,
\[
\adj\breve{\m Q}_{\mathrm{pCK}}^{\mathbb R}(A)=4|\Ima\lambda|^2\cdot \m G_{\mathrm{pCK}}^{\mathbb R}(A).
\eqedremark
\]
\end{commentx}
(Cf.  $\m Q _{\mathrm{pCK}}^{\mathbb R,\mathrm{cx}}(A))$.)
~
In this case $\breve{\m Q}_{\mathrm{pCK}}^{\mathbb R}(A)$ represents the conformal range (i.~e.~the $h$-eigenpoint) faithfully.
(There is no information loss relative to the conformal range.)
~
This yields a non-continuous but relatively natural extension of the original assigment.

(c) In general,
\begin{commentx}
\begin{multline*}
\adj\acute{\m Q}_{\mathrm{pCK}}^{\mathbb R}(A)=\\
=\frac{U_A-|D_A|}{U_A-|D_A|+2\min(E_A,|D_A|)} \cdot 2(U_A-|D_A|+2\max(E_A,|D_A|))\m G_{\mathrm{pCK}}^{\mathbb R}(A)
\end{multline*}
\end{commentx}
\begin{multline*}
\adj\breve{\m Q}_{\mathrm{pCK}}^{\mathbb R}(A)=
 \frac{(U_A-|D_A|)(U_A-|D_A|+2\min(E_A,|D_A|))}{(U_A-|D_A|+2\sqrt{|D_A|} \sqrt{E_A} )^2} \cdot\\\cdot 2(U_A-|D_A|+2\max(E_A,|D_A|))\m G_{\mathrm{pCK}}^{\mathbb R}(A)
\end{multline*}
with $\frac00=1$.\qedremark
\end{remark}

\snewpage
\subsection{The transformation properties of the matrix}

\begin{lemma}\plabel{lem:jwelltrans}
Suppose that $f$ is a M\"obius transformation $f:\lambda\mapsto \frac{a\lambda+b}{c\lambda+d}$, $a,b,c,d\in\mathbb R$, $ad-bc\neq0$,
and $A$ is $2\times2$ complex matrix such that $-\frac dc$ is not an eigenvalue of $A$.

\begin{commenty}
(a)
\end{commenty}
Then, the transformation rule
\[E_{f(A)} \cdot \mathcal C(f,A)=E_A \]
holds.

\begin{commenty}
(b) For $j\in\{1,2,3\}$, the transformation rules
\[U_j(f(A)) \cdot \mathcal C(f,A)^j=U_j(A) \]
hold.
\end{commenty}
\begin{proof}
\begin{commenty}
(a)
\end{commenty}
 By Lemma \ref{lem:UDtransform} we have the same transformation rule for $U_A$ and $|D_A|$ (even more generally).
As $E_A/|D_A|$ is generically an invariant under real M\"obius transformations, $E_A$ must have the
same rule for real M\"obius transformations.

\begin{commenty}
(b) We deal with $j$-homogeneous polynomials of $U_A$, $|D_A|$, $E_A$; the transformation rules are corresponding to this fact.
\end{commenty}
\end{proof}
\end{lemma}

\begin{lemma}
\plabel{lem:Etransform}
Suppose that $f$ is a M\"obius transformation $f:\lambda\mapsto \frac{a\lambda+b}{c\lambda+d}$, $a,b,c,d\in\mathbb R$, $ad-bc\neq0$,
and $A$ is $2\times2$ complex matrix.
Then
\begin{equation*}
\mathcal C(f,A)= \frac1{(ad-bc)^2}
\begin{bmatrix}2cd\\c^2 \\ d^2\end{bmatrix}^\top
\m G^{\mathbb R}_{\mathrm{pCK}}(A)
\begin{bmatrix}2cd \\c^2 \\ d^2\end{bmatrix}.
\end{equation*}

In particular, for a real $f$, $\mathcal C(f,A)$ depends only on the reduced five data of $A$.
\begin{proof}
This is just the transcription of \eqref{def:gtran}.
Note that $\m G^{\mathbb R}_{\mathrm{pCK}}(A)$ depends only on the reduced five data.
\end{proof}
\end{lemma}

\begin{lemma}\plabel{lem:CRtransform}
Suppose that $f$ is a M\"obius transformation $f:\lambda\mapsto \frac{a\lambda+b}{c\lambda+d}$, $a,b,c,d\in\mathbb R$, $ad-bc\neq0$,
and $A$ is $2\times2$ complex matrix such that $-\frac dc$ is not an eigenvalue of $A$.

Let $R^{[2]}_{\mathrm{pCK}}(f)$ be a matrix of determinant $-1$ or $1$ representing the projective action of $f$ in $\mathrm{pCK}$.
Then the following transformation rules holds:
\begin{equation}
\m Q^{\mathbb R}_{\mathrm{pCK}}(f(A)) \cdot  \mathcal C(f,A)^2
=\left(R^{[2]}_{\mathrm{pCK}}(f)\right)^{-1,\top}
\m Q^{\mathbb R}_{\mathrm{pCK}}(A)
\left(R^{[2]}_{\mathrm{pCK}}(f)\right)^{-1} ,
\plabel{eq:CRtrans1}
\end{equation}
\begin{equation}
\m G^{\mathbb R,0}_{\mathrm{pCK}} \m Q^{\mathbb R}_{\mathrm{pCK}}(f(A)) \cdot  \mathcal C(f,A)^2
=\left(R^{[2]}_{\mathrm{pCK}}(f)\right)
\m G^{\mathbb R,0}_{\mathrm{pCK}} \m Q^{\mathbb R}_{\mathrm{pCK}}(A)
\left(R^{[2]}_{\mathrm{pCK}}(f)\right)^{-1}
;\plabel{eq:CRtrans2}
\end{equation}
\begin{equation}
\m G^{\mathbb R}_{\mathrm{pCK}}(f(A)) \cdot \mathcal C(f,A)
=
\left(R^{[2]}_{\mathrm{pCK}}(f)\right)\m G^{\mathbb R}_{\mathrm{pCK}}(A) \left(R^{[2]}_{\mathrm{pCK}}(f)\right)^\top
;\plabel{eq:CRtrans1g}
\end{equation}
\begin{equation}
\m G^{\mathbb R}_{\mathrm{pCK}}(f(A))\m Q^{\mathbb R,0}_{\mathrm{pCK}}\cdot  \mathcal C(f,A)
=\left(R^{[2]}_{\mathrm{pCK}}(f)\right)
\m G^{\mathbb R}_{\mathrm{pCK}}(A)\m Q^{\mathbb R,0}_{\mathrm{pCK}}
\left(R^{[2]}_{\mathrm{pCK}}(f)\right)^{-1}.
\plabel{eq:CRtrans2g}
\end{equation}

\begin{proof}
By Theorem  \ref{thm:CR}, we already know that $\m Q_{\mathrm{pCK}}(A)$ transforms naturally.
The scaling factor in \eqref{eq:CRtrans1} can generically be written as
$\sqrt[3]{ \dfrac{\det\m Q^{\mathbb R}_{\mathrm{pCK}}(A)}{\det\m Q^{\mathbb R}_{\mathrm{pCK}}(f(A))} }$, but,
due to \eqref{eq:detCR} and the scaling properties of $U_A$, $|D_A|$, $E_A$, this reduces generically to $\mathcal C(f,A)^2$.
Then, by continuity, the equality \eqref{eq:CRtrans1} extends.
Equivalence to \eqref{eq:CRtrans2} follows from the invariance property
$\m Q^{\mathbb R,0}_{\mathrm{pCK}}=\left(R^{[2]}_{\mathrm{pCK}}(f)\right)^{-1,\top} \m Q^{\mathbb R,0}_{\mathrm{pCK}}\left(R^{[2]}_{\mathrm{pCK}}(f)\right)^{-1} $.
Similar argument applies to the $\m G^{\mathbb R}_{\mathrm{pCK}}(A)$.
%\end{proof}
%\begin{proof}[Alternative proof]
\texttt{Alternatively,}
\eqref{eq:CRtrans1g} and \eqref{eq:CRtrans2g} follow from \eqref{eq:LSZtrans1g} and \eqref{eq:LSZtrans2g} by restriction.
Then   \eqref{eq:CRtrans1} and \eqref{eq:CRtrans2} follow from \eqref{eq:Gadj}.
\end{proof}
\end{lemma}
\begin{proof}
[Alternative proof to Theorem \ref{thm:CReigen}]
Comparing Example \ref{ex:checkGQ} and Example \ref{ex:confrep},
the statement checks out for the  canonical representatives of Lemma \ref{lem:preCR}.
By Lemma \ref{lem:CRtransform}, real M\"obius transformations, do not change the ratio of the eigenvalues;
moreover the overall scaling factors are also correct.
\end{proof}

\begin{commenty}
In the manner of Corollary \ref{cor:welltrans}, we can prepare various naturally well-transforming
versions of $\m G^{\mathbb R}_{\mathrm{pCK}}(A)$ and $\m Q^{\mathbb R}_{\mathrm{pCK}}(A)$.
As we will see, it is useful to be opportunistic about the scaling factor.
We only have to be careful that it should be  $(-1)$-homogeneous in $U_A,|D_A|,E_A$ for $\m G^{\mathbb R}_{\mathrm{pCK}}(A)$,
and it should be $(-2)$-homogeneous in $U_A,|D_A|,E_A$ for $\m Q^{\mathbb R}_{\mathrm{pCK}}(A)$.
In particular,
\begin{cor}\plabel{cor:rwelltrans}
(a) In the  $U_1(A)\neq0$ case, the matrix is given by
\begin{equation}
\widetilde{\m G}^{\mathbb R}_{\mathrm{pCK}}(A)
=
\frac1{U_1(A)}\m G_{\mathrm{pCK}}^{\mathbb R}(A)
\plabel{nt:uperfuu}
\end{equation}
has the transformation property
\[\widetilde {\m G}^{\mathbb R}_{\mathrm{pCK}}(f(A))
=\left(R^{[2]}_{\mathrm{pCK}}(f)\right)
\widetilde {\m G}^{\mathbb R}_{\mathrm{pCK}}(A)
\left(R^{[2]}_{\mathrm{pCK}}(f)\right)^{\top} ,\]
whenever $f(A)$ makes sense.

(b) In the  $U_2(A)\neq0$ case, the matrix is given by
\begin{equation}
\widehat{\m Q}^{\mathbb R}_{\mathrm{pCK}}(A)
=
\frac1{U_2(A)}\m Q_{\mathrm{pCK}}^{\mathbb R}(A)
\plabel{nt:rperfuu}
\end{equation}
has the transformation property
\[\widehat {\m Q}^{\mathbb R}_{\mathrm{pCK}}(f(A))
=\left(R^{[2]}_{\mathrm{pCK}}(f)\right)^{-1,\top}
\widehat {\m Q}^{\mathbb R}_{\mathrm{pCK}}(A)
\left(R^{[2]}_{\mathrm{pCK}}(f)\right)^{-1} ,\]
whenever $f(A)$ makes sense.
\qed
\end{cor}

%\snewpage
\begin{remark}
\plabel{rem:CRscalar}
(a) If $A$ is normal with two distinct, non-conjugate eigenvalues, then $\widehat {\m Q}^{\mathbb R}_{\mathrm{pCK}}(A)$
contains no more information than the line of the conformal range.

(b) Assume that $A$ is a normal matrix with eigenvalues $\lambda$ or $\bar\lambda$.
Then
\eqref{nt:rperfuu}
allows a natural but not continuous extension by
\begin{equation*}
\widehat{\m Q}^{\mathbb R}_{\mathrm{pCK}}(A)
=
\frac1{8(\Ima \lambda)^2}\breve{\m Q}_{\mathrm{pCK}}^{\mathbb R}(A).
\end{equation*}

This leaves only real scalar matrices without a particular choice for $\widehat{\m Q}^{\mathbb R}_{\mathrm{pCK}}(A)$.
In that case, only $\mathbb R^+\widehat{\m Q}^{\mathbb R}_{\mathrm{pCK}}(A)$ is a sufficiently
invariant object.
\qedremark
\end{remark}
\end{commenty}

\snewpage
\subsection{Canonical representatives and the geometry of the conformal range}
\plabel{ssub:cangeo}
~\\

Although Theorem \ref{thm:CR} is sufficiently explicit, it is useful to visualize certain particular cases.
We use the canonical representatives from Lemma \ref{lem:preCR}.

\begin{example}\plabel{ex:CRBCK}
%The canonical representatives yield:
(a)
\[\m Q_{\mathrm{BCK}}^{\mathbb R}(L_{\alpha,t}^\pm)=
\begin{bmatrix}
16t^2(1+t^2)&&\\
&16((\cos\alpha)^2+t^2)(1+t^2)&\\
&&-16t^2((\cos\alpha)^2+t^2)
\end{bmatrix}\,.
\]

(b)
\[\m Q_{\mathrm{BCK}}^{\mathbb R}(S_\beta)=
\begin{bmatrix}
1&&\\
&2(\cos \beta)^2&(\cos \beta)^2\\
&(\cos \beta)^2&
\end{bmatrix}\,.
\]

(c)
\[\m Q_{\mathrm{BCK}}^{\mathbb R}(\m 0_2)=
\m 0_3\,.
\eqedexer
\]
\end{example}

\begin{theorem}\plabel{cor:CR}

The conformal ranges of the canonical representatives in the BCK model are as follows:

The zero matrix $\m 0_2$ yields the point ellipse
\[\{(0,-1)\};\]
$S_\beta$ yields the ellipse with axes
\[ \left[-\frac{\sqrt2}{2}\cos\beta,\frac{\sqrt2}{2}\cos\beta \right]\times\left\{-\frac12\right\}\qquad\text{and}\qquad  \{0\}\times[-1,0];\]
and  $L^\pm_{\alpha,t}$ yields the ellipse with axes
\begin{equation*} \left[-\frac{\sqrt{(\cos\alpha)^2+t^2 }}{\sqrt{1+t^2}} ,\frac{\sqrt{(\cos\alpha)^2+t^2 }}{\sqrt{1+t^2}} \right]\times\{0\}\qquad\text{and}\qquad  \{0\}\times\left[-\frac{t}{\sqrt{1+t^2}},\frac{t}{\sqrt{1+t^2}}\right];
%\plabel{eq:Lalphaaxes}
\end{equation*}
(the ellipses may be degenerate).
\end{theorem}

\begin{proof}
In the non-normal cases ($t>0$, $\beta\in[0,\pi/2)$) the result follows from Theorem \ref{thm:CR} via Example \ref{ex:CRBCK}.
The normal cases follow from the continuity of the conformal range as a set
(but not from the continuity of $\m Q_{\mathrm{BCK}}^{\mathbb R}(A)$).
\end{proof}

\snewpage
\begin{proof}[Proof of Theorem \ref{cor:CR} via symmetry principles.]
One can compute the conformal range in the BCK model for $L_{\alpha,t}^\pm$ of Lemma \ref{lem:preCR}
with $t\geq0,$ $\alpha\in (0,\pi/2]$ as follows:
As we deal with projections of non-degenerate $h$-tubes, we already know that the results will be ($h$-)ellipses in the BCK model.
As $A=L_{\alpha,t}^\pm$ is unitarily conjugate not only to $(A^{-1})^*$ but to $-A$ or $-A^*$, we see that the ellipses
will be symmetric to the $x_{\mathrm{BCK}}$ and $z_{\mathrm{BCK}}$ axes,
thus they will be aligned along those axes.
Now, it is easy to obtain the norms
\[\|L_{\alpha,t}^\pm\|_2=\left\| \begin{bmatrix} \cos\alpha+\mathrm i\sin\alpha& 2t\\
&-\cos\alpha\pm\mathrm i\sin\alpha \end{bmatrix} \right\|_2=t+\sqrt{1+t^2},\]
 and the norms of the Cayley transforms,
\[\left\|\frac{\Id_2-L_{\alpha,t}^\pm}{ \Id_2 +L_{\alpha,t}^\pm }  \right\|_2=
\left\| \begin{bmatrix} -\mathrm i\tan\frac\alpha2&\pm2\mathrm it \mathrm e^{\frac{\pm1-1}2\mathrm i\alpha}\csc\alpha\\
& \mp\mathrm i\cot\frac\alpha2\end{bmatrix}\right\|_2
=\frac{\sqrt{(\cos\alpha)^2+t^2}+\sqrt{1+t^2}}{\sin\alpha}.\]
(Computing only case $+$ is sufficient.)
Then, applying the Lemma of ``Extremal values in ranges'' from \cite{L2}
(there `$y_{\mathrm{BCK}}$' was used instead of `$z_{\mathrm{BCK}}$'), we find that
\[\sup z_{\mathrm{BCK}}(\DW_{\mathrm{BCK}}^{\mathbb R}(L_{\alpha,t}^\pm))
=\frac{t}{\sqrt{1+t^2}},\]
and
\[\inf x_{\mathrm{BCK}}(\DW_{\mathrm{BCK}}^{\mathbb R}(L_{\alpha,t}^\pm))
=-\frac{\sqrt{(\cos\alpha)^2+t^2 }}{\sqrt{1+t^2}}.
\]
Taking the axial symmetries in account, this proves that
$\DW_{\mathrm{BCK}}^{\mathbb R}(L_{\alpha,t}^\pm)$ is as indicated.
The case $\alpha=0$ follows by continuity.

$\DW_{\mathrm{BCK}}^{\mathbb R}(S_\beta)$ can be computed as similarly.
We have symmetry for the $z_{\mathrm{BCK}}$ axis.
Furthermore,
\[\|S_\beta\|_2=1,\qquad\qquad \|S_\beta\|_2^-=0,\]
and
\[\left\|\frac{\Id_2-S_\beta}{ \Id_2 +S_\beta }  \right\|_2=
\left\|\begin{bmatrix}
1&\frac{-2\cos \beta}{1+\mathrm i\sin\beta}\\
&\frac{1-\mathrm i\sin\beta}{1+\mathrm i\sin\beta}
\end{bmatrix}\right\|_2
=\sqrt{
\frac{\sqrt2+\cos\beta}{\sqrt2-\cos\beta}
}
.\]
Applying the Lemma of ``Extremal values in ranges'' from \cite{L2}, we find that
\[\sup z_{\mathrm{BCK}}(\DW_{\mathrm{BCK}}^{\mathbb R}(S_\beta))=0,\]
\[\inf z_{\mathrm{BCK}}(\DW_{\mathrm{BCK}}^{\mathbb R}(S_\beta))=-1,\]
and
\[\inf x_{\mathrm{BCK}}(\DW_{\mathrm{BCK}}^{\mathbb R}(S_\beta))
=-\frac{\sqrt2}2\cos\beta.
\]
Taking the symmetry into account, this proves that $\DW_{\mathrm{BCK}}^{\mathbb R}(S_\beta)$ is as indicated.

The case of $\m 0_2$ is trivial.
\end{proof}
\begin{commentx}
Of course, one can obtain proofs for Theorem \ref{cor:CR} just by specifying the proofs of Theorem \ref{thm:CR}.
We do this with respect to the rotational enveloping construction, which ``simplifies'' as follows:
\begin{remark}\plabel{lem:rotenvel}
Using the symmetry principles in the previous argument have simplified the computation of the particular cases greatly.
Without them

(a) Consider
\[A=\begin{bmatrix}
\cos\alpha+\mathrm i\sin\alpha&2t\\&-\cos\alpha+\mathrm i\sin\alpha
\end{bmatrix}\]
$(\alpha\in[0,\pi/2] , \,t\geq0)$.
In this case
\[\left(\left\|\dfrac{(\cos\frac\omega2)A-(\sin\frac\omega2)\Id}{(\sin\frac\omega2)A+(\cos\frac\omega2)\Id}\right\|_2\right)^2
=\frac{\sqrt{1+t^2}+\sqrt{(\cos\alpha)^2(\sin\omega)^2+t^2}}{\sqrt{1+t^2}-\sqrt{(\cos\alpha)^2(\sin\omega)^2+t^2}} ;\]
and
\[\widehat E^A(\omega)=\left(\frac{-((\cos\alpha)^2+t^2)\sin\omega}{\sqrt{1+t^2} \sqrt{(\cos\alpha)^2(\sin\omega)^2+t^2}},
\frac{t^2\cos\omega}{\sqrt{1+t^2} \sqrt{(\cos\alpha)^2(\sin\omega)^2+t^2}}
\right).\]
(The derived data is the same with other sign choices.)

(b) Consider
\[A=\begin{bmatrix}0&\cos\beta\\0&\mathrm i\sin\beta\end{bmatrix}\]
with $\beta\in[0,\pi]$.
Then
\[\left(\left\|\dfrac{(\cos\frac\omega2)A-(\sin\frac\omega2)\Id}{(\sin\frac\omega2)A+(\cos\frac\omega2)\Id}\right\|_2\right)^2
=\frac{2-\cos\omega+\sqrt{2(\sin\omega)^2(\cos\beta)^2+(\cos\omega)^2}}{2+\cos\omega-\sqrt{2(\sin\omega)^2(\cos\beta)^2+(\cos\omega)^2}};\]
and the rotational construction yields
\[\widehat E_{\mathrm{pCK}}^A(\omega)=\left(-\frac{(\cos\beta)^2\sin\omega}{\sqrt{2(\sin\omega)^2(\cos\beta)^2+(\cos\omega)^2}},
-\frac12+\frac12\frac{\cos\omega}{\sqrt{2(\sin\omega)^2(\cos\beta)^2+(\cos\omega)^2}} \right).\]

The advantage here is that it is not presupposed that the result will be possibly degenerate $h$-ellipse.
\qedexer
\end{remark}
\end{commentx}
~\snewpage

\subsection{The synthetic geometry of the conformal range (review) }
~\\

The following statements are explained in detail in \cite{LLL}.

A qualitative description of the conformal range of $2\times2$ matrices is given by:

\begin{theorem}
\plabel{thm:CRdonc}
(See \cite{LLL}.)
Suppose that $A$ is a linear operator on a $2$-dimensional complex Hilbert space.
We have the following possibilities:

(i) $A$ has a double non-real eigenvalue $\lambda$ (complex-parabolic case),
or two conjugate non-real eigenvalues $\lambda=\bar\lambda$ (real-elliptic case), and $A$ is normal.

Then $\DW^{\mathbb R}_*(A)$ contains only the ordinary $h$-point $\iota^{[2]}_*(\lambda)$.

(ii) $A$ has two not equal, nor conjugate non-real eigenvalues $\lambda_1,\lambda_2$
(quasihyperbolic and quasielliptic cases), and $A$ is normal.

Then $\DW^{\mathbb R}_*(A)$ is the $h$-segment connecting  $\iota^{[2]}_*(\lambda_1)$ and $\iota^{[2]}_*(\lambda_2)$.

(iii)  $A$ has two different real eigenvalues $\lambda_1,\lambda_2$ (real-hyperbolic case), and $A$ is normal.

Then $\DW^{\mathbb R}_*(A)$ is the asymptotically closed $h$-line connecting $\iota^{[2]}_*(\lambda_1)$ and $\iota^{[2]}_*(\lambda_2)$.

(iv) $A$ has a non-real eigenvalue $\lambda_1$ and a real eigenvalue $\lambda_2$ (semi-real case), and $A$ is normal.

Then $\DW^{\mathbb R}_*(A)$ is the asymptotically closed $h$-half line connecting  $\iota^{[2]}_*(\lambda_1)$ and $\iota^{[2]}_*(\lambda_2)$.

(v) $A$ has a double real eigenvalue $\lambda$ (real-parabolic case), and $A$ is normal.

Then $\DW^{\mathbb R}_*(A)$ contains only the asymptotic $h$-point $\iota^{[2]}_*(\lambda)$.

(vi) $A$ has a double non-real eigenvalue $\lambda$ (complex-parabolic case),
or two conjugate non-real eigenvalues $\lambda=\bar\lambda$ (real-elliptic case), and $A$ is not normal.

Then $\DW^{\mathbb R}_*(A)$ is an $h$-circle.
Here $\iota^{[2]}_*(\lambda)$ is an interior point.

(vii) $A$ has two not equal, nor conjugate non-real eigenvalues $\lambda_1,\lambda_2$
(quasihyperbolic and quasielliptic cases), and $A$ is not normal.

Then $\DW^{\mathbb R}_*(A)$ is a proper $h$-ellipse.
Here $\iota^{[2]}_*(\lambda_1)$ and $\iota^{[2]}_*(\lambda_2)$ are interior points.

(viii)  $A$ has two different real eigenvalues $\lambda_1,\lambda_2$ (real-hyperbolic case), and $A$ is not normal.

Then $\DW^{\mathbb R}_*(A)$ is a $h$-distance band around the $h$-line connecting $\iota^{[2]}_*(\lambda_1)$ and $\iota^{[2]}_*(\lambda_2)$.

(ix) $A$ has a non-real eigenvalue $\lambda_1$ and a real eigenvalue $\lambda_2$ (semi-real case), and $A$ is not normal.

Then $\DW^{\mathbb R}_*(A)$ is an $h$-elliptic parabola with asymptotic point $\iota^{[2]}_*(\lambda_2)$.
Here $\iota^{[2]}_*(\lambda_1)$ is an interior point.

(x) $A$ has a double real eigenvalue $\lambda$ (real-parabolic case), and $A$ is not normal.

Then $\DW^{\mathbb R}_*(A)$ is a $h$-horodisk with asymptotic point $\iota^{[2]}_*(\lambda)$.
\qed
\end{theorem}

%\snewpage

In \cite{LLL}, it is explained that how to associate
 notions like major or minor axes [segment or line]
 or   major or minor axe lengths to the conformal range of $2\times2$ complex matrices.
Here that the existence axes as segments or lines is not always
 provided but major and minor (semi)axis lengths are defined anyway (but they may be possibly infinite).
\begin{lemma}
\plabel{lem:axinv}
(See \cite{LLL}.)
If $f$ is a real M\"obius transformation
 applicable to the complex $2\times2$ matrix $A$, then the major or minor axes [segment or line]
  of the conformal ranges of $A$ and $f(A)$ are related to each other the  $h$-isometry $f_*$.
Furthermore, major and minor (semi)axis lengths are invariant under $f$.
\qed
\end{lemma}
%\snewpage
\begin{theorem}\plabel{lem:semiaxdist}
(See \cite{LLL}.)
(a) The (generalized) minor semi-axis length is
\begin{equation}
s^-=\frac12\arcosh\left(1+\frac{U_A-|D_A| }{\max(E_A,|D_A|)}\right)\equiv\arsinh \sqrt{\frac{U_A-|D_A| }{2\max(E_A,|D_A|)}}.
\plabel{eq:CRminor}
\end{equation}
In the real-parabolic normal case (v), i. e. for $E_A=|D_A|=U_A=0$, $s^{-}=0$ is declared.

(b)
In similar manner, the (generalized) major semi-axes length is
\[s^+=\frac12\arcosh\frac{U_A+E_A}{\left|E_A-|D_A|\right|}.\]
 In the real-parabolic normal case (v), i. e. for $E_A=|D_A|=U_A=0$, $s^{+}=0$ is declared.
\end{theorem}

If the eigenvalues of $A$ are $\lambda_1,\lambda_2$, then
 $\iota^{[2]}_*(\lambda_1)$ and $\iota^{[2]}_*(\lambda_2)$ are the $h$-eigenpoints of $A$.
The hyperbolic distance of the  $h$-eigenpoints of $A$,
\begin{align}
\mathrm d^*\left(\iota^{[2]}_*(\lambda_1),\iota^{[2]}_*(\lambda_2)\right) &=
\mathrm d^{\mathrm{Ph}}\left((\Rea\lambda_1,|\Ima\lambda_1|),(\Rea\lambda_2,|\Ima\lambda_2|) \right)
\plabel{eq:focdist}\\\notag&=
\mathrm d^{\mathrm{pCK}}\left((\Rea\lambda_1,|\lambda_1|^2),(\Rea\lambda_2,|\lambda_2|^2) \right)
.
\end{align}
is the $h$-eigendistance of $A$.
(However later the term `focal distance' will be justified.)

\begin{theorem}
\plabel{lem:eigdist}
(See \cite{LLL}.)
The half of the $h$-eigendistance   is
\[s^{\mathrm e}=\frac12\arcosh\frac{E_A+|D_A|}{\left|E_A-|D_A|\right|}. \]
 In the real-parabolic case (v)/(x), i. e. $E_A=|D_A|=0$, $s^{\mathrm e}$ is declared to be $0$.
\qed
\end{theorem}

\begin{lemma}
\plabel{lem:eigdistinv}
(See \cite{LLL}.)
If $f$ is a real M\"obius transformation
 applicable to the complex $2\times2$ matrix $A$, then the $h$-eigenpoints of $A$
 are related to the $h$-eigenpoints of $f(A)$ by the $h$-isometries $f_*$.
In particular, the $h$-eigendistance is invariant under $f$.
\qed
\end{lemma}

These phenomena fit together in
\begin{theorem}
\plabel{thm:addCR1}
(See \cite{LLL}.)
In the cases  (i)/(vi), (ii)/(vii), and (iv)/(ix)
 the non-asymptotic $h$-eigenpoints act as foci for the synthetic presentation of the
 possibly degenerate $h$-ellipses and $h$-elliptic parabolas respectively.
If the asymptotic points are interpreted as foci, then this extends to all cases.
\qed
\end{theorem}

After this, the expression `focal distance' (notation: $s^{\mathrm f}$) can be used synonymously to the expression
 `$h$-eigendistance' (notation: $s^{\mathrm e}$).
However, in what follows,
 we will primarily compute $h$-eigendistance as such, and we do not want to refer to geometry unnecessarily;
 therefore will generally retain the term `$h$-eigendistance' ($s^{\mathrm e}$).

In what follows we make comments regarding
geometric content presented above but in relation to
the representing matrices of the conformal range.
\\

\snewpage
\subsection{The matrix and the metrical data}
\plabel{ssub:memo}
~\\

A relatively practical statement summarizing Theorems \ref{lem:semiaxdist} and \ref{lem:eigdist}, and some geometry is
\begin{theorem}\plabel{thm:finv1}
(See \cite{LLL}.)
Assume that $A$ is a $2\times2$ complex matrix.

The unordered ratio
\begin{equation}
\{U_A-|D_A| \,:\, U_A+|D_A| \,:\, U_A-|D_A|+2E_A\}
\plabel{eq:finv1}
\end{equation}
and (when $|D_A|=E_A>0$) the possible choice of type of possibly degenerate
\[\text{$h$-distance band / $h$-elliptic parabolic disk}\]
together form a full invariant of the conformal range up to $h$-isometries.

If \eqref{eq:finv1} is $\{\chi_1:\chi_2:\chi_3\}$ with $0\leq\chi_1\leq\chi_2\leq\chi_3$,
 then
\begin{equation}
s^+=\artanh\sqrt{\frac{\chi_2}{\chi_3}},\qquad
s^-=\artanh\sqrt{\frac{\chi_1}{\chi_3}},\qquad
s^{\mathrm e}=\artanh\sqrt{\frac{\chi_2-\chi_1}{\chi_3-\chi_1}}.
\plabel{eq:recipeG}
\end{equation}
(Here the convention $\frac00=0$ applies.)

Note:
$\arcosh \frac1{\sqrt{1-u^2}}=\arsinh \frac u{\sqrt{1-u^2}}=\artanh u=\frac12\arcosh\frac{1+u^2}{1-u^2}$.
\end{theorem}

This can be formulated in relation to the matrices of the range:

\begin{theorem}\plabel{thm:CRchar}
(a)
If the unordered ratio of the eigenvalues of $\m G^{\mathbb R}_{\mathrm{pCK}}(A)\m Q^{\mathbb R,0}_{\mathrm{pCK}}$
 or  $\m G^{\mathbb R}_{\mathrm{BCK}}(A)\m Q^{\mathbb R,0}_{\mathrm{BCK}}$
 is $\{\chi_1:\chi_2:\chi_3\}$ with $0\leq\chi_1\leq\chi_2\leq\chi_3$, then, regarding the conformal range,
 \eqref{eq:recipeG} holds.
More generally, the Jordan form of the matrices above, up to non-zero scalar multiples, is a complete invariant of the
conformal range up to $h$-isometries.

(b)
In the case when $A$ is non-normal: If the unordered ratio of the  eigenvalues of
$\m G^{\mathbb R,0}_{\mathrm{pCK}} \m Q^{\mathbb R}_{\mathrm{pCK}}(A) $
or $\m G^{\mathbb R,0}_{\mathrm{BCK}}\m Q^{\mathbb R}_{\mathrm{BCK}}(A) $
is $\{\theta_1:\theta_2:\theta_3\}$ with $0\leq\theta_1\leq\theta_2\leq\theta_3$, then, regarding the conformal range,
\begin{equation}
s^+=\artanh\sqrt{\frac{\theta_1}{\theta_2}},\qquad
s^-=\artanh\sqrt{\frac{\theta_1}{\theta_3}},\qquad
s^{\mathrm e}=\artanh\sqrt{\frac{\theta_1(\theta_3-\theta_2)}{\theta_2(\theta_3-\theta_1)}}.
\plabel{eq:recipeQ}
\end{equation}
(Here the case $\frac00$ does not occur.)
More generally, the Jordan form of the matrices above, up to non-zero scalar multiples, is a complete invariant of the
conformal range up to $h$-isometries.

If $A$ is normal, then only the rank of the matrix can be derived from the ratio of the eigenvalues
or from Jordan form of the matrices  up to non-zero scalar multiples.

\begin{proof}
The metric statement is a consequence of Theorem \ref{thm:CReigen} and Theorem \ref{thm:finv1}.
Regarding the Jordan form, see that Example \ref{ex:checkGQ} that it distinguishes between
the semi-real and real-hyperbolic cases.
\end{proof}
\end{theorem}
In particular, if a matrix of the conformal range is given
(i.~e.~$\m Q^{\mathbb R}_{\mathrm{pCK}}(A)$ up to a non-zero scalar multiplier)
in the non-normal case, then the its metrical data can be recovered relatively easily.
\begin{commentx}
\begin{remark}\plabel{rem:Jordan}
Thus the nature of the Jordan forms translates to the nature of the pencils
generated by the conformal ranges and the base forms. \qedremark
\end{remark}
\end{commentx}
\begin{disc}
This raises the question the question how to associate  practical hyperbolically invariant data to the
 possible conformal ranges having (them mostly represented by their equations, that is essentially by)
  $\m G^{\mathbb R}_{\mathrm{pCK}}(A)$ given, that is equivalently, the reduced five data.
Theoretically \eqref{eq:finv1} and \eqref{eq:recipeG} can be used for this purpose as they can be computed
 from  $\m G^{\mathbb R}_{\mathrm{pCK}}(A)\m Q^{\mathbb R,0}_{\mathrm{pCK}}$.
Computing this, however, might require solving cubic equations.
In practice, instead of \eqref{eq:finv1}, one can use, equivalently, the pair of ratios
\[U_2(A):U_1(A)^2  \qquad \text{and}\qquad U_3(A):U_1(A)^3.\]
These are rational in the reduced five data, although with no obvious geometrical meaning.
\end{disc}

\begin{commenty}

\begin{theorem}\plabel{thm:finv2}
The inhomogeneous  ratio
\begin{equation}
U_1(A)\,{}^{[1]} \,:\, U_2(A)\,{}^{[2]}  \,:\, U_3(A)\,{}^{[3]}
\plabel{eq:finv2}
\end{equation}
and (when $|D_A|=E_A>0$) the possible choice of type of possibly degenerate
\[\text{$h$-distance band / $h$-elliptic parabolic disk}\]
together form a full invariant of the conformal range up to $h$-isometries.

(The inhomogeneous ratio is understood such that it scales with $\lambda,\lambda^2,\lambda^3$
in the corresponding terms for $\lambda\neq0$.)
\begin{proof}
\eqref{eq:finv1} and \eqref{eq:finv2} are equivalent by the corresponding cubic polynomaials.
\end{proof}
\end{theorem}
\end{commenty}

\snewpage
\subsection{The synthetic geometry of the conformal range}
\plabel{ssub:CRsynth}
~\\

We set
\[Q_{\mathrm{pCK}}^A( x_{\mathrm{pCK}}, z_{\mathrm{pCK}})=
\bem x_{\mathrm{pCK}}\\ z_{\mathrm{pCK}}\\1\eem^\top
\m Q_{\mathrm{pCK}}^{\mathbb R}(A)
\bem x_{\mathrm{pCK}}\\ z_{\mathrm{pCK}}\\1\eem;\]
\[Q_{\mathrm{pCK}}^\star( x_{\mathrm{pCK}}, z_{\mathrm{pCK}})=
\bem x_{\mathrm{pCK}}\\ z_{\mathrm{pCK}}\\1\eem^\top
\m Q_{\mathrm{pCK}}^{\mathbb R,0}
\bem x_{\mathrm{pCK}}\\ z_{\mathrm{pCK}}\\1\eem.\]

Let us consider the following situation:

(X) Assume that $A$ is a complex $2\times2$ matrix with no real eigenvalues.
Let $ \Lambda_1 $ and $ \Lambda_2 $ be its $h$-eigenpoints.
Let $\boldsymbol x $ be a hyperbolic point (with coordinates $\boldsymbol x_{\mathrm{pCK}}=( x_{\mathrm{pCK}} , z_{\mathrm{pCK}})$
in the $\mathrm{pCK}$ model).
%Let $\boldsymbol x_{\mathrm{pCK}}=( x_{\mathrm{pCK}} , z_{\mathrm{pCK}})$.
We will use the abbreviations
\[f_1=\mathrm d ( \boldsymbol x ,  \Lambda_1  ),\qquad
f_2=\mathrm d (  \boldsymbol x,  \Lambda_2  ),\qquad
m^+=\arcosh \frac{U_A+E_A}{\left| |D_A|-E_A\right|}.\]

\begin{theorem}
\plabel{thm:CRellipse}
Let us consider the setup of (X).
Then
\begin{multline}
\left( \cosh(f_1+f_2)-\cosh m^+ \right)\left( \cosh m^+ -\cosh(f_1-f_2) \right)=\\=\frac1{\left( |D_A|-E_A\right)^2}\cdot
\frac{Q_{\mathrm{pCK}}^A( x_{\mathrm{pCK}} , z_{\mathrm{pCK}})}{
-4Q_{\mathrm{pCK}}( x_{\mathrm{pCK}} , z_{\mathrm{pCK}})}
.
\plabel{eq:bifocal}
\end{multline}
Here, `pCK' can be changed to `BCK' throughout.
\begin{proof}
The LHS of \eqref{eq:bifocal} expands to
\[1- (\cosh f_1)^2-(\cosh f_2)^2-(\cosh m^+)^2+2(\cosh f_1)(\cosh f_2)(\cosh m^+)\]
(which is a symmetric expression in $f_1,f_2,m^+$).
This form for the LHS is particularly suitable to compute with in the BCK model, as
$f_1,f_2,m^+$ are all ``$\arcosh$''es, cf. \eqref{eq:dis}.

Now, the equation can be checked directly in terms of the data
 $\Rea\lambda_1$; $\Ima\lambda_1\neq0$; $\Rea\lambda_2$; $\Ima\lambda_2\neq0$, $\tau$.
\texttt{Alternatively}, one can notice that
 both sides of \eqref{eq:bifocal} are the naturally transforming quantities, thus it is sufficient to check \eqref{eq:bifocal}
 for the canonical representatives.
This means $A=L_{\alpha,t}^\pm$ with $0< \alpha\leq\pi/2$.
%\begin{commentx}
Then
\[(\Lambda_1)_{\mathrm{BCK}}=(\cos\alpha,0)
\qquad\text{and}\qquad
(\Lambda_2)_{\mathrm{BCK}}=(-\cos\alpha,0) \]
can be assumed.
Furthermore,
\[\frac{U_A+E_A}{\left| |D_A|-E_A\right|}=\frac{1+2t^2+(\cos\alpha)^2}{(\sin\alpha)^2},
\qquad\text{and}\qquad
\left( |D_A|-E_A\right)^2=(\sin\alpha)^4.\qedhere\]
%\end{commentx}
\end{proof}
\end{theorem}

\begin{theorem}[\cite{LLL}]
\plabel{cor:CRER1}

Let us consider the setup of (X).
Then the  points $\boldsymbol x $ of $\DW ^{\mathbb R}(A)$
are described as the solutions of the inequality
\[f_1+f_2\leq m^+,\]
with equality for the boundary.
Thus, the boundary is the synthetic $h$-ellipse with the $h$-eigenpoints as the foci and the length of the major axis as the distance sum.
\end{theorem}

\begin{proof}
This is trivial if $A$ is normal, and the conformal range is a segment.
If $A$ is not normal, then $Q_{\mathrm{pCK}}^A(x_{\mathrm{pCK}},z_{\mathrm{pCK}})\leq0$ describes the conformal range,
with equality on the boundary.
The equation is \eqref{eq:bifocal} but with the RHS replaced by $\leq0$.
Then $|f_1-f_2|$ is less or equal than the $h$-eigendistance.
The $h$-eigendistance is strictly less than the length of the major axis.
(This is immediate for the canonical representatives in the BCK model).
Consequently $\cosh m^+ -\cosh(f_1-f_2) >0$.
That makes the inequality $\cosh(f_1+f_2)-\cosh m^+\leq0$, which can be rewritten as in the statement.
\end{proof}
%\snewpage

Assume that $C$ is a $h$-horocycle with asymptotic point $P$.
Let $\mathrm d_C$ be the oriented distance from $C$ (negative inside the horodisk, positive outside).
Then $\mathrm d_C$ will be called a distance function from $P$.
Distance functions from $P$ differ from each other in an additive constant
(they belong to parallel horocycles).

Let us now consider the following setting:

(Y)
Suppose that $A$  is a complex $2\times 2$ matrix with a strictly complex eigenvalue $\lambda$ and a real eigenvalue $\lambda_0$.
Let the corresponding $h$-eigenpoints be $\Lambda$ and $\Lambda_0$, respectively.
Let $\boldsymbol x$ be the hyperbolic point represented by coordinates $( x_{\mathrm{pCK}} ,z_{\mathrm{pCK}})$
in the $\mathrm{pCK}$ model.
Let $\mathrm d_{\Lambda_0}$ be a distance function from $\Lambda_0$.
Furthermore, let $V$ be the vertex of the conformal range.

\begin{theorem}
\plabel{thm:CRpara}
Let us consider the setup (Y).
Then
\begin{multline}
\Bigl(\cosh\bigl(\mathrm d(\boldsymbol x,\Lambda)\bigr)-
\cosh\bigl(\mathrm d_{\Lambda_0}(\boldsymbol x)+\mathrm d_{\Lambda_0}(\Lambda)-2\mathrm d_{\Lambda_0}(V) \bigr)\Bigr)
\cdot{\exp(\mathrm d_{\Lambda_0}(\boldsymbol x)  - \mathrm d_{\Lambda_0}(\Lambda))}
= \\
=\frac1{4|D_A|\left( U_A+|D_A|\right)}\cdot
\frac{Q_{\mathrm{pCK}}^A(  x_{\mathrm{pCK}} ,z_{\mathrm{pCK}})}{-4Q_{\mathrm{pCK}}^\star( x_{\mathrm{pCK}} ,z_{\mathrm{pCK}})}
.
\plabel{eq:monofocal}
\end{multline}

\begin{proof}
$\mathrm d_{\Lambda_0}$ can be chosen as $\mathrm d_{[\lambda_0]}$ of \eqref{eq:salius}.
Also, `pCK' can be changed to `BCK' throughout.
Now, the equation can be checked directly in terms of the data
 $\Rea\lambda_1= \Rea\lambda$; $\Ima\lambda_1=\Ima\lambda\neq0$; $\Rea\lambda_2=\lambda_0$; $\Ima\lambda_2=0$, $\tau\geq0$.
\texttt{Alternatively},
as the RHS of \eqref{eq:monofocal} is one of the naturally transforming quantities, it is sufficient to check for the canonical representatives $S_\beta$.
%\begin{commentx}
Then
\[V_{\mathrm{BCK}}=(0,0),\qquad \Lambda_{\mathrm{BCK}}=\left(0,\frac{(\cos\beta)^2}{(\cos\beta)^2-2}\right),
\qquad\text{and}\qquad(\Lambda_0)_{\mathrm{BCK}}=(0,-1) .\]
Moreover,
\[\mathrm d_{\Lambda_0}^{\mathrm{BCK}}(x_{\mathrm{BCK}}, z_{\mathrm{BCK}})=
\log\left(\frac{1+z_{\mathrm{BCK}} }{\sqrt{1-(x_{\mathrm{BCK}})^2-(z_{\mathrm{BCK}})^2}}\right)\]
can be assumed. (Then $\mathrm d_{\Lambda_0}^{\mathrm{BCK}}(V_{\mathrm{BCK}})=0$.) Furthermore,
\[U_A+D_A=\frac12
\qquad\text{and}\qquad
 |D_A|=\frac14(\sin\beta)^2.\qedhere\]
%\end{commentx}
\end{proof}
\end{theorem}

\begin{theorem}
\plabel{cor:CRER2}
Let us consider the setup (Y).
Then the points $\boldsymbol x$ of the conformal range are given by the inequality
\begin{equation}
\mathrm d(\boldsymbol x,\Lambda) +\mathrm  d_0(\boldsymbol x)+\mathrm d_{\Lambda_0}(\Lambda)-2\mathrm d_{\Lambda_0}(V)\leq0,
\plabel{eq:ansa}
\end{equation}
with equality on the boundary.

(a) If $A$ is normal: Then the conformal range is the $h$-half line connecting $\Lambda$ and $\Lambda_0$.

(b) If $A$ is not normal: The boundary is the $h$-elliptic parabola, whose points
are of equal distance from $\Lambda$ and the horocycle with equation
$\mathrm d_{\Lambda_0}(\tilde{\boldsymbol x})+\mathrm d_{\Lambda_0}(\Lambda)-2\mathrm d_{\Lambda_0}(V)=0$.
In other terms, this is the synthetic $h$-elliptic parabola with vertex $V$ and focus $\Lambda$.
\end{theorem}

\begin{proof}
The normal case is rather easy, so will address only the non-normal case.
Then, according to \eqref{eq:monofocal}, the inequality for the conformal range is
\[ \mathrm d(\boldsymbol x,\Lambda)\leq
\bigl|\mathrm d_{\Lambda_0}(\boldsymbol x)+\mathrm d_{\Lambda_0}(\Lambda)-2\mathrm d_{\Lambda_0}(V) \bigr|\]
(as the quadric properly represents the conformal range).
It needs only a little thinking that this is equivalent to \eqref{eq:ansa}.

Consider the horodisk with equation
$\mathrm d_{\Lambda_0}(\tilde{\boldsymbol x})+\mathrm d_{\Lambda_0}(\Lambda)-2\mathrm d_{\Lambda_0}(V)\leq0$.
Then it is easy to see that $\tilde{\boldsymbol x}=\Lambda$ is in the interior.
Thus $\mathrm  d(\boldsymbol x,\Lambda))> | \mathrm d_{\Lambda_0}(\boldsymbol x)+\mathrm d_{\Lambda_0}(\Lambda)-2\mathrm d_{\Lambda_0}(V)| $
holds for any $\boldsymbol x$ in the exterior of horocycle.
Consequently, $\mathrm  d(\boldsymbol x,\Lambda)= |\mathrm  d_0(\boldsymbol x)+\mathrm d_{\Lambda_0}(\Lambda)-2\mathrm d_{\Lambda_0}(V)| $
may hold only with $  \mathrm d(\boldsymbol x,\Lambda)= -(\mathrm  d_0(\boldsymbol x)+\mathrm d_{\Lambda_0}(\Lambda)-2\mathrm d_{\Lambda_0}(V)) $.
Even $\mathrm  d(\boldsymbol x,\Lambda)\leq |\mathrm  d_0(\boldsymbol x)+\mathrm d_{\Lambda_0}(\Lambda)-2\mathrm d_{\Lambda_0}(V)| $
may hold only with $ \mathrm d(\boldsymbol x,\Lambda)\leq  -(\mathrm d_{\Lambda_0}(\boldsymbol x)+\mathrm d_{\Lambda_0}(\Lambda)-2\mathrm d_{\Lambda_0}(V)) $.
\end{proof}
%\snewpage

The remaining cases are similar to the ones of the Davis--Wielandt shell:

(XZ) is the case when $A$ has two distinct real eigenvalues:

\begin{theorem}
\plabel{thm:CRband}
Assume that $A$ has two distinct real eigenvalues $\lambda_1$, $\lambda_2$.
(In particular, $|D_A|=E_A>0$.)

(a)
Then the conformal range is a distance band with asymptotic points $\iota^{[2]}(\lambda_1)$ and $\iota^{[2]}(\lambda_2)$;
and whose radius is half of the minor axis,
\[s^-=\frac12\arcosh \frac{U_A}{|D_A|}. \]

(b) Let $\boldsymbol x$ be the hyperbolic point represented by coordinates
$( x_{\mathrm{pCK}} ,z_{\mathrm{pCK}})$ in the $\mathrm{pCK}$ model.
Let
\[d=\mathrm d\left( \boldsymbol x,\axis \DW^{\mathbb R}(A) \right). \]
Then
\[(\cosh 2d)-(\cosh 2s^-)=\frac1{|D_A|(U_A+|D_A|)}\cdot
\frac{Q_{\mathrm{pCK}}^A(x_{\mathrm{pCK}},z_{\mathrm{pCK}} )}{
-4Q_{\mathrm{pCK}}^\star(x_{\mathrm{pCK}},z_{\mathrm{pCK}} )} .\]

\begin{proof}
Due to the natural transformation properties of the participating expressions,
it is sufficient to check the conformal representatives $L_{0,t}^\pm=L_t$ in the BCK model, where it is a routine calculation.
\texttt{Alternatively,} this can be obtained from the case of Davis--Wielandt shell by restrictions.
\end{proof}
\end{theorem}

(YZ) is the case when $A$ has two equal real eigenvalues:
\snewpage
\begin{theorem}
\plabel{thm:CRhoro}
Assume that $A$ has two equal real eigenvalues, $\lambda_0$.
(Thus $|D_A|=E_A=0$.)

(a)
Then the conformal range is a possibly degenerate horodisk with asymptotic point $\iota^{[2]}(\lambda_0)$.

(b) Let $\mathrm d_{0} $ be the  oriented distance from its boundary
(positive outside the horodisk, negative inside, identically $+\infty$ in the normal case).
Let $\boldsymbol x$ be the hyperbolic point represented by coordinates
$( x_{\mathrm{pCK}} ,z_{\mathrm{pCK}})$ in the $\mathrm{pCK}$ model.
Then
\[\exp\left(2\mathrm d_{\Lambda_0}(\boldsymbol x) \right) -1=  \frac1{(U_A)^2}\cdot
\frac{Q_{\mathrm{pCK}}^A(x _{\mathrm{pCK}},z _{\mathrm{pCK}})}{-4Q_{\mathrm{pCK}}^\star(x _{\mathrm{pCK}},z _{\mathrm{pCK}})}.\]

\begin{proof}
It is sufficient to check the conformal representatives $S_0$ and $\m0_2$ in the BCK model, where it is a routine calculation.
\texttt{Alternatively,} this can be obtained from the case of Davis--Wielandt shell by restrictions.
\end{proof}
\end{theorem}
Cases (XZ) and (YZ) are simple to treat in the knowledge of the Davies--Wielandt shell,
 because the projection to canonical hyperbolic plane $i^{[2]}(\overline{H^2_*})$
 can be treated through restriction to that plane ($y_*=0$) here.
This, in fact, applies more generally to all real types.
That is not only to the real-hyperbolic case (XZ) and the real-parabolic case (YZ)
 but also to the real-elliptic case (i. e. case (X) with conjugate eigenvalues).

For these real types, analogues (restrictions) of Theorem \ref{thm:oridist} and Theorem \ref{thm:horodist}
 apply in such a straightforward manner so that their formulations are left to the reader.
In conjunction to the focal properties,  similarly to the case of the Davies--Wielandt shell, they offer
 quasi synthetical characterization of the conformal range.

%\snewpage
\begin{commentx}

The following observation applies to all real types (that is to real-hyperbolic, real-parabolic, or real-elliptic ones):
\begin{theorem}
\plabel{cor:oridist}
Suppose that $A$ is a $2\times2$ complex matrix with

(i) real eigenvalues $\lambda_1,\lambda_2$; or

(ii) conjugate strictly complex eigenvalues $\lambda,\bar\lambda$.

Then $\DW^{\mathbb R}_*(A)$ is the

(i) possibly degenerate $h$-distance band or $h$-horodisk, which is tangent to the asymptotic circle
at $\iota^{[2]}_*(\lambda_1), \iota^{[2]}_*(\lambda_2)$ (double tangency counts in the parabolic case); or

(ii)  $h$-disk with center $\iota^{[2]}_*(\lambda)= \iota^{[2]}_*(\bar\lambda)$.
\\

Furthermore,
the signed distance of $O^{\mathbb R}_*=\iota^{[2]}_*(O_*)$ from $\DW^{\mathbb R}_*(A)$ is given by formula
\begin{equation}
\overleftarrow{\mathrm{dis}}_*(O^{\mathbb R}_*,\DW^{\mathbb R}_*(A) )=\text{ $\mathrm{RHS}$ of \eqref{eq:oridist}},
\plabel{eq:CRoridist}
\end{equation}
and  the signed distance of $\DW^{\mathbb R}_*(A)$ from
$\mathcal P_*^{\infty,\mathbb R}=\mathcal P_*^\infty\cap i^{[2]}(\overline{H_*^2})$
is given by formula
\begin{equation}
\overleftarrow{\mathrm{dis}}_*( \DW^{\mathbb R}_*(A) ,\mathcal P_*^{\infty,\mathbb R})=\text{ $\mathrm{RHS}$ of \eqref{eq:horodist}}.
\plabel{eq:CRhorodist}
\end{equation}

In conjunction to the focal properties (i)/(ii),
 either one of \eqref{eq:CRoridist} or \eqref{eq:CRhorodist} characterizes
 $\DW^{\mathbb R}_*(A)$.
\begin{proof}

Here, the configuration is symmetric to the embedded $h$-plane $i^{[2]}(\overline{H_*^2})$, thus
$\overleftarrow{\mathrm{dis}}_*(O^{\mathbb R}_*,\DW^{\mathbb R}_*(A) )=
\overleftarrow{\mathrm{dis}}_*(O_*,\DW_*(A) )$
and
$\overleftarrow{\mathrm{dis}}_*( \DW^{\mathbb R}_*(A) ,\mathcal P_*^{\infty,\mathbb R})=
\overleftarrow{\mathrm{dis}}_*( \DW_*(A) ,\mathcal P_*^{\infty}) $
hold.
Then Theorem \ref{thm:oridist} and Theorem \ref{thm:horodist} can be invoked.

Sufficiency to characterize the conformal range can be seen easily.
\end{proof}
\end{theorem}

\end{commentx}
\snewpage
\subsection{The reconstruction problem}

\begin{disc}
\plabel{disc:CRrecfivedata}
Let us make some comments on the inverse problem of recovering the five data from the conformal range.
(Or, up to square roots, synonymously, the `triangularized five data' of \eqref{eq:canon2}.)
From $\DW_{\mathrm{pCK}}^{\mathbb R}(A)$,
the matrix $\m G_{\mathrm{pCK}}^{\mathbb R}(A)$ can be recovered by \eqref{eq:CRGrec} / \eqref{eq:CRGdegrec}.
Using Lemma \ref{lem:hypfocal},
if $\lambda_1,\lambda_2$ are the eigenvalues of $A$, then $\lambda_1,\bar\lambda_1,\lambda_2,\bar\lambda_2$
are the solutions of the equation
\begin{equation}
\bem-2\lambda\\ 1\\ \lambda^2\eem^\top
\m G_{\mathrm{pCK}}^{\mathbb R}(A)
\bem-2\lambda\\ 1\\ \lambda^2\eem=0.
\plabel{eq:quarticx}
\end{equation}
Having obtained the possible eigenvalue pairs, the $h$-eigenpoints are obvious.
In order to recover $A$, we have to choose concrete complex numbers for the eigenpoints.
Having them chosen, the ordering relation $|D_A| \lesseqqgtr E_A$ is clear.
Now, from Theorem \ref{thm:CReigen}(a), the set $\{U_A-|D_A|, U_A+|D_A|,U_A-|D_A|+2E_A\}$ can be recovered.
Using the ordering relation, $U_A,|D_A|,E_A$ can be recovered individually.
In particular, by \eqref{eq:canoffdiag}, the canonical triangular form of $A$ is also recovered.
Consequently, the five data of $A$ is also recovered.
As $\m G_{\mathrm{pCK}}^{\mathbb R}(A)$ is in direct relationship to the reduced five data;
this process also solves the problem of the reconstruction of the five data from the reduced five data.
(The main point is using Theorem \ref{thm:CReigen}(a);
deciding $|D_A| \lesseqqgtr E_A$ may, in fact, come after that.)
\eqed
\end{disc}
%\snewpage
\begin{remark}
\plabel{rem:CRcub}
One can  recover the `five data' $\Rea \tr A$, $\Ima \tr A$, $\Rea\det A$, $\Ima\det A$, $\tr A^*A$
from the `reduced five data' $\Rea \tr A$, $|\tr A|^2+2\Rea\det A$, $\Rea((\det A)\overline{\tr A})$, $|\det A|^2$, $\tr A^*A$
but \textit{up to finitely many choices}.
Geometrically this is equivalent to recovering a $h$-tube from its projection.
As this process can be carried out solving the quartic \eqref{eq:quarticx};
 this process can be realized by standard arithmetic and square and cubic roots.
One cannot do much simpler way, cubic roots (for us this will be an unprecise short term for
 some kind of field extensions of degree $3$) are  generally required:
\end{remark}
\begin{example}
\plabel{ex:CRcub}
Let us consider the quadric (disk)
 $8\,{x}^{2}+4\,xy+4\,{y}^{2}-8\,x-16\,y+9\leq0$
 understood with $x\equiv x_{\mathrm{pCK}}$ and $y\equiv y_{\mathrm{pCK}}$.
One can easily check (graphically) that this, indeed, lies  inside the pCK model.
Therefore the reconstruction problem
% for the matrix $A$ up to unitary conjugation,
%or, in other viewpoint, for the five data
is valid question in this case.
Using Lemma \ref{lem:CRspec}(a), one finds
$
\m G_{\mathrm{pCK}}^{\mathbb R}(A)=
\bem-1&\frac12&0\\\frac12&2&2\\0&2&1\eem;
$
which corresponds to the reduced five data $(V,W,X,Y,Z)=(0,0,1,2,4)$.
Then
\begin{equation}
\det\left(\lambda\Id_3+ 2
 \m G_{\mathrm{pCK}}^{\mathbb R}(A)\m Q_{\mathrm{pCK}}^{\mathbb R,0}\right)={\lambda}^{3}-6\,{\lambda}^{2}+10\,\lambda-\frac72.
\plabel{eq:roote}
\end{equation}
Using the Schönemann--Eisenstein argument with $p=2$, it is not hard to deduce that
 \eqref{eq:roote} is irreducible over the rational numbers.
Then elementary field theory tells us that all the roots of \eqref{eq:roote} require nontrivial cubic roots (over the rationals).
Therefore, $U_A-|D_A|$, $U_A+|D_A|$, $U_A-|D_A|+2E_A$ all require nontrivial cubic roots.
Theory also tells that even their ratios require nontrivial cubic roots.
Indeed, double roots are  exluded by irreducibility, and otherwise \eqref{eq:roote}
 and a rationally rescaled \eqref{eq:roote} would have a common factor.
In fact, the same comment about the ratios also applies to the depressed cubic with roots
$U_A-|D_A|-\frac{3U_A-|D_A|+2E_A}3=\frac{-2|D_A|-2E_A}3$,
$U_A+|D_A|-\frac{3U_A-|D_A|+2E_A}3=\frac{4|D_A|-2E_A}3$,
$U_A-|D_A|+2E_A-\frac{3U_A-|D_A|+2E_A}3=\frac{-2|D_A|+4E_A}3$.
In any case, as $U_A-|D_A|$ requires a cubic root,  so does the five data (as $U_A-|D_A|$ can be
 computed from the five data by square roots).
 In fact, as the comments about the ratios of the roots of the two polynomials in question
(\eqref{eq:roote} and the depressed version of \eqref{eq:roote})
 show,
 even the principal conformal invariants $U_A:|D_A|$ and $E_A:|D_A|$ require cubic roots.
\qedexer
\end{example}
\begin{remark}
\plabel{rem:CRcub2}
Note that the example above was in the non-normal quasielliptic / quasihyperbolic case (vii) of Theorem \ref{thm:CRdonc}.
It can be shown that in the other cases, square roots are sufficient for the reconstruction problem.
Indeed, when $A$ is not quasielliptic / quasihyperbolic, then \eqref{eq:quarticx} have double roots which, by standard methods,
 can be identified easily, and that helps to reduce problem to square roots.
But, even in the case when $A$ is quasielliptic / quasihyperbolic but normal, the problem is the
factorization of the quadratic form of $\m G_{\mathrm{pCK}}^{\mathbb R}(A)=\m G_{\mathrm{pCK}}^{\mathbb R,\mathrm{eig}}(A)$,
which is easy to do using square roots.
(It can be approached by setting some variables to $0$ first, factorizing, and then completing the factorization.)
\qedremark
\end{remark}
%\snewpage
\begin{remark}
\plabel{rem:CR}
However, if
 $\Rea \tr A$, $\Ima \tr A$, $\Rea\det A$, $\Ima\det A$, $\tr A^*A$
 are already given, then it is relatively easy to find all other possibilities
 to the same corresponding reduced data.
(This helps to understand the multiplicity structure in the reconstruction problem.)
The possibilities are given by the scheme

\[\left.\begin{matrix}
\Rea \tr A \\\\ \pm_1 \Ima \tr A \\\\[2mm] \Rea\det A \\\\[2mm] \pm_1\Ima\det A \\\\ \tr A^*A
\end{matrix}\right\}
\leftrightarrow
\left\{\begin{matrix}
\Rea \tr A \\\\
\pm_2\cdot 2\Rea\sqrt{D_A}&\equiv& \pm_3\cdot 2\Ima\sqrt{-D_A}
\\\\ \dfrac{|\tr A|^2}4-|D_A| \\\\
\pm_2\cdot \Rea((\sqrt{D_A})(\overline{\tr A}))&\equiv& \pm_3\cdot \Ima((\sqrt{-D_A})(\overline{\tr A}))
\\\\ \tr A^*A
\end{matrix}\right.
\]
where $\pm_1$ and $\pm_2$ can be chosen arbitrarily.
(I prefer to choose the complex square root $s$ such that
$\Rea s>0$ or ($\Rea s=0$ and $\Ima s\geq0$), but any consistent choice will do.)
%\snewpage

Note that
\begin{commentx}
\[|2\Rea\sqrt{D_A}|= \sqrt{2\left((\Rea D_A)+|D_A| \right)},\]
\[|\Rea((\sqrt{D_A})(\overline{\tr A}))|=\sqrt{|\det A|^2-\left(|D_A|-\frac{|\tr A|^2}4\right)^2},\]
\begin{multline*}
(2\Rea\sqrt{D_A})\cdot\Rea((\sqrt{D_A})(\overline{\tr A}))
\equiv(2\Ima\sqrt{-D_A})\cdot\Ima((\sqrt{-D_A})(\overline{\tr A}))=\\
=|D_A| (\Rea\tr A) +\Rea(D_A \overline{\tr A});
\end{multline*}
and the pairing takes
\[\{D_A,\overline{D_A}\}\quad\leftrightarrow\quad
\left\{
\left(  \frac{\Ima \tr A}2\pm_2\mathrm i\Ima\sqrt{D_A} \right)^2
\right\}=
\left\{
-\left(\Rea\sqrt{-D_A}\mp_3 \mathrm i \frac{\Ima \tr A}2\right)^2
\right\}, \]
\end{commentx}
\[ |D_A| \quad\leftrightarrow\quad E_A,\]
\[ U_A-|D_A| \quad\leftrightarrow\quad U_A-|D_A| . \]
The ambiguity is `in the spectral data', due to possible complex conjugation of the eigenvalues.
The degeneracy is

$\bullet$ $1$-fold in the real hyperbolic and real parabolic cases,
\[\text{1 real hyperbolic case}\circlearrowleft\]
\[\text{1 real parabolic case}\circlearrowleft\]

$\bullet$ $2$-fold in the semi-real case,
\[\text{2 semi-real cases }(\Ima\tr A\gtrless0)\circlearrowleft\]

$\bullet$ $3$-fold in the real elliptic / non-real parabolic cases,
\[\text{2 non-real parabolic cases }(\Ima\tr A\gtrless0)\leftrightarrow\text{1 real-elliptic case}\]

$\bullet$ $4$-fold in the quasielliptic/quasihyperbolic cases
\[\text{2 quasihyperbolic cases }(\Ima\tr A\gtrless0)\leftrightarrow\text{2 quasielliptic cases ($B_A^0$ distinguishes).} \]

Although the various ambiguities are essentially of spectral nature;
conjugation of both eigenvalues is realized by the operation $A\rightsquigarrow A^*$.
\qedremark
\end{remark}

%\snewpage
\begin{commenty}
\subsection{The principal conformal invariants in terms of the reduced five data}
~\\

Let
\[C_1(A)=\frac{U_2(A)}{(U_1(A))^2}
\qquad\text{and}\qquad
 C_2(A)=\frac{U_3(A)}{(U_1(A))^3}.\]
The quantities above are invariants of complex $2\times2$ matrices for real M\"obius transformations and unitary conjugation.
Except, we must naturally refrain from considering the otherwise simple case when $A$ is a real scalar matrix, i. e. $U_1(A)=0$.
\begin{commentx}
Remark:
$C_1(A)$ and $C_2(A)$  are related to \eqref{nt:uperfuu} %and \eqref{nt:rperfuu}
in terms of characteristic polynomials by
\begin{equation}
\det\left(\lambda\Id_3-\widetilde {\m G}^{\mathbb R}_{\mathrm{pCK}}(A) \m Q^{\mathbb R,0}_{\mathrm{pCK}}\right)
=\lambda^3+\frac14\lambda^2+\frac1{16}C_1(A)\lambda+\frac1{64}C_2(A).
\plabel{eq:granchar}
\end{equation}
\end{commentx}
\snewpage

\begin{disc}\plabel{rem:recoverUDE}
For the sake of simplicity we consider now the generic case, i.~e.~the non-normal quasielliptic / quasihyperbolic cases.
Then $0<U_A-|D_A|, |D_A|, E_A$, and $|D_A|\neq E_A$.
(Geometrically, this is the case of proper $h$-ellipses for the conformal range.)

It is trivial to see that $C_1(A)$ and $C_2(A)$  can be expressed by $\frac{E_A}{|D_A|}$ and $\frac{U_A}{|D_A|}$ generically.
Conversely, this is more complicated, as generically $\frac{E_A}{|D_A|}$ and $\frac{U_A}{|D_A|}$ allow two choices.
Nevertheless, in general, one has to solve a cubic to obtain them.
In general, it can be said that (generically) the sign of $E_A-|D_A|$ introduces a dichotomy.

Let $S$ be a symbolic variable.
Then, generically, the polynomial
\begin{align}
\left( -4\,{C_{{1}}}^{3}+{C_{{1}}}^{2}+18\,C_{{1}}C_{{2}}-27\,{C_{{2}
}}^{2}-4\,C_{{2}} \right)
&\, {S}^{3}+\plabel{eq:unireal}\\+ \notag
\left( 72\,{C_{{1}}}^{3}-99\,{C_{{1
}}}^{2}+162\,C_{{1}}C_{{2}}-243\,{C_{{2}}}^{2}+36\,C_{{1}}-36\,C_{{2}}
-4 \right)
&\, {S}^{2}+\\+ \notag
\left( -324\,{C_{{1}}}^{3}+243\,{C_{{1}}}^{2}+486
\,C_{{1}}C_{{2}}-729\,{C_{{2}}}^{2}-72\,C_{{1}}-108\,C_{{2}}+8
 \right)
 &\, S+\\+\notag
 \left(-81\,{C_{{1}}}^{2}+486\,C_{{1}}C_{{2}}-729\,{C_{{2}}}^{2}+36
\,C_{{1}}-108\,C_{{2}}-4\right)
\end{align}
factorizes as
\begin{multline*}
\frac{64\left(|D_A|\right)^2\left(E_A\right)^2\left(E_A-|D_A|\right)^2}{\left(3U_A-|D_A|+2E_A\right)^6}
\cdot
\\
\cdot
\left(S-\left(\frac{|D_A|+E_A}{ |D_A|-E_A}\right)^2\right)
\left(S-\left(\frac{2E_A-|D_A| }{ |D_A| }\right)^2\right)
\left(S-\left(\frac{2|D_A|-E_A }{ E_A }\right)^2\right);
\end{multline*}
thus $\left(\frac{|D_A|+E_A}{ |D_A|-E_A} \right)^2$ will be a root of \eqref{eq:unireal}.
We know that this root is distinguishable from the other roots on geometric grounds
(cf. focal distance), but
dichotomy appears in the sign of $\frac{|D_A|+E_A}{ |D_A|-E_A}$.
This also reflects in the indistinguishability of the other two roots.

Similarly, the polynomial
\begin{align}
\left( {C_{{1}}}^{2}-2\,C_{{1}}C_{{2}}+{C_{{2}}}^{2} \right)
&\, {S}^{3} +\plabel{eq:unicomp}\\+\notag
\left( 4\,{C_{{1}}}^{3}-3\,{C_{{1}}}^{2}+2\,C_{{1}}C_{{2}}+9\,{C_{{2}}}^{2}+4\,C_{{2}} \right)
&\, {S}^{2} +\\+\notag
 \left( -8\,{C_{{1}}}^{3}+3\,{C_{{1}
}}^{2}+18\,C_{{1}}C_{{2}}+27\,{C_{{2}}}^{2}-8\,C_{{2}} \right)
&\,S+\\+\notag
\left(4\,{C_{{1}}}^{3}-{C_{{1}}}^{2}-18\,C_{{1}}C_{{2}}+27\,{C_{{2}}}^{2}+4\,C_{{2}}\right)&
\end{align}
factorizes as
\begin{multline*}
\frac{64\left(U_A\right)^2\left(U_A+E_A\right)^2\left(U_A-|D_A|+E_A\right)^2}{\left(3U_A-|D_A|+2E_A\right)^6}
\cdot\\\cdot
\left(S-\left(\frac{|D_A|}{U_A}\right)^2\right)
\left(S-\left(\frac{E_A-|D_A|}{U_A+E_A}\right)^2\right)
\left(S-\left(\frac{E_A}{U_A-|D_A|+E_A}\right)^2\right);
\end{multline*}
thus $\left(\frac{|D_A|}{U_A}\right)^2$ will be a root of \eqref{eq:unicomp}.
Dichotomy appears in that it is indistinguishable from $\left(\frac{E_A}{U_A-|D_A|+E_A}\right)^2$.
The root $\left(\frac{E_A-|D_A|}{U_A+E_A}\right)^2$ is distinguishable on geometric grounds (cf. major axis),
but the sign of $\frac{E_A-|D_A|}{U_A+E_A}$ also represents the dichotomy.
\qedremark
\end{disc}

\end{commenty}
\snewpage
\appendix
\renewcommand{\thesubsection}{{\Alph{section}.\Alph{subsection}}}
\section{Some elementary algebra}
\plabel{sec:algproj}

\subsection{Identities for the adjugate matrix}
\plabel{ss:adjugate}
~\\

Recall that the `classical adjoint' or adjugate matrix of $\m M$ is denoted $\adj \m M$.
Its main property is that for $\det\m A\neq0$,
\[\m M^{-1}=\frac{\adj\m M}{\det\m M}.\]
Thus, whenever we are interested in matrices up to a scalar multiplier, then $\adj$ can be used instead of the inverse.
~
The contravariant multiplicative property
\begin{equation}
\adj (\m M_1\m M_2)=(\adj \m M_2)(\adj \m M_1)
\plabel{eq:adj1}
\end{equation}
is taken granted without special reference.
The involutivity property
\begin{equation}
\adj (\adj \m M)=\m M\qquad\text{for $2\times2$ matrices}
\plabel{eq:adj2}
\end{equation}
and the property
\begin{equation}
\adj (\adj \m M)=(\det \m M)\cdot \m M\qquad\text{for $3\times3$ matrices}
\plabel{eq:adj3}
\end{equation}
will also be understood% without much explanation
.
~
Assume that $\m v$, $\m w$ are classical $3$-dimensional vectors.
Then
\begin{equation}
\adj\frac12\left(\m v\m w^\top+\m w\m v^\top\right)=
-\frac14 (\m v\times\m w)(\m v\times\m w)^\top;
\plabel{eq:rif}
\end{equation}
and
\begin{equation}
\adj \left(\m v\m v^\top+\m w\m w^\top\right)=
  (\m v\times\m w)(\m v\times\m w)^\top;
\plabel{eq:raf}
\end{equation}
where `$\times$' is the ordinary vector product of Gibbs.

\subsection{The classical elementary analytic geometry of quadrics on the plane}
\plabel{ss:elem1}
~\\

Point $(x,y)$ of the ordinary real plane $\mathbb R^2$ can be also be imagined as a point of the
 real projective space with homogeneous coordinates $[x,y,1]$.
Lines on real projective space correspond are the points of the dual projective space.
In terms of homogeneous coordinates, the projective point $[\m x]=[x,y,z]$ fits
 the projective line (dual projective point) $[\m u]=[u,v,w]$ if and only if $\m u^\top\m x\equiv ux+vy+wz=0$.
(As much as possible we use the letters $x,y,z,G$ for covariant objects, $u,v,w,Q$ for contravariant objects.)
The intersection of the distinct lines $[\m u]$ and $[\m v]$ is the point $[\m u\times \m v]$;
 to the distinct points $[\m x]$ and $[\m y]$, the is the line $[\m x\times \m  y]$.

If $\m Q$ is a symmetric $3\times3$ real matrix, then
$Q(x,y)\equiv\left[\begin{smallmatrix} x\\ y \\ 1\end{smallmatrix}\right]^\top\m Q
\left[\begin{smallmatrix} x\\ y \\ 1\end{smallmatrix}\right]=0$ (the equality is meant symbolically) is its associated quadric.
(I.~e.~$Q$ is a symbolic equation, but for proper non-degenerate quadrics, and for pairs of distinct lines
it can also be identified by the set of its points.)
Here $\m Q$ is determined up to a nonzero scalar multiplier.
The quadric can also be understood in terms of the projected plane as
  $Q(x,y,z)\equiv\m x^\top\m Q\m x\equiv
  \left[\begin{smallmatrix} x\\ y \\ z\end{smallmatrix}\right]^\top\m Q
\left[\begin{smallmatrix} x\\ y \\ z\end{smallmatrix}\right]
   =0$,
  where $[\m x]\equiv[x,y,z]$ are the homogeneous coordinates of the points.

If $\det\m Q\neq0$, i. e. $\rank \,\m Q=3$, then   $Q$ is non-degenate.
It can be either proper ($\m Q$ is an indefinite matrix) or imaginary ($\m Q$ is a definite matrix).
In the first case, it is looks like either an ellipse, parabola, or hyperbola; in the second case it is
 imaginary ellipse, which has no points.

If $\rank \,\m Q=2$, then $Q$ is either a pair of lines or a projective point ellipse.
In the first case, $\m Q$ is of shape $\frac12\left(\m v\m w^\top+\m w\m v^\top\right)$, where $\m v\nparallel\m w$,
and $[\m v]$ and $[\m w]$ function as homogeneous coordinates of the participating lines.
In the second case, $\m Q$ is of shape $\pm\left(\m v\m v^\top+\m w\m w^\top\right) $, where $\m v\nparallel\m w$,
and $[\m v]$ and $[\m w]$ function as homogeneous coordinates for a pair conjugate axes.
If $\rank \,\m Q=1$, then $Q$ is double line in projective sense.
Then $\m Q$ is of shape $\pm \m v\m v^\top$, where $[\m v]$ functions as the supporting line of the quadric.

Let us assume that $\m Q$ is the matrix of a proper, non-generate quadric (so $\det\m Q\neq0$).
Then $[\m x]$ is an interior point of $Q$ if $(\sgn \m x^\top\m Q\m x)/(\sgn\det\m Q)=+1$;
$[\m x]$ is an interior point of $Q$ if $(\sgn \m x^\top\m Q\m x)/(\sgn\det\m Q)=-1$;
$[\m x]$ is a point of $Q$ if $\sgn \m x^\top\m Q\m x=0$.
Then $\m G= \m Q ^{-1}$ is the matrix of the dual quadric.
The dual quadric is also a proper non-degenerate quadric.
As, from the viewpoint of quadrics, the matrices are important up to a scalar multiplier,
 $\adj \m Q$ can also be used instead of $ \m Q ^{-1}$.
A projective line is avoiding $Q$ iff it is an interior point of $G$,
a projective line is tangent to $Q$ iff it is a point of $G$,
a projective line is intersecting to $Q$ (at two points) iff it is an exterior point of $G$.
Therefore, for the projective line $[\m u]$ the relationship $(\sgn \m u^\top\m G\m u)/(\sgn\det\m G)$
is informative.
%(We will prefer matrices with $\sgn\det\m Q=-1$ and $\sgn\det\m G=-1$.)
An identification between points and lines (dual points) is given by the correspondence the polar-pole correspondence
 $[\m u]=[\m Q\m x] \leftrightarrow [\m x]=[\m G\m u]$ with respect to $Q$ and $G$.
Interior points of $Q$ correspond to avoiding lines to $Q$,
points of $Q$  correspond to tangent lines to $Q$ (at the original points),
exterior points of $Q$  correspond to intersecting lines to $Q$.
The ``Euclidean center'' of $Q$ is the pole of the ideal line,
thus its representing vector is $\m G\left[\begin{smallmatrix}0\\0\\1\end{smallmatrix}\right]$.
Therefore, the last column of $\m G$ informs us about the projective coordinates of the Euclidean center of $Q$.

Let us assume now that $\rank \m Q=2$.
Then $\rank \adj\m Q=1$.
In fact, equations \eqref{eq:rif} and  \eqref{eq:raf} tell that $\adj\m Q$ informs about the ``supporting point'' of $Q$.
More precisely, $\adj\m Q$ is the ``double point'' corresponding to the supporting point through the polar-pole correspondence
 $[\m u]=[\m Q\m x] \leftrightarrow [\m x]=[\m G\m u]$ with respect to $Q$ and $G$.
If $\rank \m Q=1$, then $\adj\m Q=0$.
Using \eqref{eq:rif} and  \eqref{eq:raf}, and the latter observation, we can also see that
 in the case $\det\m Q=0$; one has
 $\tr\adj\m Q>0$ if $Q$ is a point ellipse,
 $\tr\adj\m Q=0$ if $Q$ is a double line,
 $\tr\adj\m Q<0$ if $Q$ is a pair of distinct lines.

\subsection{The classical elementary analytic geometry of quadrics in the space}
\plabel{ss:elem2}
~\\

This is more complicated than the case of plane, but we deal with only the basics.
Here Euclidean points $(x,y,z)$ correspond to projective point homogeneous coordinates $[x,y,z,1]$.
Points of the dual projective space are the projective planes in the original projective space.
In terms of homogeneous coordinates, the projective point $[\m x]=[x,y,z,t]$ fits to the projective plane
 $[\m u]=[u,v,s,w]$ if and only  $\m u^\top \m x=0$.
A nonzero $4\times 4$ symmetric matrix  $\m Q$ encodes a quadric $Q$; $\m G=\m Q^{-1}$ encodes the dual quadric $G$,
 but $\adj \m Q$ can also be used for the latter purpose.
The pole--polar  correspondence $[\m x]=[\m G\m u]\leftrightarrow[\m u]=[\m Q\m x] $ with respect to $Q$ and $G$
 makes a correspondence between points of projective space and the planes of the projective space.
Notably, it takes the points of $Q$ to the tangent spaces of $Q$ (at the specified points).
\snewpage

\subsection{The reduction of quadratic forms}
\plabel{ss:algproj}
~\\

(a) Assume that
$\m M=\bem \m A&\m b\\\m b^{\top}&c\eem$
is a symmetric block matrix of type $(n-1 | 1)$.
We can think that it represents the quadratic form
$\bem\m x\\y \eem \mapsto\bem\m x\\y \eem^\top\bem \m A&\m b\\\m b^{\top}&c\eem\bem\m x\\y \eem$.
Our objective is to project (not restrict!) it to its first $(n-1)\times(n-1)$ block / first $n-1$ variables.
We may consider various methods.
At this point we compare them only algebraically
(although we may have strong reasons for their equivalence on geometrical grounds).

(i) We can invert the matrix $\m M$, take restriction to the $(n-1)\times(n-1)$ block, and invert back.
(This follows the ideology that projections and restrictions are dual.)
The feasibility of this process assumes not only that $\m M$ is invertible but that $c\neq0$.
(Indeed, if $\bem \m U^{-1}&\m v\\\m v^{\top}&d\eem$
is the inverse of $\m M$, then multiplication yields $\m  U^{-1}\m b+\m v^\top c=0$;
thus $c=0$ would imply $\m b=0$ in contradiction to the invertibility of $\m M$.)
In that $c\neq0$ case,
\begin{equation}
\bem \m A&\m b\\\m b^{\top}&c\eem=\bem \Id&\frac1c \m b\\&1\eem
\bem \m A-\frac{\m b\m b^\top}c  &\\&c\eem
\bem \Id&\\\frac1c \m b^{\top}&1\eem
\plabel{eq:redpre}
\end{equation}
implies
\begin{align*}
\bem \m U^{-1}&\m v\\\m v^{\top}&d\eem
\equiv
\bem \m A&\m b\\\m b^{\top}&c\eem^{-1}
&=
\bem \Id&\\-\frac1c \m b^{\top}&1\eem
\bem (\m A-\frac{\m b\m b^\top}c)^{-1}  &\\&c^{-1}\eem
\bem \Id&-\frac1c \m b\\&1\eem
\\
&=
\bem
(\m A-\frac{\m b\m b^\top}c)^{-1}
&
-\frac1c (\m A-\frac{\m b\m b^\top}c)^{-1}\m b
\\
-\frac1c \m b^{\top}(\m A-\frac{\m b\m b^\top}c)^{-1}
&
\frac1c+\frac1{c^2} \m b^{\top}(\m A-\frac{\m b\m b^\top}c)^{-1}\m b
\eem.
\end{align*}
Thus, inversion, restriction and inversion back results $\m U=$
\begin{equation}
\m A-\frac{\m b\m b^\top}c.
\plabel{eq:red1}
\end{equation}

(ii) We may pass to the quadratic form and then reduce out $y$ according to
\[\frac\partial{\partial y}\left(
\bem\m x\\y \eem^\top\bem \m A&\m b\\\m b^{\top}&c\eem\bem\m x\\y \eem
\right)=0.\]
Then the resulted quadratic form is may be converted into a matrix.

Here the quadratic form
\[\m x^\top\m A\m x+2\m x^\top\m by+cy^2\]
yields the reducing equation
\[2\m x^\top\m b+ 2cy=0,\]
which, in turn, leads to the quadratic form
\[\m x^\top\m A\m x+2\m x^\top\m b
\left(-\frac{\m x^\top\m b}c\right)+c\left(-\frac{\m x^\top\m b}c\right)^2, \]
whose matrix is
\[\m A-\frac{\m b\m b^\top}c=
\bem \Id\\ -\frac1c\m b^\top\eem ^\top
\bem \m A&\m b\\\m b^{\top}&c\eem
\bem \Id\\ -\frac1c\m b^\top\eem. \]
The result is the same as \eqref{eq:red1}, but
this method assumes only that $c\neq0$.

\begin{commentx}
(Another interpretation of (ii) is the effect of the quadratic form on the basis
after one step of the Gram--Schmidt orthogonalization process with respect to the last basis element.
Or, uncorrelating to a variable in statistics.)
\end{commentx}

In any case, \eqref{eq:redpre} implies
\[\det\left(\m A-\frac{\m b\m b^\top}c\right)\cdot c =\det \bem \m A&\m b\\\m b^{\top}&c\eem.\]

(iii) We may simply take the discriminant of the relevant quadratic form in variable $y$, and use this
in order to obtain the matrix.
Then the discriminant is
\[(2\m x^\top\m b)^2-4 \m x^\top\m A\m x \cdot c,\]
whose matrix is
\[-4 c \m A + 4\m b\m b^\top.\]
Thus this method gives $-4c$ times the result of the previous one but no assumption on the matrix is taken.
Ultimately, in the case $c\neq0$, only a nonzero scalar factor enters.
In that case, through dividing the discriminant by $-4c$ we can recover the previous results.
However, as this more general than the previous ones, we just rather divide by $-4$.
Therefore, we set
\begin{equation}
\m M\bo\|_{\{1,\ldots,n-1\} }^{\mathrm{Schur}} =c \m A -\m b\m b^\top
\plabel{eq:red2}
\end{equation}

(b) A the notation above indicates, the shape of \eqref{eq:red1} and \eqref{eq:red2} can be interpreted much more generally.
Assume, for the sake of simplicity, that
\[\m M=\bem \m A&\m B_1\\\m B_2&\m C\eem\]
is $(p +q)\times(p +q) $ block matrix, such that $\m C$ is $q\times q$ matrix.
Then the analogue of \eqref{eq:red1} is the Schur complement
\[\m A-\m B_1\m C^{-1}\m B_2.\]
The analogue of \eqref{eq:red2} is the Schur reduction
\[\m M\bo\|_{\{1,\ldots,p\} }^{\mathrm{Schur}} =(\det\m C)\m A-\m B_1(\adj \m C)\m B_2.\]

In fact, one can recognize that reduction can happen not only to the first couple of variables
 but we can choose an arbitrary set of variables to reduce to.

\snewpage

\end{document}